\documentclass[11pt]{article}

\usepackage[USenglish]{babel}
\usepackage[T1]{fontenc}
\usepackage[ansinew]{inputenc}
\usepackage{indentfirst}
\usepackage{graphicx}
\usepackage{amsmath}
\usepackage{amsthm}
\usepackage{amsfonts}
\usepackage{amssymb}
\usepackage{dsfont}
\usepackage{enumitem} 
\usepackage{hyperref}
\usepackage{color}
\usepackage{amsmath}
\usepackage{fullpage}
\usepackage{tikz}
\usetikzlibrary{calc} 
\DeclareMathOperator{\Tr}{Tr}

\DeclareMathOperator{\Id}{Id}

\DeclareMathOperator{\RR}{\mathbb{R}}
\DeclareMathOperator{\CC}{\mathbb{C}}
\DeclareMathOperator{\NN}{\mathbb{N}}
\DeclareMathOperator{\Mat}{\mathrm{M}}
\begin{document}
	\title{Non-Hermitian expander obtained with Haar distributed unitaries}
	\author{Sarah Timhadjelt}

\maketitle

\begin{abstract}
We consider a random quantum channel obtained by taking a selection of $d$ independent and Haar distributed $N$-dimensional unitaries. We follow the argument of Hastings to bound the spectral gap in terms of eigenvalues and adapt it to give an exact estimate of the spectral gap in terms of singular values \cite{hastings2007random,harrow2007quantum}. This shows that we have constructed a random quantum expander in terms of both singular values and eigenvalues. The lower bound is an analog of the Alon-Boppana bound for $d$-regular graphs. The upper bound is obtained using Schwinger-Dyson equations.
\end{abstract}

\newcommand*\normeH[1]{\|#1\|_{\mathcal{H} }}
\newcommand*\normeBH[1]{\|#1\|_{\mathcal{H}\rightarrow \mathcal{H} }}
\newcommand{\abs}[1]{\left| #1 \right|}

\theoremstyle{plain}
\newtheorem{thm}{Theorem}[section]
\newtheorem{lem}[thm]{Lemma}
\newtheorem{prop}[thm]{Proposition}
\newtheorem{defi}[thm]{Definition}
\newtheorem{cor}[thm]{Corollary}
\newtheorem{pr}[thm]{propriety}
\theoremstyle{remark}
\newtheorem{rem}{Remark}[section]
\newtheorem{ex}{Example}[section]

\tableofcontents 
\section*{Acknowledgments}
I would like to thank the anonymous referee for his careful review and valuable comments. In particular, his review allowed Lemma \ref{Alonbop} to apply to general quantum channels and not only to those obtained with unitary matrices as Kraus operators. I thank Charles Bordenave for all the insightful discussions on this project. Finally I of course thank Guillaume Aubrun for giving me this project and for his supervision.
\section{Introduction}
\subsection{Notations and properties of finite dimensional operators}
For all integer $N \in \NN$ we denote by $\mathrm{M}_N(\CC)$ the algebra of $N$-dimensional complex matrices. For $M \in \mathrm{M}_N(\CC)$ we denote by $M^{\ast}$ its adjoint, by $M^t$ its transpose and finally by $\overline{M}$ the matrix with the conjugate entries of $M$. For operators on Hilbert space, we denote by $\|\cdot\|$ the operator norm, which is subordinate to the scalar product norm. We will also denote by $\mathrm{U}_N$ the subgroup of unitary matrices, i.e. the set of $U \in \mathrm{M}_N(\CC)$ which verify that $U U^\ast= U^{\ast}U= \Id$. We denote by $\mathrm{O}_N$ and $\mathrm{S}_N$ the subgroups of $\mathrm{U}_N$ of orthogonal and permutation matrices, respectively. For $\epsilon \in \{+,-\}$ and $U \in \mathrm{U}_N$ we set $U^{\epsilon}= U$ if $\epsilon =+$ and $U^\epsilon = U^\ast$ if $\epsilon = -$.
Regarding the spectrum, for $M$ an operator on an $N$-dimensional Hilbert space we consider the eigenvalues with multiplicities. We denote the $k$-th eigenvalue by $\lambda_k(M)$ and order them as follows:
\begin{equation*}
	|\lambda_1(M)|\ge |\lambda_{2}(M)| \ge \cdots \ge |\lambda_N(M)|.
\end{equation*}
We also denote the $k$-th singular value of the operator $M$ by $s_k(M)$, i.e., for all $1 \le k \le N$ we have $s_{k}(M)=\lambda_k(M^{\ast}M)^{1/2}$. In particular, we have:
\begin{equation*}
	\|M\| = s_1(M)\ge s_{2}(M) \ge \cdots \ge s_{N}(M).
\end{equation*}
Finally for a sequence $(A_j)_{j \in [n]} \in \Mat_N(\CC)$ we take as notation and convention:
\begin{equation}\label{orderedpdt}
    \prod_{j=1}^n A_j = A_1 A_2 \cdots A_{j} A_{j+1} \cdots A_{n-1} A_{n}.
\end{equation}
For a finite index set $[n]$ on a sequence $(A_{j})_{j\in [n]}$, we consider as convention $A_{n+k} = A_k$ for all integer $k$.\\

In quantum mechanics, we consider the state of a $N$-dimensional system to be a positive semidefinite matrix with trace $1$, usually denoted $\rho$. A transformation of such a state, e.g. due to measurements, is described by a \textit{quantum channel} $\mathcal{E} : \mathrm{M}_{N}(\CC)\rightarrow \mathrm{M}_{N}(\CC)$, i.e. linear, completely positive, self-adjoint preserving and trace preserving map \cite{aubrun2017alice,lancien2023note}. Recall that an operator $\mathcal{T} : \mathrm{M}_{N}(\CC)\rightarrow \mathrm{M}_{N}(\CC)$ is completely positive if the operator $\mathcal{T} \otimes \Id : \mathrm{M}_{N^2}(\CC)\rightarrow \mathrm{M}_{N^2}(\CC)$ is positive. For every self-adjoint preserving and completely positive operator $\mathcal{T} : \mathrm{M}_{N}(\CC)\rightarrow \mathrm{M}_{N}(\CC)$ one can find $(K(s))_{s\in [d]} \in \mathrm{M}_N(\CC)^d$ such that for all $M \in \mathrm{M}_N(\CC)$ we have (Choi's Theorem \cite[Theorem 2.21]{aubrun2017alice}):
\begin{equation}\label{QCgeneral}
	\mathcal{T}(M) = \sum_{s=1}^{d} K(s) M K(s)^\ast.
\end{equation}
The previous (not unique) way of writing the operator $\mathcal{T}$ is called \textit{Kraus representation}. The smallest integer $d\in \NN$ for which a previous writing is possible for a given self-adjoint preserving and completely positive operator $\mathcal{T}$ is called the \textit{Kraus rank} or \textit{degree} of the operator $\mathcal{T}$. Finally, $(K(s))_{s\in [d]}$ are called the \textit{Kraus operators} associated with $\mathcal{T}$. Now one can also consider that a quantum channel is a sum of tensor products of matrices. To be more specific, one can first consider the \textit{trace} defined on $\mathcal{B}(\mathrm{M}_N(\CC))$ by:
\begin{align*}
	\tau_N : \mathcal{B}(\mathrm{M}_N(\CC))& \rightarrow \CC\\
	\mathcal{T} &\mapsto \sum_{i,j} \Tr[\mathcal{T}(E_{ij})E_{ji}] = \sum_{i,j} (\mathcal{T}(E_{ij}),E_{ij})
\end{align*}
where $(E_{ij})_{i,j \in [N]}$ is the canonical basis of $\mathrm{M}_N(\CC)$ considered with the usual scalar product $( A,B ) = \Tr(AB^\ast)$ for $A,B \in \mathrm{M}_N (\CC)$. Then we can set the scalar product on $\mathcal{B}(\mathrm{M}_N (\CC))$:
\begin{align*}
	\langle \cdot,\cdot \rangle : \mathcal{B}(\mathrm{M}_N(\CC))^2& \rightarrow \CC\\
	(\mathcal{T}, \mathcal{V}) &\mapsto \tau_N(\mathcal{T}\mathcal{V}^\ast).
\end{align*}
To find the correspondence between $\mathcal{B}(\mathrm{M}_{N}(\CC))$ and $\mathrm{M}_{N}(\CC) \otimes \mathrm{M}_N(\CC)$ one can consider the maps:
\begin{align*}
	\mathcal{T}_{A,B} : \mathrm{M}_N(\CC)& \rightarrow \mathrm{M}_N(\CC)\\
	M &\mapsto AMB
\end{align*}
for $A,B \in \mathrm{M}_N(\CC)$. We have that the application:
\begin{align*}
	M_{\cdot}:(\mathcal{B}(\mathrm{M}_N(\CC)), \langle \cdot, \cdot\rangle) & \rightarrow (\mathrm{M}_N(\CC)\otimes\mathrm{M}_N(\CC),( \cdot, \cdot))\\
	\mathcal{T}_{A,B} &\mapsto A\otimes B^t =M_{\mathcal{T}_{A,B}}
\end{align*}
induces an isometry for the norm induced by the respective scalar product. Now for any quantum channel $\mathcal{T}$ one can consider $(K(s))_{s\in[d]}$ its Kraus operators for a given Kraus decomposition (as in \eqref{QCgeneral}) and we have:
\begin{align*}
	M_{\mathcal{T}} = \sum_{s=1}^{d} K(s) \otimes \overline{K(s)}.
\end{align*}
The spectral distributions of $\mathcal{E}$ and $M_{\mathcal{E}}$ are then equal.

\subsection{Quantum information problematic and previous results}
  For a family of $N$-dimensional matrices $(K(s))_{s\in [d]}$, the condition
  \begin{equation}\label{tracepres}
  \sum_{s=1}^{d}  K(s)^\ast K(s) = \Id_{N}
  \end{equation}
  implies that the operator defined by Equation \eqref{QCgeneral} is indeed a quantum channel. The eigenvalue of largest modulus $\lambda_1(\cdot)$ for a quantum channel is always $1$, and there exists an associated eigenvector which is positive and semidefinite, called \textit{fixed state} and denoted $\hat{\rho}$ (see \cite[Introduction]{lancien2023note}, \cite[Chapter 6]{wolf2012quantum}). As with Markov chains, the second largest eigenvalue (or second largest singular value) of a quantum channel can be seen as a quantification of the distance between the considered quantum channel $\mathcal{E}$  and the ideal quantum channel that would send any state to the fixed point of the quantum channel $\mathcal{E}_{\hat{\rho}} : \rho \mapsto \Tr[\rho]\hat{\rho}$.
Given these considerations, it is of interest in quantum mechanics to construct a quantum channel that has a small second eigenvalue as the dimension grows. Therefore, we introduce here a definition of quantum expanders that depends on the spectral distribution of the channel.
\begin{defi}[Quantum expander (eigenvalues)]\label{defiexpeig}
	Let $\epsilon>0$ and $\mathcal{E}$ be a quantum channel. We say that $\mathcal{E}$ is a $\epsilon$-\textup{quantum expander in eigenvalues} if:
	\begin{equation*}
		|\lambda_2(\mathcal{E})| \le 1-\epsilon.
	\end{equation*}
\end{defi} 
This definition corresponds to the classical graph expander definition when defined by the control of the second largest eigenvalue \cite{lancien2023limiting}. One could also consider a control on the second largest singular value. We denote by $\Pi_N \in \mathcal{B}(\mathrm{M}_N(\CC))$ the orthogonal projector of rank one on $\mathbb{\CC}\Id$.
\begin{defi}[Quantum expander (singular values)]\label{defiexpsing}
	Let $\epsilon>0$ and $\mathcal{E}$ be a quantum channel. We say that $\mathcal{E}$ is a $\epsilon$-\textup{quantum expander in singular values} if:
	\begin{equation*}
		\|\mathcal{E} - \Pi_N\| \le 1-\epsilon.
	\end{equation*}
\end{defi} 
%The complete positivity condition on this family of operator is more subtle (see \cite[Section 3.2]{aubrun2017alice}).\\
In fact, the previous definition corresponds to the second singular value in the case where the eigenspace of $\lambda_1(\mathcal{E})$ is of dimension $1$ and the eigenvector is given by $\hat{\rho}:= \frac{1}{N}\Id$. This is the case if we consider $(\frac{1}{\sqrt{d}}U(s))_{s\in [d]}$ as the Kraus decomposition of $\mathcal{E}$ with $(U(s))_{s\in [d]}$ iid Haar distributed. The previous spectral definitions can be seen as a consequence of convergences of operators in terms of free probability. Indeed, asymptotic freeness allows the construction of an explicit operator $\mathcal{E}_{\mathrm{free}} \in \mathcal{B}(\mathrm{H})$ such that the spectrum (\textit{resp.} the empirical spectral measure) of $\mathcal{E}$ \textit{converges}  asymptotically to the spectrum (\textit{resp.} spectral measure) of $\mathcal{E}_{\mathrm{free}}$ in a sense yet to be defined. Considering unitary Haar distributed and independent matrices, the asymptotic operators are the left representations of the free group. For any integer $d$ we then denote $\mathbf{u}= (u(s))_{s\in [d]}$ the generators of $\mathbb{F}_d$ the free group with $d$ generators. We keep the same notation $(u(s))_{s\in [d]}$ for the left-representations, i.e. the operators defined by:

\begin{align*}
	u(s) : \ell^2(\mathbb{F}_d) &\rightarrow \ell^2(\mathbb{F}_d)\\
	 \delta_g &\mapsto \delta_{u(s)g}
\end{align*}
where $\delta_g := (\mathds{1}(h=g))_{h\in \mathbb{F}_d} \in \ell^2(\mathbb{F}_d)$. Seen as a family of operators, $\mathbf{u}$ is a $d$-\textit{Haar free family}. Then the fact that a family of possibly random unit matrices $ \mathbf{M}=(M(s))_{s\in [d]}$ are \textit{asymptotically free} means that for any non-commutative polynomial $P \in \CC \langle X(s), X(s)^\ast |s\in [d]\rangle$ we have:
\begin{equation}\label{defiasymptotfree}
	\lim_{N \rightarrow \infty} \Tr_N [P(\mathbf{M})] = \langle P(\mathbf{u})\delta_e, \delta_e\rangle
\end{equation}
 where $e \in \mathbb{F}_d$ is the neutral element. If such a convergence occurs it already gives spectral information over some operators obtained with the family $\mathbf{M}$. If one considers $P$ a symmetric polynomial, $A = P(\mathbf{M})$ and $a = P(\mathbf{u})$ then the previous convergence implies:
 \begin{equation}\label{spectralasymp}
 \mu_{A} := \frac{1}{N} \sum_{k=1}^N \delta_{\lambda_k(A)} \underset{N\rightarrow \infty}{\rightharpoonup} \mu_a
 \end{equation}
 where $\mu_a$ is the spectral measure given by the spectral theorem applied to $a$ and the convergence is weak. To apply the previous convergences to expansion of possibly random graphs or quantum channels one can consider the adjacency matrix:
 
 \begin{equation*}
 	A = \frac{1}{2d}\sum_{s=1}^d M(s) + M(s)^{\ast}
 \end{equation*}
where $\mathbf{M}=(M(s))_{s\in [d]}$ is a family of permutations or tensor products of permutation. The asymptotic operator $a = \frac{1}{2d}\sum_s u(s) + u(s)^\ast$ is then the adjacency operator of $\mathbb{F}_d$ Cayley graph. The corresponding spectral measure is given by Kesten-McKay law $\mu_a(\mathrm{d}x)= f(\cdot)\mathrm{d}x$ where the density is given by:
\begin{align*}
	f : \mathbb{R} &\rightarrow \RR\\
	x & \mapsto \frac{d}{2\pi} \frac{\sqrt{4(d-1) -x^2}}{d^2 -x^2} \mathds{1}(|x|\le 2 \sqrt{d-1}).
\end{align*}
Having asymptotic freeness for the family $\mathbf{M}$, i.e. the convergence given in \eqref{defiasymptotfree} and therefore the convergence \eqref{spectralasymp}, also means:

\begin{equation*}
	|\sigma(A) \cap [-2\frac{\sqrt{d-1}}{d},2\frac{\sqrt{d-1}}{d}]^\mathrm{c}| = o (N).
\end{equation*}
In particular we have at most $o(N)$ \textit{outliers}, i.e. eigenvalues greater in modulus than the Alon-Boppana bound $2\sqrt{d-1}/d$. Also we have as expected that the Alon-Boppana bound is sharp since the weak convergence gives also that for all $\epsilon >0$ we have $|\sigma(A) \cap [2 \sqrt{d-1}/d - \epsilon , 2 \sqrt{d-1}/d]| >\delta$ for some $\delta >0$ and $N$ large enough. The asymptotic freeness \eqref{defiasymptotfree} had been proved for $\mathbf{M}$ a family of Haar distributed unitaries of $\mathrm{U}_N$ by Voiculescu \cite{voiculescu1991limit} in 1991 and for random permutations by Nica \cite{nica1993asymptotically} in 1993. The fact that the convergence holds when one replaces the converging family $\mathbf{M} = (M(s))_{s\in [d]}$ by the family $\mathbf{M}^{\otimes} = (M(s)\otimes \overline{M(s)})_{s\in [d]}$ is due to the absorption phenomenon of the Haar unit family $\mathbf{u} =(u(s))_{s\in [d]}$ proved by Collins and Gaudreau-Lamarre \cite{collins2017freeness}. In particular the previous asymptotic freeness results combined with the absorption phenomenon allows to say that considering the quantum channel:
\begin{align}
	\mathcal{E}_{\mathrm{h}} := \frac{1}{2d} \sum_{s=1}^d U(s)\otimes\overline{U(s)} + U(s)^\ast\otimes U(s)^t
\end{align}
where $(U(s))_{s \in [d]} \in \mathrm{G}_N^d$ are iid Haar distributed unitaries in $\mathrm{G}_N \in \{\mathrm{U}_N, \mathrm{S}_N\}$, we have for all $\epsilon>0$:

\begin{equation*}
		|\sigma(\mathcal{E}_{\mathrm{h}}) \cap [-2\frac{\sqrt{d-1}}{d},2\frac{\sqrt{d-1}}{d}]^\mathrm{c}| = o (N) ~~~\text{and}~~~|\sigma(\mathcal{E}) \cap [2 \sqrt{d-1}/d - \epsilon , 2 \sqrt{d-1}/d]| >\delta
\end{equation*}
where the second statement above holds for some $\delta >0$ and for $N$ large enough. It is also of interest and possible to have similar results over the spectrum of $\mathcal{E}$ when one considers now that the Kraus operators $\mathbf{K} = (K(s))_{s\in [d]}$ are not unitaries but still asymptotically free. The problem for the computation of the asymptotic measure $\mu_a$ in this case is the absence of absorption phenomenon. Still, the convergence \eqref{spectralasymp} is showed for Kraus operators $(K(s))_{s\in [d]}$ iid, Hermitian, converging in distribution and asymptotically free by Lancien, Oliveira Santos and Youssef \cite{lancien2023limiting}. They also explicit the computation for the limit distribution $\mu_a$ in \cite{lancien2024centrallimittheoremtensor} when one considers the operator $\mathcal{E} = \frac{1}{\sqrt{N}} \sum_{s=1}^N K(s) \otimes \overline{K(s)}$. \\

In any case, the convergence given in \eqref{defiasymptotfree} is not sufficient to conclude about the expansion of operators. In fact, it is precisely the absence of \textit{outliers} that is needed to satisfy the expanders definitions in \ref{defiexpeig} and \ref{defiexpsing}. One way to obtain these expansions is to show \textit{strong} asymptotic freeness of the family $\mathbf{K}^{\otimes}=(K(s)\otimes \overline{K(s)})_{s\in [d]}$. Again sticking to the unitary case, we will say that a family of unitary matrices $\mathbf{M}= (M(s))_{s\in [d]}$ is \textit{strongly asymptotically free} if for any non-commutative polynomial $P \in \CC \langle X(s),X(s)^\ast| s\in [d] \rangle$ we have \eqref{defiasymptotfree} and furthermore:
\begin{equation}\label{defiasympfreestrong}
	\lim_{N \rightarrow \infty} \| P(\mathbf{M})\| = \| P(\mathbf{u})\|,
\end{equation}
where the norms above are the operator norms associated to the inner-product of the Hermitian spaces $\CC^N$ and $\ell^2(\mathbb{F}_d)$.
 Both the Hermitian and non-Hermitian cases are treated in the case where $\mathbf{M}=(M(s))_{s\in[d]}$ are random permutations \cite{CBeigen}, or Haar distributed unitaries \cite{bordenave2022strong}. Indeed,  Bordenave and Collins showed the following theorem using Weingarten calculus.
\begin{thm}[Bordenave, Collins 2018,2022]
	Let $d\ge 1$ be an integer and $(U(s))_{s\in [d]} \in \mathrm{G}_N$ iid Haar distributed where $\mathrm{G}_N \in \{\mathrm{U}_N, \mathrm{O}_N, \mathrm{S}_N\}$. The family $\mathbf{U}^{\otimes}=(U(s)\otimes \overline{U}(s)|_{\mathds{1}^{\perp}})_{s\in[d]}$ is strongly asymptotically free.
\end{thm}
\begin{figure}[h]
	\centering
	\includegraphics[scale=0.7]{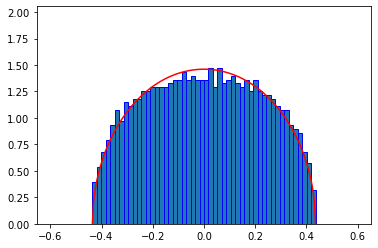}
	\caption{Plot of the eigenvalues of $\mathcal{E}$ in the hermitian case for $N=40$ and $d=10$. In red the semi-circle law of radius $\lambda_{\mathrm{herm},20}:=2\frac{\sqrt{19}}{20}$.}
\end{figure}
Recall that the case where $\mathrm{G}_N = \mathrm{S}_N$ \cite{CBeigen} was established before the case where $\mathrm{G}_N \in \{\mathrm{U}_N, \mathrm{O}_N\}$ \cite{bordenave2022strong}. In both cases, it implies expansion for the random quantum channel $\mathcal{E}$ in both the Hermitian and non-Hermitian cases. To be more precise, in the Hermitian case we have with probability one:
$$\lim_{N \rightarrow \infty} \lambda_{2} (\mathcal{E}_{\mathrm{h}}) = \frac{2\sqrt{d-1}}{d}= \rho\left(\frac{1}{2d} \sum_{s=1}^{d} u(s) + u(s)^\ast\right)$$ 
which gives the expansion in terms of the eigenvalues. In the non-Hermitian case we have:
$$\lim_{N \rightarrow \infty} s_{2} (\mathcal{E}) = \frac{2\sqrt{d-1}}{d} = \|\frac{1}{d} \sum_{s=1}^{d} u(s)\|$$ 
and thus the expansion with respect to the singular values. When it comes to the speed of convergence, the absorption phenomenon leads to think that convergences and convergence rates proved for iid Haar distributed unitary matrices $(U(s))_{s\in[d]}$ remain true for the representation $(U(s)\otimes U(s))_{s\in[d]}$. Indeed Bordenave and Collins proved a convergence speed in $N^{-c/d}$ for the norm of polynomials in Haar distributed unitaries $(U(s))_{s\in [d]}$ \cite[Corollary 1.2]{bordenave2023norm}. One can also refer to \cite[Theorem 1.1]{parraud2023asymptotic} for a speed in $N^{-k}$ for convergence in expectation for $\mathcal{C}^{4k +7}$ functions applied to polynomials in Haar distributed unitaries proved by Parraud. The convergence speeds and convergence itself, computed by Bordenave and Collins \cite{bordenave2022strong,bordenave2023norm} use Weingarten calculus to compute moments in Haar distributed unitaries. An advantage of using Weingarten calculus in their proof is that it applies to other subgroups such as the orthogonal one. In this paper we obtain a speed of convergence in $c\ln(N)/N^{1/12}$ (see Theorem \ref{powers1}) for some constant $c>0$ valid for all Kraus degree $d$ adapting Hastings method and using Schwinger-Dyson equation instead of Weingarten. The use of Schwinger-Dyson equation is specific to the unitary group. The equations are obtained by using the invariance by translation of the Haar distribution on $\mathrm{U}_N$ (see the proof of Schwinger-Dyson equation, Proposition \ref{SDPROP}). This does not hold for other subgroups of interest, such as the orthogonal or the symmetric.

\subsection{Model, Alon-Boppana bound and main theorem} 

Throughout this article, for $d=d_N\in \mathbb{N}^\ast$ sequence of Kraus degrees, we consider the optimal asymptotic second largest eigenvalue in the non-Hermitian case:

\begin{equation*}
	\rho_d := \frac{1}{\sqrt{d}},
\end{equation*}
and the optimal asymptotic second largest singular value:
\begin{equation*}
	\sigma_d := \frac{2 \sqrt{d-1}}{d}.
\end{equation*}
 The first Alon-Boppana bound we give is general, that is to say it applies to any $\mathcal{T}$ quantum channel, i.e. an operator given by $\eqref{QCgeneral}$ and verifying \eqref{tracepres}.

\begin{lem}\label{Alonbop}
  For any $d=d_N$ sequence of degrees, let $(K(s))_{s\in [d]} \in \Mat_N(\CC)$ be a family of matrices verifying \eqref{tracepres}. Let $\mathcal{T}$ be the operator defined by \eqref{QCgeneral}. We set $m=m_N := \lfloor  \frac{\ln(N)}{4\ln(d)}\rfloor$. For all $N \ge 2$ and for all $2 \le d \le N^{1/4}$ we have:
  \begin{equation*}
  s_{2}(\mathcal{T}^m)^{1/m}\ge \rho_d \exp(-\frac{1}{2N}).
  \end{equation*}
  \end{lem}
The previous Lemma gives a lower bound for the approximation of $\lambda_2(\mathcal{T})$ given by $s_{2}(\mathcal{T}^{m})^{1/m}$ with $N$ fixed. This is not sufficient to directly infer a lower bound for $\lambda_2(\mathcal{T})$ as we will discuss later (see Corollary \ref{hastings2}). For the second Alon-Boppana bound and for all the other results of this paper we now restrict to $(U(s))_{s\in[ d]} \in \mathrm{U}_N^d$ a sequence of $d$-tuple unitaries and we consider the following possibly random operator:
\begin{align}
	\mathcal{E} ~:~ \mathrm{M}_{N}(\mathbb{C}) &\rightarrow \mathrm{M}_{N}(\mathbb{C}) \notag \\
	M 						&\mapsto    \frac{1}{d} \sum_{s=1}^{d} U(s)^{\ast}M U(s). \label{defE}
\end{align}
This operator is completely positive, unit-preserving, self-adjointness preserving and trace-preserving, and is therefore a quantum channel. The following lemma gives the lower bound for $s_{2}(\mathcal{E})$ directly.
\begin{lem}\label{Alonbop2}   
	For any $d=d_N$ sequence of degrees, let $(U(s))_{s\in [d]} \in \mathrm{U}_N^d$ be any sequence of $d$ unitary matrices and let $\mathcal{E}$ be the operator defined by \eqref{defE}. We set $p=p_N = \lfloor \frac{\ln(N)}{2\ln(1/\sigma_d)}\rfloor$. There exists $c>0$ such that for all $N \ge 2$ and for all $2 \le d=d_N \le N^{1/4}$ we have:
	\begin{equation*}
		s_{2}(\mathcal{E})\ge\sigma_d\left(\frac{c}{\ln(N)^{3/2}}\right)^{1/2p}.
	\end{equation*}	
\end{lem}
The previous lemma implies that for all sequence $(d_N)_{N\ge 2}$ such that:
\begin{equation}\label{condiond}
	\ln(d_N) << \frac{\ln(N)}{\ln(\ln(N))}
\end{equation}
we have that almost surely:
$$s_{2}(\mathcal{E}) \ge \sigma_d\left(1 - \epsilon_N\right),$$
for an explicit $\epsilon_N\underset{N\rightarrow \infty}{\rightarrow} 0$. The previous lemma gives an upper bound on the best singular value expansion one can get.

The following theorems show that for $(U(s))_{s\in [d]}$ iid Haar distributed unitary matrices of dimension $N$ the operator $\mathcal{E}$ given by \eqref{defE} is indeed a quantum expander with high probability and its expansion is optimal with respect to the previous lemmas (for singular value and eigenvalue approximation). 

%\begin{thm}\label{hastings}
%	There exists $c> 0$ a numerical constant such that for all $d=d_N$ sequence of degrees and $(U(s))_{s\in [d]}$ %iid Haar distributed unitary matrices of dimension $N$ we have
%	$$\mathbb{E}\left( |\lambda_2 (\mathcal{E})| \right) \le \rho_d \left( 1 + c \frac{\ln(N)}{N^{1/12}} \right) %,$$
%	where $\mathcal{E}$ is the quantum channel defined by Equation \eqref{defE}.
%\end{thm}  

\begin{thm}\label{hastingsdetails}
	There exist numerical constants $\kappa,c>0$ such that for all $d=d_N$ sequence of degrees, considering $(U(s))_{s\in [d]}$ iid Haar distributed unitary matrices of dimension $N$ and setting $m = m_N = \lfloor \kappa N^{1/12}\rfloor$ we have:
	\begin{equation*}
	\mathbb{E}\left(|\lambda_2(\mathcal{E})|\right) \le \mathbb{E}\left( s_2 (\mathcal{E}^m)^2 \right)^{1/2m} \le \rho_d \left( 1 + c \frac{\ln(N)}{N^{1/12}} \right),
	\end{equation*}
	where $\mathcal{E}$ is the quantum channel defined by Equation \eqref{defE}.
\end{thm}
Using the explicit constant in the previous theorem we have the following corollary.
\begin{cor}\label{hastingscor}
    In the setting of Theorem \ref{hastingsdetails}, for all $\epsilon >0$, we have that for $N$ large enough:
	$$\mathbb{P} \left(|\lambda_2(\mathcal{E})| \ge \rho_d(1 +\epsilon) \right)  \le e^{-\frac{1}{8\sqrt{2}}\ln(1+\epsilon)N^{1/12}},$$
	where $\mathcal{E}$ is the quantum channel defined by Equation \eqref{defE}.
\end{cor}
The previous theorem can be shown directly adapting Hastings proof to the non-Hermitian case \cite{hastings2007random}. 
%\begin{cor}
%	For all $d=d_N$ sequence of degrees we consider $(U(s))_{s\in [d]}$ iid Haar distributed unitary matrices of dimension $N$ and for $N \ge 2$ integer we set $m = m_N = \lfloor \frac{1}{4 \sqrt{2}} N^{1/12}\rfloor$. For all $\epsilon>0$, for all sequence of degrees $d=d_N$ we have:
%	\begin{equation*}
%		\mathbb{P}\left(s_{2}(\mathcal{E}^m)^{1/m}\ge \rho_d(1 + \epsilon) \right) \le \frac{1}{8}N^{13/6} e^{-\frac{1}{4 \sqrt{2}}\ln(1+ \epsilon)N^{1/12}},
%	\end{equation*}
%	where $\mathcal{E}$ is the quantum channel defined by Equation \eqref{defE}. In particular we have that for $N$ large enough:
%	\begin{equation*}
%		\mathbb{P}\left(s_{2}(\mathcal{E}^m)^{1/m}\ge \rho_d(1+  \epsilon) \right) \le e^{-\frac{1}{8\sqrt{2}}\ln(1+ \epsilon)N^{1/12}}.
%	\end{equation*}
%\end{cor}
However Theorem \ref{hastingsdetails} can also be seen as a consequence of the following proposition.

\begin{prop}\label{powers}
	For all $d=d_N$ sequence of degrees, $ 1\le m=m_N \le \frac{1}{2}\lfloor\frac{1}{4\sqrt{2}}N^{1/12}\rfloor$  and $\epsilon >0$, considering $(U(s))_{s\in [d]}$ iid Haar distributed unitary matrices of dimension $N$ we have:
	\begin{equation*}
		\mathbb{P}\left(s_{2}(\mathcal{E}^m)^{1/m}\ge \rho_d(1 + \epsilon) \right) \le \frac{\sqrt{2}(m+1)}{m}N^{25/12}e^{[2\frac{\ln(m+1)}{m} -\ln(1+\epsilon)]\frac{1}{4\sqrt{2}}N^{1/12}},
	\end{equation*}
	for $N$ large enough and where $\mathcal{E}$ is the quantum channel defined by Equation \eqref{defE}. In particular if furthermore $m_{N}\underset{N\rightarrow \infty}{\rightarrow} \infty$ then for $N$ large enough we have:
	\begin{equation*}
		\mathbb{P}\left(s_{2}(\mathcal{E}^m)^{1/m}\ge \rho_d(1 + \epsilon) \right) \le e^{-\frac{1}{8\sqrt{2}}\ln(1+ \epsilon)N^{1/12}}.
	\end{equation*}
\end{prop}

Detailing the previous proposition computation in the special case where $m=1$ gives the expansion in terms of singular values.
\begin{thm}\label{powers1}
	There exists a numerical constant $c>0$ such that for all $d=d_N$ sequence of degrees, considering $(U(s))_{s\in [d]}$ iid Haar distributed unitary matrices of dimension $N$ we have:
	\begin{equation*}
	\mathbb{E}\left( s_2(\mathcal{E}) \right) \le \sigma_d \left( 1 + c \frac{\ln(N)}{N^{1/12}} \right),
	\end{equation*}
	where $\mathcal{E}$ is the quantum channel defined by Equation \eqref{defE}.
\end{thm}

%\begin{thm}\label{powers1}
%	For all $d=d_N$ sequence of degrees we consider $(U(s))_{s\in [d]}$ iid Haar distributed unitary matrices of dimension $N$. For all $\epsilon > 0$ and for $N$ large enough we have for  all sequence of degree $d_N$:
%	\begin{equation*}
%		\mathbb{P}\left(s_{2}(\mathcal{E})\ge \sigma_d(1 + \epsilon) \right) \le N^{25/12}e^{ -\ln(1+\epsilon)\frac{1}{4\sqrt{2}}N^{1/12}}.
%	\end{equation*}
%	where $\mathcal{E}$ is the quantum channel defined by Equation \eqref{defE}.
%\end{thm}
Once again, using the explicit constant of the previous theorem we have the following corollary. 

\begin{cor}\label{powers1cor}
In the setting of Theorem \ref{powers1}, for all $0 < \epsilon <1$ we have that for $N$ large enough:

$$\mathbb{P} \left( s_{2}(\mathcal{E}) \ge \sigma_d \left(1 + \epsilon\right) \right) \le e^{- \frac{1}{8 \sqrt{2}} \ln(1+ \epsilon) N^{1/12}}$$

where $\mathcal{E}$ is the quantum channel defined by \eqref{defE}.
    
\end{cor}
The previous theorems together with Lemma \ref{Alonbop} and Lemma \ref{Alonbop2} give the following corollaries, stating convergence either for an approximation of the second largest eigenvalue (i.e., the second largest singular value of the $m$-th power of $\mathcal{E}$) or for the second largest singular value of $\mathcal{E}$.

\begin{cor}\label{hastings2}
	For all $\epsilon > 0$ and $2\le d=d_N \le N^{1/4}$ sequence of degrees, considering $(U(s))_{s\in [d]} $ iid Haar distributed unitary matrices of dimension $N$, we have for $N$ large enough:
	\begin{equation*}
		\mathbb{P}\left(\left|s_{2}(\mathcal{E}^m)^{1/m}-\rho_d\right| \ge \rho_d \epsilon \right) \le e^{-\frac{1}{8\sqrt{2}}\ln(1+ \epsilon)N^{1/12}}
	\end{equation*} 
	where $ m = \lfloor 2 \frac{\ln(N)}{\ln(d)}\rfloor$ and $\mathcal{E}$ is the operator given by \eqref{defE}.
\end{cor}

The upper bound given by Proposition \ref{powers} gives us the expansion in terms of eigenvalues of $\mathcal{E}$, since we have the following inequality:
$$|\lambda_{2}(\mathcal{E})| \le s_{2}(\mathcal{E}^m)^{1/m}.$$
Also from Gelfand Theorem we have:
$$|\lambda_{2} (\mathcal{E})| = \liminf_{m\rightarrow \infty} s_{2}(\mathcal{E}^m)^{1/m}.$$
However, these facts are not sufficient to give direct information about the behavior of $\Delta(\mathcal{E})=|\lambda_{2}(\mathcal{E})-\rho_d|$ as the dimension grows. In fact, an upper bound for $\Delta(\mathcal{E})$ would be a consequence of an intermediate result which states that:
$$\lambda_2(\mathcal{E}) \ge \left(1 + o(1)\right) s_{2}(\mathcal{E}^{m})^{1/m}$$
for any sequence $m=m_N$ suitable for Proposition \ref{powers}. 

\begin{figure}[h]
	\centering
	\includegraphics[scale=0.7]{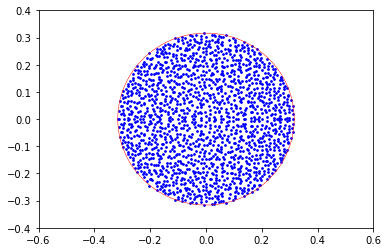}
	\caption{Plot of the eigenvalues of $\mathcal{E}$ for $N=40$ and $d=10$. In red the circle of radius $\lambda_{10}:=\frac{1}{\sqrt{10}}$.}
\end{figure}
We have the following corollary for the convergence of the second largest singular value, where one needs to use the explicit constant given in Theorem \ref{powers1} to conclude. 
\begin{cor}\label{hastings3}
	For all $\epsilon > 0$ and $2\le d=d_N \le N^{1/4}$ sequence of degrees, considering $(U(s))_{s\in [d]} $ iid Haar distributed unitary matrices of dimension $N$, we have for $N$ large enough:
	\begin{equation*}
		\mathbb{P}\left(\left|s_{2}(\mathcal{E})-\sigma_d\right| \ge \sigma_d \epsilon \right) \le e^{-\frac{1}{8\sqrt{2}}\ln(1+ \epsilon)N^{1/12}}
	\end{equation*} 
	where $\mathcal{E}$ is the operator given by \eqref{defE}.
\end{cor}
\begin{figure}[h]
	\centering
	\includegraphics[scale=0.7]{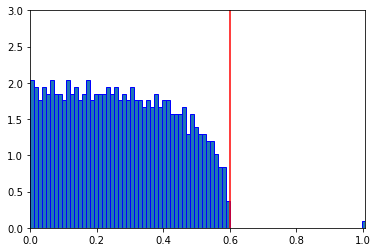}
	\caption{Plot of the eigenvalues of $|\mathcal{E}|:= \sqrt{\mathcal{E}^\ast\mathcal{E}}$ for $N=30$ and $d=10$. In red $x=\bar{\lambda}_{10}=2\frac{\sqrt{d-1}}{d} = \frac{3}{5}$.} 
\end{figure}
Except for Lemma \ref{Alonbop} and Lemma \ref{Alonbop2}, all the proofs will rely on an algorithm based on Schwinger-Dyson equations \eqref{SD} used by Hastings \cite{hastings2007random}, which gives good approximations to the expectation of the product of traces of words in Haar distributed unitary matrices (see Section \ref{Schwinger-Dyson Equations}). The lack of dependence on $d_N$ for Theorem \ref{hastingsdetails}, Proposition \ref{powers} and Theorem \ref{powers1} is due to the fact that iterations of the Schwinger-Dyson equation give upper bounds on the expectations of traces in monomials of Haar distributed unitaries that depend only on the degree of the original monomial (see Section \ref{Schwinger-Dyson Equations}, Lemma \ref{lemcv}). 

\subsection{Proof of Lemma \ref{Alonbop} and Lemma \ref{Alonbop2}}\label{computationoftraces}

  We begin with a proof of the Alon-Boppana bounds given by Lemma \ref{Alonbop} and Lemma \ref{Alonbop2}. We first consider an arbitrary family of $N$-dimensional matrices $(K(s))_{s \in [d]}$ verifying \eqref{tracepres} and the associated quantum channel $\mathcal{T}$ defined by \eqref{QCgeneral}. To obtain the lower bound $s_{2}(\mathcal{T}^m)^{1/m}$ we set, for any dimension $N$ and $m \ge 1$ integer:
  \begin{equation*}
  \mathrm{E}_1 := \tau_{N}({\mathcal{T}^\ast}^m\mathcal{T}^{m})=\sum_{i,j} \langle \mathcal{T}^m(E_{ij}), \mathcal{T}^m(E_{ij}) \rangle = \sum_{a=1}^{N^2} s_a(\mathcal{T}^m)^{2}.
  \end{equation*}
 Using the convention given by \eqref{orderedpdt}, for any integer $m\ge 1$ and $\mathcal{S} = (s_1,...,s_{2m})\in[d]^{2m}$ we set:
  \begin{align} \label{defK(S)}
  \mathcal{K}(\mathcal{S}) := K(s_{2m})\cdots K(s_{m+1})K(s_m)^\ast \cdots K(s_1)^\ast= \prod_{j=1}^{m} K(s_{2m-j+1})\prod_{j=m+1}^{2m}K(s_{2m-j+1})^\ast
  \end{align}
  such that developing the powers of $\mathrm{E}_1$ we obtain:
  \begin{align*}
  \mathrm{E}_1 & =\sum_{(s_1,...,s_{2m}) \in [d]^{2m}}\Tr[\prod_{j=1}^{m} K(s_{2m-j+1})\prod_{j=m+1}^{2m}K(s_{2m-j+1})^\ast]\Tr[\prod_{j=1}^{m} K(s_{j})\prod_{j=m+1}^{2m}K(s_{j})^\ast]\\
  &=\sum_{(s_1,...,s_{2m}) \in [d]^{2m}} |\Tr[\mathcal{K}(\mathcal{S})]|^2.
  \end{align*}
  Lemma \ref{Alonbop} is deterministic and, as for the Alon-Boppana bound in the case of $d$-regular graphs, it shows that $1-\rho_d$ is the best expansion for the approximation of $\abs{\lambda_{2}(\mathcal{T})}$ by $s_{2}(\mathcal{T}^{m})^{1/m}$ one can get.

\begin{proof}[Proof of Lemma \ref{Alonbop}]
  For all $N\ge 1$ and $d \ge1$ we consider a fixed collection $(K(s))_{s\in[d]}$ of $d$ Kraus operators of dimension $N$, i.e. matrices verifying \eqref{tracepres}. For all $\mathcal{S}=(s_{t})_{t \in [2m]} \in [d]^{2m}$ we recall the definition of $\mathcal{K}(\mathcal{S})$ by \eqref{defK(S)} and we additionally define, in the case where $(s_1,...,s_{m})= (s_{2m},...,s_{m+1})$,
  $$\mathcal{K}^{\prime}(\mathcal{S}) = K(s_{2m})\cdots K(s_{m+1}).$$
  We have
  \begin{align*}
  1 + N^{2}s_{2}(\mathcal{T}^m)^2 &\ge1 + (N^{2}-1)\lambda_{2}({\mathcal{T}^\ast}^m \mathcal{T}^m ) \ge \mathrm{E}_1\\
  & \ge \sum_{\substack{(s_1,...,s_m)\\~~~~= (s_{2m},...,s_{m+1})}} |\Tr[\mathcal{K}(\mathcal{S})]|^2  + \sum_{\substack{(s_1,...,s_m)\\~~\neq (s_{2m},...,s_{m+1})}} \underbrace{|\Tr[\mathcal{K}(\mathcal{S})]|^2}_{\ge 0} \\
  &\ge d^{-m} \left( \Tr[\sum_{\substack{(s_1,...,s_m)\\~~~~= (s_{2m},...,s_{m+1})}} \mathcal{K}^\prime(\mathcal{S})\Id_{N}\mathcal{K}^\prime(\mathcal{S})^{\ast}] \right)^{2} = d^{-m} \underbrace{\Tr[\mathcal{T}^m(\Id_N)]^2}_{=N^2},
  \end{align*}
  where between the second line and the third we used Cauchy-Schwarz inequality on the first sum. Therefore we have:
  \begin{align*}
  s_{2}(\mathcal{T}^m)^{1/m} &\ge \rho_d \exp\left[{\frac{1}{2m}\ln\left(1 - \frac{1}{N^2\rho_d^{2m}}\right)}\right]\\
  & \ge \rho_d \exp\left[- \frac{1}{2N}\right]
  \end{align*}
  where for the second line we took $m=\lfloor \frac{\ln(N)}{4\ln(d)}\rfloor$ and supposed $1 \le d \le N^{1/4}$.
  \end{proof}
We now give all the settings we need to prove Lemma \ref{Alonbop2}. We consider $(U(s))_{s\in [d]}$ a family of unitary matrices of dimension $N$ and $\mathcal{E}$ the operator defined by \eqref{defE}. For all dimension $N$ and $p\ge 1$ integer we set:
\begin{align*}
	\mathrm{E}_3 := \tau_{N}[({\mathcal{E}^\ast}\mathcal{E})^p]=\sum_{i,j} \langle ({\mathcal{E}^\ast}\mathcal{E})^p(E_{ij}), E_{ij}\rangle = \sum_{a=1}^{N^2} s_a(\mathcal{E})^{2p}.
\end{align*}
When there is no confusion, for all $\mathcal{S} = (s^1,...,s^{2p})\in[d]^{2p}$ we set:
\begin{align*} 
	\mathcal{U}(\mathcal{S}) &:= U(s^{2p})U(s^{2p-1})^\ast\cdots U(s^{2})U(s^1)^\ast \\
\end{align*} 
such that developing the powers of $\mathrm{E}_3$ we obtain:
\begin{align*}
	\mathrm{E}_3 & =\left(\frac{1}{d}\right)^{2p} \sum_{(s_1,...,s_{2p}) \in [d]^{2p}}\Tr[\prod_{t=1}^{p} U(s^{2p-2t+2})^\ast U(s^{2p-2t+1})]\Tr[\prod_{t=1}^{p} U(s^{2t-1})^\ast U(s^{2t})]\\
	&= \sum_{(s^1,...,s^{2p}) \in [d]^{2p}} |\Tr[\mathcal{U}(\mathcal{S})]|^2.
\end{align*}
Lemma \ref{Alonbop2} is also deterministic and shows that $1 - \sigma_d$ is the best expansion one can get for singular values (see Definition \ref{defiexpsing}). 

\begin{proof}[Proof of Lemma \ref{Alonbop2}] 
	For all $N\ge 1$ and $d \ge 1$ we consider a collection $(U(s))_{s\in[d]}$ of $d$ unitary matrices of dimension $N$.
	For any $p \ge 1$ we denote by $\mathrm{N}^\prime(p,0,d)$ the number of $(s^t)_{t\in[2p]} \in [d]^{2p}$ such that 
	$$\mathcal{U}(\mathcal{S})=U(s^{2p})U(s^{2p-1})^\ast \cdots U(s^{2})U(s^1)^\ast= \Id.$$
	Using the notations above we have:
	\begin{align*}
		1 + N^{2}s_{2}(\mathcal{E})^{2p} &\ge1 + (N^{2}-1)\lambda_{2}(({\mathcal{E}}^\ast \mathcal{E})^p  ) \ge \mathrm{E}_3\\
		& \ge \frac{1}{d^{2p}} \left( \mathrm{N}^\prime(p,0,d)N^2+ \sum_{(s^{t})_{t \in [2p]} \in [d]^{2p}} \mathds{1}(\mathcal{U}(\mathcal{S})\neq \Id)\underbrace{|\Tr[\mathcal{U}(\mathcal{S})]|^2}_{\ge 0}\right)\ge \frac{\mathrm{N}^{\prime}(p,0,d)}{d^{2p}}N^2.
	\end{align*}
	We consider the following random walk:
	\begin{align*}
		(s_{t})_{t\in\mathbb{N}} &\sim \mathrm{Unif}([d])\\
		X_0 &:= \Id \\
		\forall t\in \NN,~\forall s \in [d]~~~~\mathbb{P}(X_{2t+1}=X_{2t}U(s)) = \frac{1}{d},~&~\mathbb{P}(X_{2t+2}=X_{2t+1}U(s)^\ast) = \frac{1}{d},
	\end{align*}
	where $\mathrm{Unif}(V)$ denotes the uniform law on the set $V$. Also we consider the random walk defined as above but with $(U^{\prime}(s))_{s \in [d]}$ iid Haar distributed unitary matrices, i.e.:
	\begin{align*}
		X_0^\prime &:= \Id \\
		\forall t\in \NN,~\forall s \in [d]~~~~\mathbb{P}(X_{2t+1}^\prime=X_{2t}^\prime U(s)^\prime) = \frac{1}{d},~&~\mathbb{P}(X_{2t+2}^\prime=X_{2t+1}^\prime {U(s)^\prime}^\ast) = \frac{1}{d}.
	\end{align*}
	We denote by $\mathrm{N}(p,0,d)$ the number of $(s^t)_{t\in [2p]} \in [d]^{2p}$ such that:
	$$U^\prime(s^{2p})U^{\prime}(s^{2p-1})^\ast \cdots U^\prime(s^2) U^\prime(s^1)^\ast = \Id$$
	and we have that:
	$$\mathbb{P}(X_{2p}= \Id) = \frac{\mathrm{N}^\prime(p,0,d)}{d^{2p}} \ge \mathbb{P}(X_{2p}^\prime= \Id) = \frac{\mathrm{N}(p,0,d)}{d^{2p}}.$$
	The previous inequality is due to the following inclusion:
	$$\{ (s^{t})_{t\in [2p]} \in [d]^{2p}:~ \prod_{j=0}^{p-1} U^\prime(s^{2p -2j})U^\prime(s^{2p-2j-1})^\ast = \Id\} \subset \{(s^{t})_{t\in [2p]} \in [d]^{2p}:~ \prod_{j=0}^{p-1} U(s^{2p -2j})U(s^{2p-2j-1})^\ast = \Id\}.$$
	The random walk $(X_{t}^\prime)_{t\ge 1}$ is a random walk on a $d$-regular graph. Therefore there exists $C> 0$ independent from $d$ such that for any integer $p$ we have:
	$$\mathrm{N}(p,0,d) \ge C \frac{(2 \sqrt{d-1})^{2p}}{p^{3/2}}$$
	(see \cite[Theorem 5.3]{hoory2006expander} and \cite{nilli1991second,friedman1993some}). We now consider:
	$$p := \lfloor \frac{\ln(N)}{2\ln(1/\sigma_d)}\rfloor.$$
	Then there exists $c^\prime> 0$ such that for any $N\ge 2$ and $2 \le d\le N^{1/4}$:
	\begin{align*}
		s_{2}(\mathcal{E}) &\ge \sigma_d\left(\frac{c}{p^{3/2}} - \frac{1}{N^{2}\sigma_d^{2p}}\right)^{1/2p}  \ge \sigma_d\left(\frac{c}{p^{3/2}} - \frac{1}{N}\right)^{1/2p}\\
		& \ge \sigma_d\left(\frac{c^\prime}{\ln(N)^{3/2}}\right)^{1/2p}.
	\end{align*}
\end{proof}
\begin{rem}
    The generalisation of the second Alon-Boppana bound given by Lemma \ref{Alonbop2} to the general quantum channel given by \eqref{QCgeneral} is not direct. One way could be to assume a lower bound on the trace $\frac{1}{N} \Tr[ K(s^{2p}) K(s^{2p-1})^{\ast} \cdots K(s^{1})]$ but this would not take advantage of the condition \eqref{tracepres}. Moreover, one has a lower bound for $\|\mathcal{T}|_{\mathds{1}^{\perp}}\|$ in the cases where $(K(s)\otimes \overline{K(s)})_{s\in [d]}$ is known to converge in distribution to a family $(a(s))_{s\in [d]}$, and if the distribution of the limit family can be computed. We recall that without the unitary hypothesis on $(K(s))_{s \in [d]}$, the asymptotic freeness of $(K(s))_{s\in [d]}$ does not imply freeness for $(K(s)\otimes \overline{K(s)})_{s\in [d]}$ and the computation of the limit distribution of the sum is not trivial (see tensor convolution \cite[Theorem 2.2]{lancien2023limiting}). Nevertheless, if $\mathcal{T}$ defined by \eqref{QCgeneral} converges in distribution to an operator $b \in \left(\mathcal{A},\tau\right)$ where $\tau$ is a faithful trace, then we have $\|\mathcal{T}|_{\mathds{1}^{\perp}}\| \ge  \left( \frac{1}{N} \Tr \left[ \left( \mathcal{T}\mathcal{T}^{\ast}\right)^{m} \right]\right)^{1/2m} \ge \left(\tau \left[ \left(bb^\ast \right)^m\right]\left(1 + o(1)\right)\right)^{1/2m}\underset{m\rightarrow \infty}{\rightarrow } \|b\|$.  The lower bound becomes tight as soon as the spectrum of $b$ is continuous.
\end{rem}
\section{Trace method and Schwinger-Dyson equations}\label{introtrace}
From now on, for all dimensions $N$, for all Kraus degrees $d=d_N$, we consider $(U(s))_{s\in [d]}$ iid Haar distributed unitary matrices of dimension $N$. 
\subsection{Strategy}
The proof of Theorem \ref{hastingsdetails} comes from the following inequalities on singular and eigenvalues. For any integer $m \ge 1$ we have:
\begin{align*}
	|\lambda_{2}(\mathcal{E})|^{2m} = |\lambda_{2}(\mathcal{E}^{m})|^2
	\le s_{2}(\mathcal{E}^m)^2.
\end{align*}
The previous inequality on singular and eigenvalue is not true in general. Here it holds because the vector space spanned by $\mathbf{1} = (1,...,1)$, corresponding to the largest singular and eigenvalue space, is stable for $\mathcal{E}$ and $\mathcal{E}^{\ast}$. Therefore we have $\{\mathcal{E}|_{\mathbf{1}^\perp}\}^m=\mathcal{E}^m|_{\mathbf{1}^\perp}$ and $\lambda_1(\mathcal{T}|_{\mathbf{1}^\perp}) = \lambda_2(\mathcal{T}) \le s_{1}(\mathcal{T}|_{\mathbf{1}^\perp}) =s_2(\mathcal{T})$ for $\mathcal{T}= \mathcal{E}^m$ for all integer $m$. To prove Theorem \ref{hastingsdetails}, Proposition \ref{powers} or Theorem \ref{powers1} we use the trace method and Markov inequality. For Corollary \ref{hastingscor} we will use the fact that for any $\epsilon>0$ and for any integer $m \ge 1$ we have:
\begin{equation*}
	\mathbb{P} \left(|\lambda_2(\mathcal{E})| \ge \rho_d(1 + \epsilon) \right) \le \mathbb{P} \left(s_2(\mathcal{E}^m)^{2} \ge  \rho_d^{2m}(1+ \epsilon)^{2m}\right) \le \frac{\mathbb{E}(s_2(\mathcal{E}^m)^{2})}{\rho_d^{2m}(1 + \epsilon)^{2m}}.
\end{equation*}  
Using the notation of Section \ref{computationoftraces}, we bound the right hand side of the previous inequality using the following:

\begin{align*}
	1 + \mathbb{E}\left(s_{2}(\mathcal{E}^m)^2\right) \le \mathbb{E}\left(\sum_{a=1}^{N^2}s_{a}(\mathcal{E}^m)^2\right)= \mathbb{E}\left(\mathrm{E}_1\right).
\end{align*}
For all integers $m$, when there is no confusion, we set:
\begin{equation}\label{defE_1}
	\mathbb{E}_1 := \mathbb{E}\left(\mathrm{E}_1\right) = \left(\frac{1}{d}\right)^{2m} \sum_{\mathcal{S} =(s_j)_{j\in[2m]} \in [d]^{2m}} \mathbb{E}_0 \left(\mathcal{S}\right),
\end{equation} 
where for all $\mathcal{S}=(s_1,...,s_{2m}) \in [d]^{2m}$ we denote:
\begin{equation}\label{defEtrm}
	\mathbb{E}_{0}(\mathcal{S}) := \mathbb{E}\left(|\Tr[\mathcal{U}(\mathcal{S})]|^2\right) = \mathbb{E}\left(|\Tr[U(s_{2m})\cdots U(s_{m+1})U(s_m)^\ast \cdots U(s_1)^\ast]|^2\right).
\end{equation}
The Markov inequality above becomes:
\begin{equation}\label{markov} 
	\mathbb{P} \left(|\lambda_2(\mathcal{E})| \ge \rho_d(1 + \epsilon) \right) \le \frac{\mathbb{E}_1-1}{\rho_d^{2m}(1 + \epsilon)^{2m}}.
\end{equation}  
To prove Proposition \ref{powers} we will use the fact that for all integers $m,p \ge 1$, for all $\epsilon >0$ we have:
\begin{equation*}
	\mathbb{P} \left( s_{2}(\mathcal{E}^{m})^{1/m} \ge \rho_d(1 + \epsilon) \right) = 	\mathbb{P} \left( \lambda_2(\{{\mathcal{E}^{\ast}}^m\mathcal{E}^m\}^p) \ge \rho_d^{2mp}(1 + \epsilon)^{2mp} \right) \le \frac{\mathbb{E}\left(\lambda_2(\{{\mathcal{E}^{\ast}}^m\mathcal{E}^m\}^p)\right)}{ \rho_d^{2mp}(1 + \epsilon)^{2mp} }.
\end{equation*} 
In order to bound the right hand side of the previous inequality we set for all integers $m,p \ge 1$:
\begin{align*}
	\mathrm{E}_2 :=\tau_N(\{{\mathcal{E}^{\ast}}^m\mathcal{E}^m\}^p)=\sum_{i,j} \langle \{{\mathcal{E}^{\ast}}^m\mathcal{E}^m\}^p(E_{ij}), E_{ij} \rangle = \sum_{a=1}^{N^2} \lambda_a(\{{\mathcal{E}^{\ast}}^m\mathcal{E}^m\})^{p} \ge 1 + \lambda_2(\{{\mathcal{E}^{\ast}}^m\mathcal{E}^m\})^{p}.
\end{align*}
Developing the powers of $\mathcal{E}$ in the expression of $\mathrm{E}_2$ we obtain: 

\begin{align}\label{E2}
	\mathrm{E}_2 &=  \left(\frac{1}{d}\right)^{2mp} \sum_{(s_{j}^{t})_{j \in [m], t\in [2p]} \in [d]^{2mp}}  \Tr[\prod_{t=1}^{2p}\prod_{j=1}^{m} U(s_{2m-j+1}^{2p-t+1})^{\epsilon_{2p-t+1}}]\Tr[\prod_{t=1}^{2p}\prod_{j=1}^{m} U(s_{j}^t)^{-\epsilon_t}] 
\end{align}
where for all $1 \le k \le p$ we have $\epsilon_{2k} = +$ and  $\epsilon_{2k+1} = -$. Therefore for all integers $m,p \ge 1$ and all $\mathcal{S}= (s_{j}^{t})_{t\in [2p], j \in [m]} \in [d]^{2pm}$, when there is no confusion we set:
\begin{align}
	\mathcal{U}(\mathcal{S}) &:= U(s_{m}^{2p})\cdots U(s_{1})U(s_m^{2p-1})^\ast \cdots U(s_1^{2p-1})^\ast U(s_{m}^{2p-2}) \cdots U(s_{1}^{2})U(s_{m}^{1})^\ast\cdots U(s_{1}^{1})^\ast \notag\\
	&= \prod_{t=1}^{2p}\prod_{j=1}^{m} U(s_{2m-j+1}^{2p-t+1})^{\epsilon_{2p-t+1}}\notag\\
	\text{and similarly}~\mathbb{E}_0 (\mathcal{S}) &:= \mathbb{E}(\Tr[\mathcal{U}(\mathcal{S})]\Tr[\mathcal{U}(\mathcal{S})^\ast]). \label{defEtrmp}
\end{align}
Taking notations from Equation \eqref{E2} and \eqref{defEtrmp} we introduce the following expectation:
\begin{align}
	\mathbb{E}_2 &:= \mathbb{E}(\mathrm{E}_2) = \left(\frac{1}{d}\right)^{2mp} \sum_{\mathcal{S} = (s^{t}_j)_{(t,j)\in [2p]\times[m]} ~ \in [d]^{2mp}} \mathbb{E}_0 (\mathcal{S}). \label{defE_2}
\end{align}
The Markov inequality in this case becomes:
\begin{equation}\label{markov1}
	\mathbb{P} \left( s_{2}(\mathcal{E}^{m})^{1/m} \ge \rho_d(1 + \epsilon) \right) \le \frac{\mathbb{E}_2-1}{ \rho_d^{2mp}(1 + \epsilon)^{2mp} }.
\end{equation} 
Finally in order to prove Theorem \ref{powers1} and Corollary \ref{hastings2} we consider the expression of $\mathrm{E}_2$ given by Equation \eqref{E2} in the case where $m=1$. Indeed using Markov inequality, for all integers $p \ge 1$ we have:
\begin{equation*}
	\mathbb{P} \left( s_{2}(\mathcal{E}) \ge \sigma_d(1 + \epsilon) \right) = 	\mathbb{P} \left( \lambda_2(({\mathcal{E}^\ast}\mathcal{E})^p) \ge \sigma_d^{2p}(1 + \epsilon)^{2p} \right) \le \frac{\mathbb{E}\left(\lambda_2(({\mathcal{E}^{\ast}}\mathcal{E})^p)\right)}{ \sigma_d^{2p}(1 + \epsilon)^{2p} }.
\end{equation*} 
Therefore for all integers $p \ge 1$ we have on the one hand:
\begin{align*}
	\mathrm{E}_3 :=\tau_N(({\mathcal{E}^\ast}\mathcal{E})^p)=\sum_{i,j} \langle ({\mathcal{E}^\ast}\mathcal{E})^p(E_{ij}), E_{ij} \rangle = \sum_{a=1}^{N^2} \lambda_a({\mathcal{E}^\ast}\mathcal{E})^{p} \ge 1 + \lambda_2({\mathcal{E}^\ast}\mathcal{E})^{p}.
\end{align*}
On the other hand, considering Equation \eqref{E2} in the case $m=1$, we have:
\begin{align}\label{E3}
	\mathrm{E}_3 &=  \left(\frac{1}{d}\right)^{2p} \sum_{(s^{t})_{ t\in [2p]} \in [d]^{2p}}  \Tr[\prod_{t=1}^{2p} U(s^{2p-t+1})^{\epsilon_{2p-t+1}}]\Tr[\prod_{t=1}^{2p}U(s^t)^{-\epsilon_t}] 
\end{align}
where for any $1 \le k \le p$ we have $\epsilon_{2k} = +$ and  $\epsilon_{2k+1} = -$. For all integers $p \ge 1$ and all $\mathcal{S}= (s^{t})_{t\in [2p]} \in [d]^{2p}$, when there is no confusion, we now set:
\begin{align}
	\mathcal{U}(\mathcal{S}) &:= U(s^{2p})U(s^{2p-1})^\ast \cdots U(s^{2})U(s^{1})^\ast \notag\\
	\text{and similarly}~\mathbb{E}_0 (\mathcal{S}) &:= \mathbb{E}(\Tr[\mathcal{U}(\mathcal{S})]\Tr[\mathcal{U}(\mathcal{S})^\ast]). \label{defEtrp}
\end{align}
Taking notations from Equation \eqref{E3} and \eqref{defEtrp} we introduce the following expectation:
\begin{align}
	\mathbb{E}_3 &:= \mathbb{E}(\mathrm{E}_3) = \left(\frac{1}{d}\right)^{2p} \sum_{\mathcal{S} = (s^{t})_{t\in [2p]} ~ \in [d]^{2p}} \mathbb{E}_0 (\mathcal{S}). \label{defE_3}
\end{align}
In this case Markov inequality becomes:
\begin{equation}\label{markov2}
	\mathbb{P} \left( s_{2}(\mathcal{E}) \ge \sigma_d(1 + \epsilon) \right) \le \frac{\mathbb{E}_3-1}{ \sigma_d^{2p}(1 + \epsilon)^{2p} }.
\end{equation} 
We then apply Hastings strategy, i.e. we iterate Schwinger-Dyson equation (see Section \ref{Schwinger-Dyson Equations}) over the expressions \eqref{defEtrm}, \eqref{defEtrmp} and \eqref{defEtrp}, in order to upper bound $\mathbb{E}_1$, $\mathbb{E}_2$ and $\mathbb{E}_3$ (respectively given by Equation \eqref{defE_1},\eqref{defE_2} and \eqref{defE_3}). The expectations defined by Equation \eqref{defE_1},\eqref{defE_2} and \eqref{defE_3} make the notation $\mathbb{E}_0(\mathcal{S})$ inconsistent since we have given the same notation for three different definitions. This is only for notational convenience, and since we will be upper-bounding each $\mathbb{E}_i$ in independent Sections (Section \ref{proofHastings}, Section \ref{proofpowers} and \ref{proofpowers1} respectively).\\
Let us now assume that Proposition \ref{powers} and Corollary \ref{powers1cor} are proved, and explain why Corollary \ref{hastings2} and \ref{hastings3} hold. The probability in Corollary \ref{hastings2} verifies:
\begin{equation*} 
	\mathbb{P} \left( |s_{2}(\mathcal{E}^{m})^{1/m} - \rho_d| > \epsilon\rho_d \right) = \mathbb{P} \left( s_{2}(\mathcal{E}^{m})^{1/m} \ge \rho_d\left( 1 + \epsilon\right) \right) + \mathbb{P} \left( s_{2}(\mathcal{E}^{m})^{1/m} \le \rho_d\left(1 - \epsilon\right) \right).
\end{equation*}
Under the hypothesis on $d=d_N$, $m=m_N$ and $(U(s))_{s\in [d]}$ of Corollary \ref{hastings2} the first probability on the right side of the inequality is bounded by applying Proposition \ref{powers} and the second is equal to zero for $N$ large enough by the Alon-Boppana bound given by Lemma \ref{Alonbop}. Likewise the probability in Corollary \ref{hastings3} verifies: 
\begin{equation*} 
	\mathbb{P} \left( |s_{2}(\mathcal{E}) - \sigma_d| > \epsilon\sigma_d \right) = \mathbb{P} \left( s_{2}(\mathcal{E}) \ge \sigma_d\left(1 + \epsilon\right) \right) + \mathbb{P} \left( s_{2}(\mathcal{E}) \le \sigma_d\left( 1 - \epsilon\right) \right).
\end{equation*}
Under the hypothesis on $d=d_N$ and $(U(s))_{s\in [d]}$ of Corollary \ref{hastings3} the first probability on the right side of the inequality is bounded by applying Corollary \ref{powers1cor} and the second is equal to zero for $N$ large enough by Lemma \ref{Alonbop2}.

\subsection{Overview of proofs using Schwinger-Dyson equation}
In this section we give an overview of the proof of Theorem \ref{hastingsdetails} using Schwinger-Dyson equation \eqref{SD} given in the following section. The use of Schwinger-Dyson equation in the proofs of Proposition \ref{powers} and Theorem \ref{powers1} is then quite similar, the main differences being the combinatorial arguments given in Section \ref{proofpowers} and \ref{proofpowers1}. The idea is to show that:

\begin{equation}\label{goal}
    \mathbb{E}_1 \le 1+ \mathrm{C}_N \rho_{d}^{2m}
\end{equation}
where $\mathbb{E}_1$ is defined in the previous section \eqref{defE_1}. The constant $\mathrm{C}_N>0$ goes to infinity as $N$ grows, and one will have to consider the power $m = m_N$ such that $\mathrm{C}_{N}^{1/2m} = \left( 1 + o_N(1) \right)$. We start by rewriting $\mathbb{E}_1$ as a sum of the form:

\begin{equation}\label{appearancetraces}
    \mathbb{E}_1 = \left(\frac{1}{d}\right)^{2m}\sum_{\mathcal{S} \in [d]^{2m}} \mathbb{E}\left( L_1 (\mathcal{S}) L_2(\mathcal{S}) \right)
\end{equation}
where for $\mathcal{S}= (s_1,...,s_{2m}) \in [d]^{2m}$ we set:
\begin{align*}
    L_1 (\mathcal{S}) &= \Tr[U(s_{2m}) U(s_{2m-1}) \cdots U(s_{m+1}) U(s_{m})^\ast U(s_{m-1})^\ast \cdots U(s_1)^\ast]\\
    L_{2} (\mathcal{S}) &=\overline{L_{1}(\mathcal{S})}.
\end{align*}
We will later call $\mathcal{S} =((s_{2m},+),...,(s_{m+1},+),(s_{m},-),...,(s_1,-))$ a \textit{$1$-word} (see Definition \ref{defword}).\\

On the other hand, we introduce Schwinger-Dyson equation in Section \ref{Schwinger-Dyson Equations}. We will now consider a slightly more general expectation of traces than the one given in \eqref{appearancetraces}, i.e:
\begin{equation}\label{generalcase}
    \mathbb{E}_0\left(\mathcal{S}\right) := \mathbb{E} \left(L_1 (\mathcal{S}) \cdots L_k(\mathcal{S}) \right)
\end{equation}
where $k$ is an integer, $\mathcal{S} \in \left([d]\times \{+,-\}\right)^{m_1}\times \cdots \times \left([d]\times \{+,-\}\right)^{m_k} $ and for all $1 \le \ell \le k$ we have:
\begin{equation*}
    L_{\ell}(\mathcal{S}) = \Tr[U(s_{\ell1})^{\epsilon_{\ell1}} \cdots U(s_{\ell m_{\ell}})^{\epsilon_{\ell m_{\ell}}} ] =\Tr[\underbrace{\mathcal{U}(\mathcal{S}(\ell))}_{\in \mathrm{U}_N}].
\end{equation*}
Schwinger-Dyson equation can be rewritten:
\begin{align*}
\mathbb{E}\left(L_1 (\mathcal{S}) \cdots L_k(\mathcal{S}) \right) =\frac{1}{N} \sum_{\mathrm{P} = (\ell, j , \epsilon) \in \NN \times \NN \times \{ +,-\} }   \underbrace{\epsilon (\mathrm{P})\mathbb{E} ( L_1 (\mathcal{S}(\mathrm{P})) \cdots L_{k^\prime}( \mathcal{S} (\mathrm{P}))}_{\overline{e}(\mathcal{S},\mathrm{P})},
\end{align*}
where for all \textit{patterns} $\mathrm{P}$, we have $\epsilon (\mathrm{P}) \in \{-1,0,1\}$ and $ k-1 \le k^{\prime}=k^{\prime}(\mathcal{S},\mathrm{P}) \le k+1$. In the sum above we have at most $2m^2$ non-zero terms where $m = \sum m_{\ell}$. The fact that we are going to apply Schwinger-Dyson equation to \eqref{appearancetraces} and that one can have $k^{\prime} > k$ is the reason why one has to consider the general case \eqref{generalcase}. One can have $\mathcal{U}(\mathcal{S}(\mathrm{P}) (\ell)) = \Id$ and therefore $L_{\ell} (\mathcal{S}(\mathrm{P}) )= N$. If for all $1\le \ell \le k^{\prime}$ we have $\mathcal{U}(\mathcal{S}(\mathrm{P})(\ell))=\Id$ we say that the path \textit{terminates} (see Definition \ref{mouv&term}). In this case we have $\overline{e}(\mathcal{S},\mathrm{P})= N^{k^\prime}$. Up to some simplifications one can rewrite the previous equation:
\begin{align*}
\mathbb{E}\left(L_1 (\mathcal{S}) \cdots L_k(\mathcal{S}) \right) = \underbrace{\sum_{\text{finishing terms}} \overline{e}(\mathcal{S},\mathrm{P})}_{\mathrm{F}_1} + \frac{1}{N} \sum_{\mathrm{P} = (\ell, j , \epsilon) \in \NN \times \NN \times \{ +,-\} }  N^{b_1}\epsilon (\mathrm{P}) \mathbb{E} ( L_1 (\mathcal{S}(\mathrm{P})) \cdots L_{k^\prime}( \mathcal{S} (\mathrm{P})),
\end{align*}
where now for all pattern $\mathrm{P}$, $L_{\ell}(\mathcal{S}(\mathrm{P}))$ has no \textit{trivial} traces (see Section \ref{encodingmat}, Definition \ref{defword}). We then iterate Schwinger-Dyson equation and obtain:

\begin{align*}
\mathbb{E}\left(L_1 (\mathcal{S}) \cdots L_k(\mathcal{S}) \right) &=\mathrm{F}_1 + \frac{1}{N} \sum_{\mathrm{P} =(\ell, j , \epsilon) \in \NN \times \NN \times \{ +,-\} }  N^{b_1}\epsilon (\mathrm{P}) \mathbb{E} ( L_1 (\mathcal{S}(\mathrm{P})) \cdots L_{k^\prime}( \mathcal{S} (\mathrm{P}))\\
&=  \mathrm{F}_2 + \frac{1}{N^2}\sum_{\mathrm{P} = (\ell_b, j_b , \epsilon_b)_{b \le 2} }  N^{b_2}\epsilon (\mathrm{P}) \mathbb{E} ( L_1 (\mathcal{S}(\mathrm{P})) \cdots L_{k^\prime}( \mathcal{S} (\mathrm{P}))\\
& \vdots \\
& = \mathrm{F}_n + \frac{1}{N^n}\sum_{\mathrm{P} = ( \ell_{b}, j_b ,\epsilon_{b})_{b \in [n]} } N^{b_n}\epsilon (\mathrm{P}) \mathbb{E} ( L_1 (\mathcal{S}(\mathrm{P})) \cdots L_{k^\prime}( \mathcal{S} (\mathrm{P})),
\end{align*}
where the last sum is over all patterns $\mathrm{P} = (\ell_b,j_b,\epsilon_b)_{b\in [n]} \in \left( \NN \times \NN \times \{+,-\} \right)^{n}$. In Section \ref{encodingmat} we describe the algorithm behind the iteration of Schwinger-Dyson equation. In particular we construct \textit{tracking functions} such that for any original matrix $U(s_{\ell j })^{\epsilon_{\ell j}}$, given that $(\ell,j)$ determines its initial \textit{position}, then, given the pattern $\mathrm{P}$ of length $n$ we choose, one can track the matrix after $n$ iterations. In Section \ref{writ_iterations} we show that the sequence $(\mathrm{F}_n)_{n}$ in the previous equation is a convergent series, equal to the initial expectation $\mathbb{E}\left(L_1\cdots L_{k} \right)$ to which we apply the algorithm (see Proposition \ref{cvseries}). In particular the convergence requires conditions on the total number of original matrices $m = \sum_{\ell} m_{\ell}$ and no conditions on the Kraus rank $d$. \\

In Section \ref{rungcanc} we make a special case of the terms obtained with \textit{rung cancellation}. These terms are specific to the initial expectations we are dealing with, i.e. of the form:
\begin{equation*}
    \mathbb{E}\left(L_1(\mathcal{S}) L_{2}(\mathcal{S}) \right)= \mathbb{E} \left (\Tr[ U(s_{2m}) \cdots U(s_i)^{\epsilon_i} \cdots U(s_{1})^{\ast}] \Tr[ U(s_1) \cdots U(s_i)^{-\epsilon_i} \cdots U(s_{2m})^\ast]\right).
\end{equation*}
To roughly describe what it means to have a \textit{rung cancellation}, one needs to replace the matrix $U(s_i)^{\epsilon_i}$ in position $(1,i)$  (\textit{resp}. $U(s_{i})^{-\epsilon_i}$) in position $(2,2m-i)$, see Definition \ref{defposition})  in the previous expression by a Haar distributed matrix $X$ (\textit{resp}. $X^{\ast}$ its inverse) independent of $(U(s))_{s \in [d]}$. Then the algorithm applied to the new original word $\mathcal{S}^i$ will give rise to finishing terms that will also appear in the algorithm applied to $\mathcal{S}$ (see Proposition \ref{rungcanc}). These terms are particularly important because they give exactly the largest singular/eigenvalue in the sum $\mathbb{E}_1 = \sum_{a} \mathbb{E}(s_{a}(\mathcal{E}^m))^2 = 1 + \sum_{a\ge 2} \mathbb{E}(s_{a}(\mathcal{E}^m))^2$ (see Proposition \ref{computexprungcancel} and Lemma \ref{cancelrungcancel}). We then obtain:
\begin{equation*}
    \mathbb{E}_1 = 1 + \sum_{\mathcal{S}\in [d]^{2m}} \sum_{n}\sum_{\ast} \overline{e}(\mathcal{S},\mathrm{P})
\end{equation*}
where the summand $\ast$ is the set of patterns $\mathrm{P}= (\ell_b,j_b,\epsilon_b)_{b \in [n]}$ such that $\overline{e}(\mathcal{S},\mathrm{P})$ \textit{terminates with no rung cancellation}. This corresponds to the equality \eqref{rigoraftercancelrung} in the proof of Theorem \ref{hastingsdetails}, Section \ref{proofHastings}. It then remains to study in more detail the combinatorial properties of the set of words $\mathcal{S}$ and patterns $\mathrm{P}$ that give rise to terms terminating with no rung cancellation (see Definition \ref{equclasspattern}, Lemma \ref{termsnorungcancel} and Lemma \ref{cardeqclass}) to obtain \eqref{goal}. Equation \eqref{goal} corresponds to Equation \eqref{goal1} in the proof of Theorem \ref{hastingsdetails} (Section \ref{proofHastings}).

\subsection{Schwinger-Dyson Equations}\label{Schwinger-Dyson Equations}
In this section we consider $k \ge 1$ and $ m_1,...,m_k$ fixed integers. For all $1 \le \ell \le k$ we consider $\mathcal{S}(\ell)=((s_{\ell1},\epsilon_{\ell 1}),...,(s_{\ell m_{\ell}},\epsilon_{\ell m_{\ell}}))\in ([d]\times \{+,-\})^{m_{\ell}}$ and we set:  

\begin{align}
	\mathcal{U}({\mathcal{S}(\ell)}) := U(s_{\ell1})^{\epsilon_{\ell1}} U(s_{\ell 2 })^{\epsilon_{\ell 2}}\cdots U(s_{\ell{m_\ell}})^{\epsilon_{\ell m_{\ell}}} ~~ \text{and}~~L_{\ell} := \Tr[\mathcal{U}(\mathcal{S}(\ell))]. \label{defLell}
\end{align}
Denoting $\mathcal{S} = (\mathcal{S}(1),...,\mathcal{S}(k))$ we set:
\begin{equation}\label{Etr}
	\mathbb{E}^\prime (\mathcal{S}) := \mathbb{E}(L_1 \cdots L_k).
\end{equation}
\begin{rem}\label{remlinE0Eprime}
	One can notice that the definition of $\mathbb{E}^\prime$ given in \eqref{Etr} is linked with the definition of $\mathbb{E}_{0}$ used in the previous section and given by \eqref{defEtrmp} and \eqref{defEtrm}. In fact if we consider $\mathcal{S}=(s_{j}^t)_{j\in [m], t\in [2p]}\in[d]^{2mp}$, then we can set $\tilde{\mathcal{S}} = (\tilde{\mathcal{S}}(1),\tilde{\mathcal{S}}(2))$ defined by:
	\begin{align*}
		\tilde{\mathcal{S}}(1)&= ((s_{m}^{2p},+),...,(s_{1}^{2p},+),(s_{m}^{2p-1},-),...,(s_{j+1}^t,\epsilon_t),(s_{j}^{t},\epsilon_t),..., (s_{1}^1,-))\\
		\tilde{\mathcal{S}}(2)&=((s_{1}^{1},+),(s_{2}^1,+),..., (s_{m}^1,+),(s_{1}^2,-),...,(s_{j}^{t},-\epsilon_t),(s_{j+1}^{t},-\epsilon_t),...,(s_{m}^{2p},-))
	\end{align*}
	where $\epsilon_{t}= +$ for $t$ even and $\epsilon_{t}= -$ for $t$ odd. We then have that $\mathcal{U}(\tilde{\mathcal{S}}(1))= \mathcal{U}(\tilde{\mathcal{S}}(2))^\ast$ and:
	$$\mathbb{E}_{0}(\mathcal{S}) = \mathbb{E}^\prime(\tilde{\mathcal{S}})$$
	where $\mathbb{E}_{0}(\mathcal{S})$ is given by \eqref{defEtrmp}.
\end{rem}
The following proposition states the Schwinger-Dyson equation given by Hastings \cite[Equation (19,20)]{hastings2007random}.
\begin{prop}[Schwinger-Dyson equation]\label{SDPROP}
	Let $(U(s))_{s \in [d]}$ be iid Haar distributed $N$-dimensional unitary matrices. Let $k$ and $m_1,...,m_k$ be integers. For all $1 \le \ell \le k$ let $\mathcal{S}(\ell) = (s_{\ell j},\epsilon_{\ell j})_{j\in [m_\ell]} \in ([d] \times \{+,-\})^{m_\ell}$. For all $2 \le j \le m_1$ we set:
	\begin{align*}
		L_{1}^1(1,j,+)&:= \Tr[U(s_{11})^{\epsilon_{11}}\cdots U(s_{1j-1})^{\epsilon_{1j-1}}] \\
		L_{1}^2(1,j,+)&=\Tr[U(s_{1j})^{\epsilon_{1j}} \cdots U(s_{1m_1})^{\epsilon_{1m_1}}] \\
		L_{1}^1(1,j,-)&:=\Tr[U(s_{11})^{\epsilon_{11}}\cdots U(s_{1j})^{\epsilon_{1j}}]  \\
		L_{1}^2(1,j,-)&:=\Tr[U(s_{1j+1})^{\epsilon_{1j+1}} \cdots U(s_{1m_1})^{\epsilon_{1m_1}}] 
	\end{align*}
	and for all $1 \le \ell \le k$ and all $1 \le j \le m_\ell$ we set:
	\begin{align*}
		L_{1}(\ell,j,+)&:=\Tr[U(s_{11})^{\epsilon_{11}}\cdots U(s_{1m_1})^{\epsilon_{1m_1}}U(s_{\ell j})^{\epsilon_{\ell j}} \cdots U(s_{\ell m_\ell})^{\epsilon_{\ell m_\ell}}U(s_{\ell 1})^{\epsilon_{\ell 1}} \cdots U(s_{\ell j-1})^{\epsilon_{\ell j-1}}  ]\\
		L_{1}(\ell,j,-)&=\Tr\left[U(s_{11})^{\epsilon_{11}}\cdots U(s_{1m_1})^{\epsilon_{1m_1}}U(s_{\ell j+1})^{\epsilon_{\ell j+1}} \cdots U(s_{\ell m_\ell})^{\epsilon_{\ell m_\ell}}U(s_{\ell 1})^{\epsilon_{\ell 1}} \cdots U(s_{\ell j})^{\epsilon_{\ell j}}  \right].
	\end{align*}
	Using notations of Equations \eqref{defLell} and \eqref{Etr} we have the following equality:
	
	\begin{align}
		\mathbb{E}^{\prime}(\mathcal{S})&=\mathbb{E}(\Tr[U(s_{11})^{\epsilon_{11}} U(s_{1 2 })^{\epsilon_{1 2}}\cdots U(s_{1{m_1}})^{\epsilon_{1 m_{1}}}]L_2 \cdots L_k) \label{SD} \\
		&= -\frac{1}{N}\sum_{j=2}^{m_1} \mathds{1}((s_{11},\epsilon_{11})=(s_{1j},\epsilon_{1j})) \mathbb{E}(L_{1}^1(1,j,+) L_{1}^2(1,j,+) L_2 \cdots L_k)\label{SD1}\\
		& + \frac{1}{N} \sum_{j=2}^{m_1}\mathds{1}((s_{11},\epsilon_{11})=(s_{1j},-\epsilon_{1j})) \mathbb{E}(L_{1}^1(1,j,-)L_{1}^2(1,j,-)L_2 \cdots L_k)\label{SD2}\\
		& - \frac{1}{N}\sum_{\ell=2}^{k} \sum_{j=1}^{m_{\ell}} \mathds{1}((s_{11},\epsilon_{11})=(s_{\ell j},\epsilon_{\ell j})) \mathbb{E}(L_{1}(\ell,j,+)L_2 \cdots L_{\ell-1} L_{\ell+1} L_k)\label{SD3}\\
		& + \frac{1}{N}\sum_{\ell=2}^{k} \sum_{j=1}^{m_{\ell}} \mathds{1}((s_{11},\epsilon_{11})=(s_{\ell j},-\epsilon_{\ell j})) \mathbb{E}(	L_{1}(\ell,j,-) L_2 \cdots L_{\ell-1} L_{\ell+1} L_k)\label{SD4}.
	\end{align}
\end{prop}

In order to prove the previous equations one needs to consider the matrices:
\begin{equation*}
	\Delta_{l k} = \frac{1}{\sqrt{2}} (E_{lk} + E_{kl}), ~~
	\delta_{lk} = \frac{i}{\sqrt{2}} (E_{lk} - E_{kl}).
\end{equation*}
The set of these matrices for all $(l,k) \in [N]^2$ is a basis for the $\mathbb{R}$-vector space of $N$-dimensional hermitian matrices. We denote by $\Theta$ this basis. 
\begin{lem}\label{Aubrun}
	Let $X, Y \in \mathrm{M}_{N} (\mathbb{C})$. We have:
	\begin{align*}
		\sum_{T \in \Theta} \Tr(X T^2) = N \Tr(X)\\
		\sum_{T \in \Theta} \Tr(X T Y T) = \Tr(X) \Tr(Y)\\
		\sum_{T \in \Theta} \Tr(X T) \Tr(Y T) = \Tr(XY).
	\end{align*}
	
	\begin{proof}[Proof of Schwinger-Dyson Equation \ref{SDPROP}]
		We consider $s_0 \in [d]$ such that $s_{0}= s_{11}$ and $T \in \Theta$. For $t \ge 0$, we operate the change of variable:
		$$U(s_{0}) \rightarrow \mathrm{e}^{itT}U(s_0)^{\epsilon_{11}}$$
		in the expectation $\mathbb{E}(\Tr[TU(s_{11})^{\epsilon_{11}} U(s_{1 2 })^{\epsilon_{1 2}}\cdots U(s_{1{m_1}})^{\epsilon_{1 m_{1}}}]L_2 \cdots L_k)$. For all $(\ell,j)$ we set:
		\begin{align*}
			&\text{if}~~ (s_{\ell j }, \epsilon_{\ell j})= (s_{11},\epsilon_{11}):~~X_{\ell j}^t = \mathrm{e}^{itT}U(s_0)^{\epsilon_{11}}\\
			&\text{if}~~ (s_{\ell j }, \epsilon_{\ell j})= (s_{11},-\epsilon_{11}):~~X_{\ell j}^t = U(s_0)^{-\epsilon_{11}}\mathrm{e}^{-itT}\\
			&\text{otherwise}:~~X_{\ell j}^t = X_{\ell j}= U(s_{\ell j})^{\epsilon_{\ell j}}.
		\end{align*}
		For all $t\ge 0$ and $1 \le \ell \le k$ we set $L_{\ell}^t := \Tr[ X_{\ell 1}^t \cdots X_{\ell m_{\ell}}^t ]$. Due to the invariance of the Haar measure under unitary deterministic transformation we have $\Tr[T X_{\ell 1}^t \cdots X_{\ell m_{\ell}}^t] L_{2}^t \cdots L_{k}^t \sim \Tr[T X_{\ell 1} \cdots X_{\ell m_{\ell}}] L_{2} \cdots L_{k}$. Therefore, the first order expression in $t$ of $\mathbb{E}(\Tr[T X_{\ell 1}^t \cdots X_{\ell m_{\ell}}^t] L_{2}^t \cdots L_{k}^t)$ is:
		\begin{align*}
			0 =& \mathbb{E}(\Tr[T^2 X_{11} \cdots X_{1 m_{1}}]L_{2} \cdots L_{k}) \\
			&+ \sum_{j=2}^{m_1} \mathds{1}((s_{11},\epsilon_{11})=(s_{1j},\epsilon_{1j}))\mathbb{E}(\Tr[T X_{11} \cdots X_{1j-1}TX_{1j} X_{1 j+1}\cdots X_{1 m_{1}}]L_{2} \cdots L_{k}]\\
			& -\sum_{j=2}^{m_1} \mathds{1}((s_{11},\epsilon_{11})=(s_{1j},-\epsilon_{1j}))\mathbb{E}(\Tr[T X_{11} \cdots X_{1j} T X_{1 j+1}\cdots X_{1 m_{1}}]L_{2} \cdots L_{k}]\\
			& + \sum_{\ell=2}^{k} \sum_{j=1}^{m_\ell} \mathds{1}((s_{11},\epsilon_{11})=(s_{\ell j},\epsilon_{\ell j}))\mathbb{E}(\Tr[T X_{1 1} \cdots X_{1 m_{1}}]\Tr[T X_{\ell 1} \cdots X_{\ell j-1} T X_{\ell j}\cdots X_{\ell m_{\ell}}]L_{2} \cdots L_{\ell-1} L_{\ell+1}\cdots L_{k}]\\
			&- \sum_{\ell=2}^{k}\sum_{j=1}^{m_1} \mathds{1}((s_{11},\epsilon_{11})=(s_{1j},-\epsilon_{\ell j}))\mathbb{E}(\Tr[T X_{11} \cdots X_{1 m_{1}}]\Tr[X_{\ell 1} \cdots X_{\ell j} T X_{1 j+1}\cdots X_{\ell m_{\ell }}]L_{2} \cdots L_{\ell-1} L_{\ell+1}\cdots L_{k}].
		\end{align*} 
		Applying Lemma \ref{Aubrun}, we obtain the desired equation.
	\end{proof} 
\end{lem}

\section{Coding iterations of Schwinger Dyson equations}\label{algo}
Applying Schwinger-Dyson equation to any term of the form $\mathbb{E}(L_1 \cdots L_k)$ of Section \ref{Schwinger-Dyson Equations} gives rise to terms $\mathbb{E}(L_1^{\prime} \cdots L_{k^\prime}^{\prime})$ of the same form, with the same original matrices $(s_{\ell j },\epsilon_{\ell j})$ up to some cancellations and permutations of their position in the expressions of $L^{\prime}_{\ell^\prime}$. This can be seen as a random process obtained by an algorithm. Before giving the algorithm, we introduce some definitions and notations.

\begin{defi}\label{defword}
	Let $p,p^\prime \ge 0$ be integers. We call $\mathcal{S}=((s_1,\epsilon_{1}),...,(s_{p},\epsilon_{p})) \in ([d]\times\{+,-\})^{p}$ and $\mathcal{S}^\prime=((s_1^{\prime},\epsilon_{1}^\prime),...,(s_{p^\prime}^\prime,\epsilon_{p^\prime}^\prime)) \in ([d]\times\{+,-\})^{p^\prime}$ two \textbf{$1$-words}. We call $p$ (\textit{resp} $p^\prime$) \textbf{the length} of the $1$-word $\mathcal{S}$ (\textit{resp}. $\mathcal{S}^\prime$). We say that these two words \textbf{are equivalent} if there exists $o \in \mathbb{N}$ such that:
	$$U(s_{o+1})^{\epsilon_{o+1}}U(s_{o+2})^{\epsilon_{o+2}}\cdots U(s_p)^{\epsilon_p}U(s_{1})^{\epsilon_1}\cdots U(s_{o})^{\epsilon_{o}}=U(s_{1}^\prime)^{\epsilon_{1}^\prime}\cdots U(s_{p^\prime}^\prime)^{\epsilon_{p^\prime}^\prime}.$$
	If $\mathcal{S}$ and $\mathcal{S}^\prime$ are two equivalent $1$-words we denote:
	$$\mathcal{S} \sim \mathcal{S}^\prime.$$
	For all $1$-word $\mathcal{S}$ we denote by $\ell(\mathcal{S})$ the minimal length for a $1$-word $\mathcal{S}^{\prime}$ equivalent to $\mathcal{S}$. We then call $\mathcal{S}^\prime=((s_1^{\prime},\epsilon_{1}^\prime),...,(s_{p^\prime}^\prime,\epsilon_{p^\prime}^\prime))\sim \mathcal{S}$ with length $p^{\prime}=\ell(\mathcal{S})$ a \textbf{minimal writing} of the equivalence class of $\mathcal{S}$. For a given minimal writing $\mathcal{S}^\prime$, all minimal writings are obtained by considering all $((s_{o+1}^{\prime},\epsilon_{o+1}^\prime),...,(s_{p^\prime}^{\prime},\epsilon_{p^\prime}^\prime),(s_{1}^\prime,\epsilon_{1}^\prime),..., (s_{o}^\prime,\epsilon_{o}^\prime))$ for $o$ an integer. We say that $\mathcal{S}^\prime$ is \textbf{the minimal writing of $\mathcal{S}$} if $((s_1^{\prime},\epsilon_{1}^\prime),...,(s_{p^\prime}^\prime,\epsilon_{p^\prime}^\prime))$ is an non-decreasing sub-sequence of $((s_1,\epsilon_{1}),...,(s_{p},\epsilon_{p}))$. Finally we say that the $1$-word $\mathcal{S}$ is \textbf{of trivial trace} if it is equivalent to the empty word, i.e. when for some integer $o \in \mathbb{N}$ we have that:
	$$U(s_{o+1})^{\epsilon_{o+1}}U(s_{o+2})^{\epsilon_{o+2}}\cdots U(s_p)^{\epsilon_p}U(s_{1})^{\epsilon_1}\cdots U(s_{o})^{\epsilon_{o}} = \mathrm{Id}.$$ 
	In general we set:
	$$ \mathcal{U}(\mathcal{S}) := U(s_{1})^{\epsilon_1} \cdots U(s_{p})^{\epsilon_p}.$$
	\\
	We call $\mathcal{S} = ((s_{\ell j},\epsilon_{\ell j})_{j\in[m_\ell]})_{\ell \le k}$ and $\mathcal{S}^\prime = ((s_{\ell j}^\prime,\epsilon_{\ell j}^\prime)_{j\in[m_\ell^\prime]})_{\ell \le k}$ two \textbf{$k$-words}. We say that these two $k$-words \textbf{are equivalent} if for all $1 \le \ell \le k$ we have $\mathcal{S}(\ell)=((s_{\ell 1},\epsilon_{\ell 1}),...,(s_{\ell m_{\ell}},\epsilon_{\ell m_{\ell}})) \sim \mathcal{S}^\prime(\ell)=((s_{\ell 1}^\prime,\epsilon_{\ell 1}^\prime),...,(s_{\ell m_{\ell}^\prime}^\prime,\epsilon_{\ell m_{\ell}^\prime}^\prime))$. We say that $\mathcal{S}^{\prime}$ is \textbf{a (\textup{resp.} the) minimal writing} of the equivalence class of $\mathcal{S}$ if for all $1 \le \ell\le k$, $\mathcal{S}^\prime(\ell)$ is a (\textup{resp.} the) minimal writing of its equivalence class. All minimal writings are called \textbf{minimal word}. Finally we say that the $k$-word $\mathcal{S}$ is \textbf{of trivial trace} if for all $1 \le \ell \le k$, $\mathcal{S}(\ell)$ is of trivial trace.
\end{defi}

%\begin{align}\label{Traceorigins}
%\mathbb{E}_{0}(\mathcal{S}) := \mathbb{E}(\Tr[\mathcal{U}(\mathcal{S})]\Tr[\mathcal{U}(\mathcal{S})^\ast])
%\end{align}
%where we set:
%\begin{equation}\label{mathcalUS}
%\mathcal{U}(\mathcal{S}) := U(s_{2m})^{\epsilon_{2m}} \cdots U(s_{1})^{\epsilon_1}.
%\end{equation}
%To re-use notations of Section \ref{Schwinger-Dyson Equations} we notice that denoting:
%\begin{equation*}
%X_{1j} := U(s_{2m+1-j})^{\epsilon_{2m+1-j}} ~~\text{and}~~X_{2j}=U(s_j)^{\epsilon_j}
%\end{equation*}

%we have:
%$$X_{11}\dots X_{1m_1}=\mathcal{U}(\mathcal{S})=(X_{21}\dots X_{2m_2})^\ast=\mathcal{U}(\mathcal{S})^\ast.$$
%We therefore set for our fixed original word $\mathcal{S}=((s_1,\epsilon_1),...,(s_{2m},\epsilon_{2m})) \in ([d]\times\{+,-\})^{2m}$:
%\begin{align}\label{codeorigins}
%	&(s_{1j}, \epsilon_{1j}) = (s_{2m+1-j},-\epsilon_{2m+1-j}) \\
%	&(s_{2j}, \epsilon_{2j}) = (s_{j},\epsilon_j). \notag
%\end{align}

\begin{rem}
	These notions derive their interest from the fact that, when iterating the Schwinger-Dyson equation, we can consider that we are applying the equation to some minimal $k$-word $\mathcal{S}$ which gives rise to other $k^\prime$-words $\mathcal{S}^{\prime}$. Before reapplying the Schwinger-Dyson equation, one must choose one of the resulting $\mathcal{S}^{\prime}$ and consider its minimal writing in order to proceed. In particular it is important to note that for $\mathcal{S}=((s_1,\epsilon_{1}),...,(s_{m},\epsilon_{m}))$ and $\mathcal{S}^\prime=((s_1^{\prime},\epsilon_{1}^\prime),...,(s_{m^\prime}^\prime,\epsilon_{m^\prime}^\prime))$ two equivalent \text{$1$-words}, we do not necessarily have $\mathcal{U}(\mathcal{S}) = \mathcal{U}(\mathcal{S}^\prime)$ but:
	$$\Tr(\mathcal{U}(\mathcal{S}))=\Tr(\mathcal{U}(\mathcal{S}^\prime)).$$
\end{rem}

\begin{ex} Let $U(1),U(2), U(3)$ be three independent Haar distributed unitaries. We set $\mathcal{S}=((1,+),(1,-),(2,+),(3,-),(2,+),(2,-))$. We consider:
	$$\mathcal{U}(\mathcal{S}):= U(1) U(1)^\ast U(2)U(3)^{\ast}U(2)U(2)^\ast.$$ 
	\begin{itemize}
		\item The word $\mathcal{S}$ is equivalent to $\mathcal{S^\prime}=((2,+),(3,-),(2,+),(2,-))$ (we have $\mathcal{U}(\mathcal{S}^{\prime}):=  U(2)U(3)^{\ast}U(2)U(2)^\ast$)
		\item A minimal writing of $\mathcal{S}$ is $\mathcal{S^\prime}=((2,+),(3,-))$ (we have $\mathcal{U}(\mathcal{S}^{\prime}):=  U(2)U(3)^{\ast} \neq \mathcal{U}(\mathcal{S})$)
		\item The minimal writing of $\mathcal{S}$ is  $\mathcal{S^\prime}=((3,-),(2,+))$.
	\end{itemize}
\end{ex}

\subsection{Encoding the matrices' movements}\label{encodingmat}

Throughout this section we fix a minimal $k$-word $\mathcal{S} = ((s_{\ell j},\epsilon_{\ell j})_{j\in[m_\ell]})_{\ell \le k}$, where $k$ and $m_1,...,m_k$ are fixed integers (see Definition \ref{defword}). We denote: 
\begin{equation}\label{defm}
	m:= \sum_{\ell=1}^k m_{\ell}.
\end{equation}
The aim of this section is to encode the movements of the matrix $(s_{\ell j},\epsilon_{\ell j}) \sim U(s_{\ell j })^{\epsilon_{\ell j}}$ initially in position $(\ell,j)$ after successive iterations of Schwinger-Dyson equations applied to \eqref{Etr}. We start our algorithm by setting the \textbf{initial path}:
$$e_{0} = ((1,j)_{j\in[m_1]},...,(k,j)_{j\in[m_k]})~~~p_{0}=0, ~~\epsilon^{(0)}=1.$$
Also we set:
$$\overline{e}_{0} = \mathbb{E}^\prime (\mathcal{S})$$
where $\mathbb{E}^\prime(\cdot)$ is defined by \eqref{Etr}. The idea is to give an explicit algorithm that gives us the evolution of the traces on the right hand side of the equation when we iterate Schwinger-Dyson equations. For all $n\ge 0$ we define for all sequences $(\ell_1,\ell_2,...,\ell_n) \in \mathbb{N}^{n}$, $(j_{1},...,j_{n}) \in \mathbb{N}^n$ and $(\epsilon_1,...,\epsilon_n) \in \{+,-\}^n$ the sequence of \textbf{paths} $e((\ell_b,j_b,\epsilon_b)_{b\in[n]})$ following the coming algorithm. We initiate our algorithm with $n=1$ and a triplet $\mathrm{P}=(\ell_1,j_1,\epsilon_1)$ that we will call a \textbf{pattern} of length $1$. \\
If $\ell_1 = 1$ and $j_1 = 1$, or $\ell_1 \ge k+1$, or $j_1\ge m_{\ell_1}+1$, or $s_{11} \neq s_{\ell_1,j_{1}}$ we set:
$$e(\ell_1,j_1,\epsilon_1) = \emptyset ~~~\text{and}~~~ \bar{e}(\ell_1,j_1,\epsilon_1)=0.$$
If $\ell_1=1$, $2 \le j_1 \le m_1$ and $s_{11} = s_{\ell_1,j_{1}}$ then:
\begin{itemize}
	\item if we have $\epsilon_1 =+$ and $\epsilon_{11} = \epsilon_{1} \epsilon_{1 j_1}=\epsilon_{1 j_1}$, i.e. when we consider a term in the first line of Schwinger-Dyson equation \eqref{SD1},  then we set:
	\begin{align} 
		\epsilon^{(1)} &= -1 \notag\\
		e(1,j_1,\epsilon_1) &= ((1,j)_{j\le {j_1-1}}, (1,j)_{j_1 \le j\le m_1},(2,j)_{j\in[m_2]},...,(k,j)_{j\in [m_k]}),\notag\\
		e_{rs}(1,j_1,\epsilon_1) &= ((1,j)_{j\le {j_1-1}}, (2,j)_{j\le m_1-j_1+1},(3,j)_{j\in[m_2]},...,(k+1,j)_{j\in [m_k]})\notag \\
		\mathcal{S}(\ell_1,j_1,\epsilon_1)&= ((s_{1j}, \epsilon_{1j})_{ j \le j_1-1},(s_{1j}, \epsilon_{1j})_{j_1  \le j \le m_1},(s_{2j},\epsilon_{2j})_{j \le m_2},...,(s_{k j},\epsilon_{k j})_{j \in [m_k]}). \label{mathS1}
	\end{align}
	We call $e(\ell_1,j_{1},\epsilon_1)$ \textbf{the path of} the pattern $(\ell_1, j_1,\epsilon_1)$, $e_{rs}$ its \textbf{rescaled path} and $\mathcal{S}(\ell_1,j_1,\epsilon_1)$ the \textbf{word of generation 1}. In this case and with respect to $\mathcal{S}$ minimality, for all $1\le \ell \le k+1$ the word $\mathcal{S}(\ell_1,j_1,\epsilon_1)(\ell)$ can not be of trivial trace. Therefore in this case we set:
	\begin{equation*}
		p_{1} = p_{0} = 0.
	\end{equation*}
	We also define the function:
	\begin{align*}
		f_{1} : &\mathbb{N} \times \mathbb{N}&&\rightarrow \mathbb{N} \times \mathbb{N}\\
		& (1,j) &&\mapsto (1,j) ~\text{if}~ j \le j_1 -1\\
		& (1,j) &&\mapsto (2,j-j_1) ~\text{if}~ j_1 \le j \le m_1\\
		&(\ell,j) &&\mapsto (\ell +1,j) ~~\text{for all}~~ 2 \le \ell \le k ~~\text{and all}~~ 1 \le j \le m_\ell.\\
	\end{align*}
	We call $f_{1}$ the \textbf{first tracking function}.
	
	We denote by $\mathcal{S}_{min}(\cdot)$ the minimal writing of $\mathcal{S}(\cdot)$, that is:
	$$\mathcal{S}_{min}(\ell_1,j_1,\epsilon_1)= ((s_{1j}^{\prime},\epsilon_{1j}^{\prime})_{j \le m_1^\prime},...,(s_{k+1 j}^{\prime},\epsilon_{k+1 j}^{\prime})_{j \le m_{k+1}^\prime})$$
	where we have for all $1 \le \ell  \le k+1$, $\mathcal{S}_{min}(\ell_1,j_1,\epsilon_1)(\ell)$ is the minimal writing for $\mathcal{S}(\ell_1,j_1,\epsilon_1)(\ell)$ defined by Equation \eqref{mathS1}. The integers $m_1^{\prime},...,m_{k+1}^\prime$ are completely determined by the simplifications in  $\mathcal{U}(\mathcal{S}(\ell_1,j_1,\epsilon_1)(\ell))$, i.e. when in the product given by \eqref{defLell} we see $U(s)^{\epsilon}U(s)^{-\epsilon}$. We set furthermore the \textbf{minimal path} for $(\ell_1,j_1,\epsilon_1)$ as:
	$$e_{min}(\ell_1,j_1,\epsilon_1) = ((1,j)_{j \le m_1^{\prime}},...,(k+1,j)_{j \le m_{k+1}^\prime}).$$
	Finally we define the \textbf{mean of the path $(\ell_1,j_1,\epsilon_1)$}:
	\begin{align*}
		\bar{e}(\ell_1,j_1,\epsilon_1) &= \epsilon^{(1)} \frac{N^{p_1}}{N} \mathbb{E}^\prime(\mathcal{S}_{min}(\ell_1,j_1,\epsilon_1))
	\end{align*}
	where we recall that $\mathbb{E}^\prime(\cdot)$ is defined by Equation \eqref{Etr}. By definition of the minimal writing, there exists strictly increasing functions $\Phi_{1} : [m_1^\prime]\rightarrow [j_1-1] $ and $\Phi_{2} : [m_2^\prime] \rightarrow [j_1;m_1]$ such that for all $1 \le j \le m_1^\prime$, $(s^\prime_{1j},\epsilon_{1 j}^\prime) = (s_{1 \Phi_1(j)},\epsilon_{1 \Phi_1(j)})$ and likewise for $\ell=2$.
	We then define the \textbf{rescaled tracking function}:
	\begin{align*}
		f^{rs}_{1} : &\mathbb{N} \times \mathbb{N} &&\rightarrow \mathbb{N} \times \mathbb{N}\\
		& (1,j) &&\mapsto (1,\Phi^{-1}_1(j)) ~~\text{if}~~ j\in \Phi_1([m_1^\prime])\\
		& (1,j) &&\mapsto (2,\Phi^{-1}_2(j)) ~~\text{if}~~j\in \Phi_2([m_2^\prime])\\
		&(\ell,j) &&\mapsto (\ell+1,j)~~ \text{for all}~~ 2 \le \ell \le k+1 ~~\text{and all}~~1 \le j \le m_\ell \\
		&(\ell,j) &&\mapsto (0,0)~~ \text{otherwise}.
	\end{align*}
	If $\epsilon_{1}= +$ and $ \epsilon_{11} = -\epsilon_{1 j_1}$ we set $e(\ell_1,j_1,\epsilon_1) = \emptyset$ again.
	\item If we have $\epsilon_{1} =-$ and $\epsilon_{11} = -\epsilon_{1 j_1}$
	,i.e. when we consider a term in the second line of Schwinger-Dyson equation \eqref{SD2},  then we set:
	\begin{align}
		e(1,j_1,\epsilon_1) &= ((1,j)_{j\le {j_1}}, (1,j)_{j_1+1 \le j\le m_1},(2,j)_{j\in[m_2]},...,(k,j)_{j\in [m_k]}), \notag\\
		e_{rs}(1,j_1,\epsilon_1) &=  ((1,j)_{j\le {j_1}}, (2,j)_{j\le m_1-j_1},(3,j)_{j\in[m_2]},...,(k+1,j)_{j\in [m_k]})) \notag\\
		\mathcal{S}(\ell_1,j_1,\epsilon_1)&= (((s_{1j}, \epsilon_{1j})_{ j \le j_1}),(s_{1j}, \epsilon_{1j})_{j_1+1  \le j \le m_1},(s_{2j},\epsilon_{2j})_{j \le m_2},...,(s_{kj},\epsilon_{kj})_{j \le m_k}). \notag
	\end{align}
	We then construct $f_1$, $f_{1}^{rs}$, $e_{min}(\ell_1,j_1,\epsilon_1)$ and $\mathcal{S}_{min}(\ell_1,j_1,\epsilon_1)$ as we did in the previous case. 
	Again for $\epsilon_1 = -$ and $\epsilon_{11} =  \epsilon_{1 j_1}$ we set $e(\ell_1,j_1,\epsilon_1) = \emptyset$.
\end{itemize}

%% We finally set the \textit{simplified words of generation one} $(1,j_1,\epsilon_1)$:
%%$$\mathcal{S}_s(\ell_1,j_1,\epsilon_1)= ((s_{1j}^{\prime},\epsilon_{1j}^{\prime})_{j \le m_1^\prime},(s_{2j}^{\prime},\epsilon_{2j}^{\prime})_{j \le m_2^\prime},(s_{3j}^{\prime},\epsilon_{3j}^{\prime})_{j \le m_3^\prime})$$
%%where we considered $((s_{1j}^{\prime},\epsilon_{1j}^{\prime})_{j \le m_1^\prime}) = ((s_{1j}, \epsilon_{1j})_{k \le j \le j_1 -k-1})$, $((s_{2j}^{\prime},\epsilon_{2j}^{\prime})_{j \le m_2^\prime}) = ((s_{1j}, \epsilon_{1j})_{j_1 +k^{\prime} \le j \le m_1-k^{\prime}})$ and $(s_{3j}^{\prime},\epsilon_{3j}^{\prime})_{j \le m_3^\prime} = (s_{2j},\epsilon_{2j})_{j \le m_2}$.
\noindent We now consider $\ell_1 \ge 2$ and $s_{11}= s_{\ell j_1}$.
\begin{itemize}
\item If $\epsilon_1 = +1$ and $\epsilon_{11} = \epsilon_{1} \epsilon_{\ell j_1} = \epsilon_{\ell j_1}$ we refer to \eqref{SD3} and we have: 
\begin{align*} 
	\epsilon^{(1)} &:= -1 \\
	e(\ell_1,j_1,\epsilon_1) &:= (((1,1),...,(1,m_1),(\ell_{1},j_{1}),(\ell_{1},j_{1}+1),...,(\ell_{1},m_{\ell_{1}}),(\ell_{1},1),...,(\ell_{1},j_{1}-1)),(2,j)_{j\in m_{2}},\\
	&~~~~~~~~~~~~~~~~...,(\ell_1 -2,j)_{j\in [m_{\ell-2}]},(\ell_1 -1,j)_{j\in [m_{\ell-1}]},(\ell_1 +1,j)_{j\in [m_{\ell+1}]},...,(k,j)_{j\in [m_k]}).
\end{align*} 

 If $\epsilon_1 = +$ and $\epsilon_{11}=-\epsilon_{\ell j_1}$ we set $e(\ell_1,j_1,\epsilon_1)= \emptyset$.
 
 \item If $\epsilon_1 = -1$ and $\epsilon_{11} = -\epsilon_{\ell_1 j_1}$ we refer to \eqref{SD4} and we have:
\begin{align*}
	\epsilon^{(1)} &= +1 \\
	e(\ell_1,j_1,\epsilon_1) &:= (((1,1),...,(1,m_1),(\ell_{1},j_{1}+1),(\ell_{1},j_{1}+2),...,(\ell_{1},m_{\ell_{1}}),(\ell_{1},1),...,(\ell_{1},j_{1})),(2,j)_{j\in m_{2}},\\
	&~~~~~~~~~~~~~~~~...,(\ell_1 -2,j)_{j\in [m_{\ell-2}]},(\ell_1 -1,j)_{j\in [m_{\ell-1}]},(\ell_1 +1,j)_{j\in [m_{\ell+1}]},...,(k,j)_{j\in [m_k]}).
\end{align*}
 Finally if $\epsilon_{1} =-$ and $\epsilon_{11} = \epsilon_{\ell_1 j_1}$ we set $e(\ell_1,j_1,\epsilon_1) = \emptyset$. It remains to construct the corresponding $\mathcal{S}(\ell_1,j_1,\epsilon_1)$, $\mathcal{S}_{min}(\ell_1,j_1,\epsilon_1)$ $f_1$, $f_1^{rs}$, $e_{min}(\ell_1,j_1,\epsilon_1)$ and $\bar{e}(\ell_1,j_1,\epsilon_1)$ as we did in the first case. In all the possible cases, if $e(\ell_1,j_1,\epsilon_1) \neq \emptyset$, then there exists $k-1 \le k^\prime \le k+1$, $m_1^{\prime},...,m_{k^\prime}^{\prime}$ and $(s_{\ell j}^{\prime},\epsilon_{\ell j}^{\prime})_{j \le m_\ell^\prime} \in ([d]\times\{+,-\})^{m_{\ell}} $ such that:
$$\mathcal{S}_{min}(\ell_1,j_1,\epsilon_1)= ((s_{1j}^{\prime},\epsilon_{1j}^{\prime})_{j \le m_1^\prime},...,(s_{kj}^{\prime},\epsilon_{kj}^{\prime})_{j \le m_{k^\prime}}^\prime).$$
The data of $\mathcal{S} = ((s_{\ell j},\epsilon_{\ell j})_{j\in [m_{\ell}]})_{\ell\in [k]}$ minimal $k$-word and $(\ell_1,j_1,\epsilon_1)$ pattern of length $1$ are the only data needed to recover $p_1$, $\epsilon^{(1)}$, $e(\ell_1,j_1,\epsilon_1)$, $e_{rs}(\ell_1,j_1,\epsilon_1)$ $e_{min}(\ell_1,j_1,\epsilon_1)$, $\mathcal{S}(\ell_1,j_1,\epsilon_1)$, $f_1$ and $f_1^{rs}$.
 
\end{itemize}
We just described the first step of the algorithm applied to a minimal word $\mathcal{S}$ for a given pattern $(\ell_1,j_1,\epsilon_1)$. 
\begin{defi}\label{mouv&term}
	Recall that $\mathcal{S} = ((s_{\ell j},\epsilon_{\ell j})_{j\in [m_{\ell}]})_{\ell\in [k]}$ is fixed and minimal and that we consider $m:= \sum_{\ell} m_\ell$. We fix $1 \le  \ell_0 \le k$, $1 \le i \le m_{\ell_0}$ and a pattern $(\ell_1,j_1,\epsilon_1)\in [k]\times [m] \times \{+,-\}$. 
	\begin{itemize}
		\item We say that \textbf{the matrix in position $(\ell_0,i)$ moves to position $(\ell,j)$ in the path $(\ell_1,j_1,\epsilon_1)$} if $f_{1}(\ell_0,i) = (\ell,j)$. We denote:
		$$ (\ell_0,i) \overset{1}{\rightarrow} f_1(\ell_0,i) = (\ell,j).$$
		\item We say that \textbf{the path $e(\ell_1,j_1,\epsilon_1)$ terminates after one iteration} if either  $e(\ell_1,j_1,\epsilon_1)=\emptyset$ or if denoting: 
		$$\mathcal{S}(\ell_1,j_1,\epsilon_1)= ((s_{1j}^{\prime},\epsilon_{1j}^{\prime})_{j \le m_1^\prime},...,(s_{kj}^{\prime},\epsilon_{kj}^{\prime})_{j \le m_k^\prime})$$
		the word of first generation, we have that for all $1 \le \ell \le k$ the word $\mathcal{S}(\ell_1,j_1,\epsilon_1)(\ell) = (s_{\ell j}^{\prime},\epsilon_{\ell j}^{\prime})_{j \le m_\ell^\prime}$ is of trivial trace. It is equivalent to saying that the minimal word of first generation $\mathcal{S}_{min}(\ell_1,j_{1},\epsilon_1)$ is the empty word. We denote by $\mathcal{F}_{1}(\mathcal{S}) \subset \{e(\ell_1,j_1,\epsilon_1);~ (\ell_1,j_1,\epsilon_1) \in [k] \times [m] \times \{+,-\}\}$ the set of paths that terminate after one iteration.
	\end{itemize}
\end{defi}
\begin{ex}\label{ex1}
	We consider $\mathcal{S}=(\mathcal{S}(1),\mathcal{S}(2))$ where $\mathcal{S}(1)=((1,+),(1,+),(2,-),(3,-))$ and  $\mathcal{S}(2)=((3,+),(2,+),(1,-),(1,-))$. We consider all the possible patterns of first generation $(\ell_1,j_1,\epsilon_1) \in \{1,2\}\times [4]\times \{+,-\}$. In this example the patterns $(1,2,+)$, $(2,3,-)$ and $(2,4,-)$ are the patterns that give rise to paths $e(\cdot)$ different from the empty path. In other word the paths $e(1,1,\cdot)$, $e(1,2,-)$, $e(1,3,\cdot)$, $e(1,4,\cdot)$, $e(2,1,\cdot)$, $e(2,2,\cdot)$, $e(2,3,+)$ and $e(2,4,+)$ terminate after one iteration.
	\begin{center}
		\begin{figure}[h]
			\begin{tikzpicture}[node distance={10mm}, scale=0.8, thick, main/.style = {draw, circle}] 
				\node[main] (1) [scale=0.8] {\tiny$\substack{e(\mathrm{P}_1)\\= \emptyset}$}; 
				%\node (1l) [node distance= 15mm, right of=1] {\tiny$p_1=(1,1,+)$}; 
				\node[main] (2) [below of=1, right of=1, scale=0.8] {\tiny$\substack{e(\mathrm{P}_2)\\= \emptyset}$}; 
				\node[main] (3) [below of=2, right of=2,scale=0.8] {\tiny$e(\mathrm{P}_3)$}; 
				\node[main] (4) [below of=3, right of=3,scale=0.8] {$ \emptyset$}; 
				\node[main] (5) [below of=4, right of=4,scale=0.8] {$\emptyset$}; 
				\node[main] (6) [below of=5, right of=5,scale=0.8] {$\emptyset$}; 
				\node[main] (7) [below of=6, right of=6,scale=0.8] {$\emptyset$}; 
				\node[main] (8) [below of=7, right of=7,scale=0.8] {$\emptyset$};
				\node[main] (9) [right of=8,scale=0.8] {$\emptyset$};
				\node[main] (10) [above of=9, right of=9,scale=0.8] {$\emptyset$};
				\node[main] (11) [above of=10, right of=10,scale=0.8] {$\emptyset$};  
				\node[main] (12) [above of=11, right of=11,scale=0.8] {$\emptyset$};
				\node[main] (13) [above of=12, right of=12,scale=0.8] {$\emptyset$};
				\node[main] (14) [above of=13, right of=13,scale=0.8] {\tiny$e(\mathrm{P}_{14})$};
				\node[main] (15) [above of=14, right of=14,scale=0.8] {$\emptyset$};
				\node[main] (16) [above of=15, right of=15,scale=0.8] {\tiny$e(\mathrm{P}_{16})$};
				\node[main] (17)at ($(1)!0.5!(16)$) {\tiny$e_0$};
				\draw[->,dashed](17) -- node[midway, above right, sloped, pos=1] {\tiny$\mathrm{P}_1=(1,1,+)$}(1); 
				\draw[->,dashed](17) -- node[midway, above right, sloped, pos=1] {\tiny$\mathrm{P}_2=(1,1,-)$}(2); 
				\draw[->](17) -- node[midway, above right, sloped, pos=1] {\tiny$\mathrm{P}_3=(1,2,+)$}(3); 
				\draw[->,dashed](17) -- node[midway, above right, sloped, pos=1] {\tiny$\mathrm{P}_4=(1,2,-)$}(4);
				\draw[->,dashed](17) -- node[midway, above right, sloped, pos=1] {\tiny$\mathrm{P}_5=(1,3,+)$}(5);
				\draw[->,dashed](17) -- node[midway, above right, sloped, pos=1] {\tiny$\mathrm{P}_6=(1,3,-)$}(6);
				\draw[->,dashed](17) -- node[midway, above right, sloped, pos=1] {\tiny$\mathrm{P}_7=(1,4,+)$}(7);
				\draw[->,dashed](17) -- node[midway, above right, sloped, pos=1] {\tiny$\mathrm{P}_8=(1,4,-)$}(8);
				\draw[->,dashed](17) -- node[midway, below left, sloped, pos=1] {\tiny$\mathrm{P}_9=(2,1,+)$}(9);
				\draw[->,dashed](17) -- node[midway, below left, sloped, pos=1] {\tiny$\mathrm{P}_{10}=(2,1,-)$}(10);
				\draw[->,dashed](17) -- node[midway, below left, sloped, pos=1] {\tiny$\mathrm{P}_{11}=(2,2,+)$}(11);
				\draw[->,dashed](17) -- node[midway, below left, sloped, pos=1] {\tiny$\mathrm{P}_{12}=(2,2,-)$}(12);
				\draw[->,dashed](17) -- node[midway, below left, sloped, pos=1] {\tiny$\mathrm{P}_{13}=(2,3,+)$}(13);
				\draw[->](17) -- node[midway, below left, sloped, pos=1] {\tiny$\mathrm{P}_{14}=(2,3,-)$}(14);
				\draw[->,dashed](17) -- node[midway, below left, sloped, pos=1] {\tiny$\mathrm{P}_{15}=(2,4,+)$}(15);
				\draw[->](17) -- node[midway, below left, sloped, pos=1] {\tiny$\mathrm{P}_{16}=(2,4,-)$}(16);
			\end{tikzpicture} 
			\caption{Result of the algorithm after one iteration of Schwinger-Dyson equation}
			\label{Fig1}
		\end{figure}
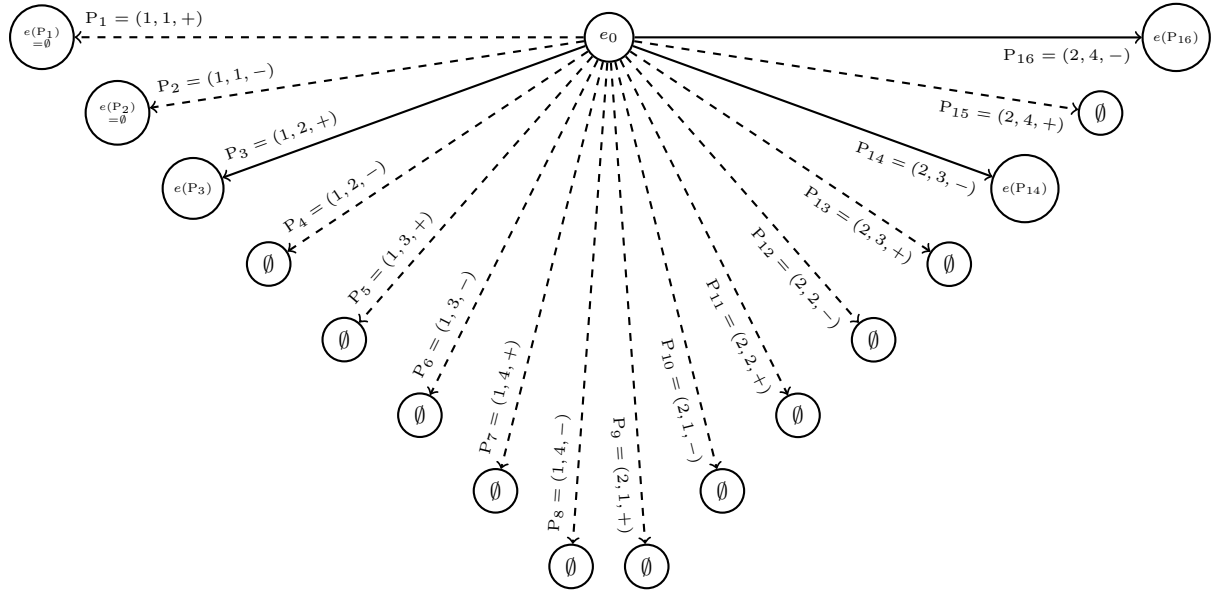
	\end{center}
	
	With the algorithm we are describing, we obtain a rooted tree structure that depends only on the original minimal word $\mathcal{S} = ((s_{\ell j},\epsilon_{\ell j})_{j\in [m_{\ell}]})_{\ell\in [k]}$. In Figure \ref{Fig1}, the rooted tree in question is the tree obtained by removing the dashed edges (see Figure \ref{Fig2}). In this example, the root $e_0$ is of degree $3$. In any case, for a given $m=\sum_{\ell}m_{\ell}$, and under the condition of having something other than an empty word, all nodes will have a degree bounded by $m$.
	\begin{figure}[h]
		\begin{center}
			\begin{tikzpicture}[node distance={10mm}, scale=0.8, thick, main/.style = {draw, circle}] 
				\node[main] (17)at ($(0,0)$) {\tiny$e_0$};	
				\node[main] (3) at ($(-4,-3)$) [scale=0.8] {\tiny$e(p_3)$}; 
				\node[main] (14) at ($(0,-5)$) [scale=0.8] {\tiny$e(p_{14})$};
				\node[main] (16) at ($(4,-3)$)[scale=0.8] {\tiny$e(p_{16})$};
				\draw[->](17) -- node[midway, above right, sloped, pos=1] {\tiny$p_3=(1,2,+)$}(3); 
				\draw[->](17) -- node[midway, below left, sloped, pos=1] {\tiny$p_{14}=(2,3,-)$}(14);
				\draw[->](17) -- node[midway, below left, sloped, pos=1] {\tiny$p_{16}=(2,4,-)$}(16);
			\end{tikzpicture}
			\caption{Tree structure for the first generation}
			\label{Fig2}
		\end{center}
	\end{figure}
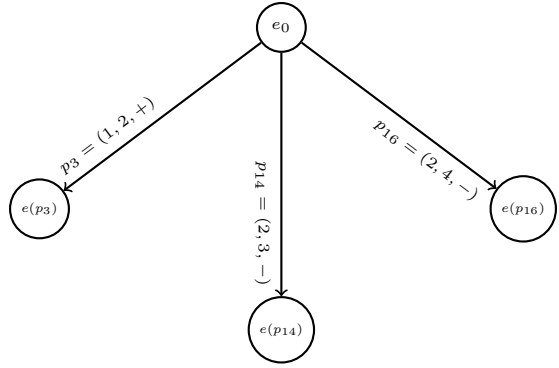
\end{ex}
\begin{prop}\label{init}
	Let $\mathcal{S} = ((s_{\ell j},\epsilon_{\ell j})_{j\in [m_{\ell}]})_{\ell\in [k]}$ be a fixed initial $k$-word. For all choices of $(\ell_1,j_{1},\epsilon_1)$, following the previous construction we have that:
	$$\sum_{\ell=1}^{k^\prime} m_{\ell}^\prime \le m,$$
	where $k^\prime \ge 0$ is such that $\mathcal{S}_{min}(\ell_1,j_{1},\epsilon_{1})$ is a $k^\prime$-word and for all $1 \le \ell \le k^{\prime}$, $m^{\prime}_{\ell}$ is the length of the $\ell$-th word in $\mathcal{S}_{min}(\ell_1,j_{1},\epsilon_{1})$.
	Also in regards of Schwinger-Dyson equation \eqref{SD} we have:
	$$\mathbb{E}^\prime(\mathcal{S}) = \sum_{\ell_1=1}^{k}\sum_{j_1 = 1}^{m_{\ell}}\sum_{\epsilon_1 \in \{+,-\}} \bar{e}(\ell_1,j_1,\epsilon_1).$$
\end{prop} 

\begin{ex} We continue with Example \ref{ex1}. Applying Schwinger-Dyson equation to 
	$$\mathbb{E}^\prime(\mathcal{S}):=\mathbb{E}(\Tr[U(1)^2U(2)^\ast U(3)^{\ast}]\Tr[U(3) U(2){U(1)^\ast}^2])$$
	we obtain indeed:
	\begin{align}
		\mathbb{E}^\prime(\mathcal{S})= &-\frac{1}{N}\mathbb{E}(\Tr[U(1)]\Tr[U(1)U(2)^\ast U(3)^{\ast}]\Tr[U(3) U(2){U(1)^\ast}^2]) \label{exl1}\\
		&~~~~~~~~~~~~~~~~~~~~~~~~~~~~~+\frac{1}{N}\mathbb{E}(\Tr[U(1)U(2)^\ast U(3)^{\ast}U(1)^{\ast}U(3)U(2)]) \label{exl2}\\
		&~~~~~~~~~~~~~~~~~~~~~~~~~~~~~~~~~~~~~+\frac{1}{N}\mathbb{E}(\Tr[U(1)U(2)^\ast U(3)^{\ast}U(3)U(2)U(1)^\ast]).\label{exl3}
	\end{align}
	In fact, following the previous algorithm, we obtain that the term \eqref{exl1} corresponds to the term $\bar{e}(1,2,+)$, the term \eqref{exl2} to $\bar{e}(2,3,-)$ and finally the term \eqref{exl3} corresponds to $\bar{e}(2,4,-)$.
\end{ex}
We now consider $n \ge 1$ and a \textbf{pattern} $\mathrm{P}=(\ell_b,j_b,\epsilon_b)_{b\in[n+1]} \in (\NN \times \NN \times \{+,-\})^{n+1}$ of length $n+1$. We suppose that the data of $\mathcal{S} = ((s_{\ell j},\epsilon_{\ell j})_{j\in [m_{\ell}]})_{\ell\in [k]}$ and $(\ell_b,j_b,\epsilon_b)_{b\in [n]}$ gave us the construction of $e(\cdot)$, $\mathcal{S}(\cdot)$, $e_{rs}(\cdot)$, $f_b$, $\mathcal{S}_{min}(\cdot)$, $e_{min}(\cdot)$, $f^{rs}_b$, $\epsilon^{(b)}$, $p_b$ and $\bar{e}(\cdot)$ of $(\ell_{b^\prime},j_{b^\prime},\epsilon_{b^\prime})_{b^\prime \in [b]}$ for all $1 \le b \le n$. To ease notation in this second part of the algorithm, the pattern being fixed, for all $1\le b \le n+1$, we set $e_{b} = e((\ell_{b^\prime},j_{b^\prime},\epsilon_{b^\prime})_{b^\prime \in [b]})$. Finally we suppose that for all $1 \le b \le n$, denoting:

$$\mathcal{S}_{min}((\ell_{b^\prime},j_{b^\prime},\epsilon_{b^\prime})_{b^\prime\in [b]}) = ((s_{1j}^{(b)},\epsilon_{1j}^{(b)})_{j \in [m^{(b)}_1]}, ..., (s_{k_bj}^{(b)},\epsilon_{k_bj}^{(b)})_{j \in [m^{(b)}_{k_b}]})$$
the minimal writing of the word of generation $b$, we have that:
$$\sum_{\ell=1}^{k_b} m_{\ell}^{(b)} \le m. $$
In particular if that is true for all patterns of length $1 \le b \le n$ then one only has to consider $(\ell_b,j_b,\epsilon_b)_{b\in[n]} \in ([m]\times [m] \times \{+,-\})^{n}$. Under these assumptions we construct $e_{n+1} = e(\mathrm{P})$. If for some $1\le b \le n$, $e_{b}$ terminates after $b$ iterations we set $e_{n+1}=\emptyset$. Otherwise, we proceed as with the first generation. We now consider:
\begin{equation}\label{minngene}
	\mathcal{S}_{min}((\ell_b,j_b,\epsilon_b)_{b\in [n]})= ((s_{1j}^{(n)},\epsilon_{1j}^{(n)})_{j \le m_1^{(n)}},...,(s_{kj}^{(n)},\epsilon_{kj}^{(n)})_{j \le m_{k_n}^{(n)}})
\end{equation}
the minimal word of the $n$-th generation. We consider $\epsilon^{(n)} \in \{-1,+1\}$ and $p_{n} \in \mathbb{N}$ such that:
\begin{equation*}
	\bar{e}_n = \epsilon^{(n)}\frac{N^{p_n}}{N^n} \mathbb{E}^\prime(\mathcal{S}_{min}((\ell_b,j_b,\epsilon_b)_{b\in [n]})).
\end{equation*}
We recall that $\epsilon^{(n)}$ and $p_n$ depend only on the pattern $(\ell_b,j_b,\epsilon_b)_{b\in [n]}$ and $\mathcal{S}$ the original word. Once again if $\ell_{n+1}=1$ and $j_{n+1}=1$, or $\ell_{n+1} \ge k_n+1$, or $j_{n+1} \ge m_{\ell_{n+1}}+1$, or $s_{11}^{(n)} \neq s_{\ell_{n+1} j_{n+1}}^{(n)}$ we set $e_{n+1} =\emptyset$ and $\bar{e}_{n+1} =0$. We now consider the case where $\ell_{n+1} \ge 2$, $\epsilon_{n+1}= +$ and $\epsilon_{11}^{(n)}=\epsilon_{n+1} \epsilon_{\ell_{n+1} j_{n+1}}^{(n)}= \epsilon_{\ell_{n+1} j_{n+1}}^{(n)}$, i.e when one considers a term of the third line of the Schwinger-Dyson equation \eqref{SD3}. In this case we have: 

\begin{align*}
	\epsilon^{(n+1)} &= - \epsilon^{(n)} \\
	\text{and}~~~e_{n+1} &:= (((1,1),...,(1,m_1^{(n)}),(\ell_{n+1},j_{n+1}),(\ell_{n+1},j_{n+1}+1),...,(\ell_{n+1},m_{\ell_{n+1}}^{(n)}),(\ell_{n+1},1),...,(\ell_{n+1},j_{n+1}-1))\\	
	&(2,j)_{j\le m_{2}^{(n)}},...,(\ell_{n+1}-1,j)_{j \le m_{\ell_{n+1}-1}^{(n)}},(\ell_{n+1}+1,j)_{j\le m_{\ell_{n+1}+1}^{(n)}},...,(k_n,j)_{j\le m_{k_n}^{(n)}}) 
\end{align*}
where all the notations refer to the notation chosen to explicit the minimal word of generation $n$ given by \eqref{minngene}. We also construct the word of the $(n+1)$-th generation $\mathcal{S}_{n+1}=\mathcal{S}(\mathrm{P})$ as the $(k_n -1)$-word that verifies:
\begin{align*}
	\mathcal{S}_{n+1}(1)= &((s_{11}^{(n)},\epsilon_{11}^{(n)}),...,(s_{1 m_1^{(n)}}^{(n)},\epsilon_{1 m_1^{(n)}}^{(n)}),(s_{\ell_{n+1} j_{n+1}}^{(n)},\epsilon_{n+1}),(s_{\ell_{n+1} j_{n+1}},\epsilon_{n+1}),(s_{\ell_{n+1} j_{n+1}+1},\epsilon_{\ell_{n+1} j_{n+1}+1}),...,\\
	& ~~~~~~~~~~~~~~~~~~~~~~~~~~...,(s_{\ell_{n+1} j_{n+1}-1},\epsilon_{\ell_{n+1} j_{n+1}-1})),\\
	\mathcal{S}_{n+1}(\ell)=&(s_{\ell j},\epsilon_{\ell j})_{j\le m_\ell^{(n)} },~~~ \text{for}~~ 2\le \ell \le \ell_{n+1}-1\\
	\mathcal{S}_{n+1}(\ell-1)= &(s_{\ell j},\epsilon_{\ell j})_{j\le m_\ell^{(n)} },~~~  \text{for}~~ \ell_{n+1}+1\le \ell \le k_n .
\end{align*}
Finally, we have the \textbf{$n+1$-th tracking function}:
\begin{align}\label{trackingfunction}
	f_{n+1} : &\mathbb{N} \times \mathbb{N}&&\rightarrow \mathbb{N} \times \mathbb{N}\notag \\
	&(\ell,j) &&\mapsto (\ell,j) ~~\text{for all}~~ 1 \le \ell \le \ell_{n+1}-1.\notag \\
	& (\ell_{n+1},j) &&\mapsto (1,m_1+1+j-j_{1})\notag \\
	&(\ell,j) &&\mapsto (\ell-1,j) ~~\text{for all}~~ \ell_{n+1}+1 \le \ell \le k_n.\notag \\
\end{align}
We now may recover the corresponding word $\mathcal{S}((\ell_p,j_p,\epsilon_p)_{p\in[n+1]})$, its corresponding minimal writing, the rescaled path $e_{rs}$, the minimal path $e_{min}$ and the rescaled tracking function $f_{n+1}^{rs}$. We also consider $t_{n+1}$ the number of $1 \le \ell \le k_n-1$ such that $\mathcal{S}_{n+1}(\ell)$ is of trivial trace and we set $p_{n+1}=p_n + t_{n+1}$. We have now recovered the mean of the path:
$$\bar{e}((\ell_p,j_p,\epsilon_p)_{p\in[n+1]})= \epsilon^{(n+1)}\frac{N^{p_{n+1}}}{N^{n+1}} \mathbb{E}^\prime(\mathcal{S}_{min}((\ell_p,j_p,\epsilon_p)_{p\in[n+1]})).$$
We proceed similarly for the other cases of $(\ell_{n+1},j_{n+1},\epsilon_{n+1})$. We may now introduce the following Definition of terminating terms after $n$ iterations of Schwinger-Dyson equation.
\begin{defi}
	We say that \textbf{the path $e((\ell_b,j_b,\epsilon_b)_{b\in [n]})$ terminates after $n$ iterations} if either  $e((\ell_b,j_b,\epsilon_b)_{b\in [n]})=\emptyset$ or if denoting:
	$$\mathcal{S}((\ell_b,j_b,\epsilon_b)_{b\in [n]})= ((s_{1j}^{(n)},\epsilon_{1j}^{(n)})_{j \in[m_1^{(n)}]},...,(s_{kj}^{(n)},\epsilon_{k_n j}^{(n)})_{j \in [m_{k_n}^{(n)}]})$$
	the word of generation $n$, we have that for all $1 \le \ell \le k_n$, the $1$-word $\mathcal{S}(\ell) = (s_{\ell j}^{(n)},\epsilon_{\ell j}^{(n)})_{j \le m_\ell^{(n)}}$ is of trivial trace. We denote by $\mathcal{F}_{n}(\mathcal{S})  \subset \{e((\ell_b,j_b,\epsilon_b)_{b\in [n]}) \in (\mathbb{N}\times\mathbb{N} \times \{+,-\})^n\}$ the set of paths that terminate after $n$ iteration.
\end{defi}
\begin{prop}\label{beforecvseries}
	We consider $\mathcal{S}$ an initial fixed minimal $k$-word. Let $n \ge 1$ be an integer and $\mathrm{P} =(\ell_b,j_b,\epsilon_b)_{b\in[n]} \in ([m] \times [m] \times \{+,-\})^{n}$ a pattern. Using the notations of the previous algorithm applied to $\mathcal{S}$ and $\mathrm{P}$ we have:
	$$\sum_{\ell =1}^{k_n} m_{\ell}^{(n)} \le m$$
	where we used the notations given for the minimal writing of the $n$-th generation word obtained by the algorithm \eqref{minngene}. Also we have:
	\begin{equation}\label{nfixed}
		\mathbb{E}^\prime(\mathcal{S}) = \sum_{b=1}^{n-1}~~~ \sum_{e \in \mathcal{F}_{b}(\mathcal{S})} \bar{e} +~~~ \sum_{\mathrm{P}\in ([m] \times [m] \times \{+,-\})^{n}} \bar{e}(\mathrm{P}).
	\end{equation}
	Finally the number of paths $e(\mathrm{P})$ with $\mathrm{P}\in ([m] \times [m] \times \{+,-\})^{n}$ that did not terminate before the $n$-th iteration is bounded by $(m-1)^{n}$.
\end{prop}
\begin{proof}
	For all $n\ge 1$ and all patterns $\mathrm{P}=(\ell_b,j_b,\epsilon_b)_{b\in [n]}$, the construction of $\mathcal{S}(\mathrm{P})$ is such that the total number of matrices appearing in $\mathcal{S}(\mathrm{P})$ is exactly given by the total number of matrices in $\mathcal{S}_{min}((\ell_b,j_b,\epsilon_b)_{b\in [n-1]})$. Thus the total number of matrices after simplification, i.e. after replacing the sub-words by their minimal writing, is bounded by the original number of matrices. The second statement is proved by induction on $n\ge 1$. In fact, the initiation is given by Proposition \ref{init}. If we now assume that the formula \eqref{nfixed} is verified for some $n\ge 1$. We apply Schwinger-Dyson equation \eqref{SD} to each term $\bar{e}((\ell_b,j_b,\epsilon_{b})_{b\in[n]})$ and obtain the same formula for $n+1$. To prove the last statement, recall that we denote by 
	$$\mathcal{S}_{min}((\ell_{b^{\prime}},j_{b^\prime},\epsilon_{b^\prime})_{b^\prime\in [b]}) = ((s_{1j}^{(b)},\epsilon_{1j}^{(b)})_{j \in [m^{(b)}_1]}, ..., (s_{k_bj}^{(b)},\epsilon_{k_bj}^{(b)})_{j \in [m^{(b)}_{k_b}]})$$
	the corresponding minimal word of generation $b$. To have a non-empty path of $(b+1)$-th generation one has to consider $(\ell_{b+1},j_{b+1},\epsilon_{b+1})$ such that $(\ell_{b+1},j_{b+1}) \neq (1,1)$ and $s_{11}^{(b)}= s_{\ell_{b+1},j_{b+1}}^{(b)}$. If we are in this case and we have $\epsilon_{b+1}=+$, then the corresponding path $e_{b+1}$ of generation $b+1$ is non-empty if and only if $\epsilon_{\ell_{b+1} j_{b+1}}^{(b)} = \epsilon_{11}^{(b)}$. If we have $\epsilon_{b+1}=-$, the corresponding path of generation $b+1$ is non-empty if and only if $\epsilon_{\ell_{b+1} j_{b+1}}^{(b)} = -\epsilon_{11}^{(b)}$. Therefore for all $(\ell_{b^\prime},j_{b^\prime},\epsilon_{b^\prime})_{b^\prime\in [b+1]}$, if $e(\ell_1,j_1,\epsilon_1,...,\ell_{b+1},j_{b+1},\epsilon_{b+1}) \neq \emptyset$ then $e(\ell_1,j_1,\epsilon_1,...,\ell_{b+1},j_{b+1},-\epsilon_{b+1})=\emptyset$. Therefore for $(\ell_{b^\prime},j_{b^\prime},\epsilon_{b^\prime})_{b^\prime \in [b]}$ fixed we have:
	\begin{align*}
		|\{(\ell_{b+1},j_{b+1},\epsilon_{b+1})\in [m]\times[m]\times\{+,-\};~e((\ell_{b^\prime},j_{b^\prime},\epsilon_{b^\prime})_{b^\prime\in[b+1]})\neq \emptyset\}| &\le \sum_{\ell,j} \mathds{1}(s_{11}^{(b)}=s_{\ell j}^{(b)},~(1,1)\neq (\ell,j)) \\
		& \le m-1.
	\end{align*}
	From this last inequality we obtain the proposition's last assertion by induction on $1 \le b \le n$.
\end{proof}

\subsection{Writing the iterations of Schwinger-Dyson and convergence of series}\label{writ_iterations}
With the previous proposition, and for some values of $m\ge 1$, we can now express the expectation of the traces as a series of terminated terms given by the algorithm.

\begin{prop}\label{cvseries}
	Let:
	$$\mathcal{S}= ((s_{1j},\epsilon_{1j})_{j \in [m_1]},...,(s_{kj},\epsilon_{kj})_{j\in [m_k]})$$
	be an initial minimal $k$-word such that $m := \sum_{\ell} m_{\ell} < N^{2/3}+1$. We apply the previous algorithm and denote by $\mathcal{F}_n(\mathcal{S})$ the set of finishing path after $n$ iterations. We have that:
	$$\sum_{n} \sum_{ e \in \mathcal{F}_n(\mathcal{S})} \bar{e}$$ 
	is convergent and furthermore considering the expectation given by \eqref{Etr} we have for $N$ large enough:
	$$\mathbb{E}^\prime(\mathcal{S})= \sum_{n=0}^\infty \sum_{ e \in \mathcal{F}_n(\mathcal{S})} \bar{e}.$$ 
\end{prop} 

The main argument to prove the previous proposition is the following lemma.

\begin{lem}\label{lemcv}
	For an initial minimal $k$-word denoted:
	$$\mathcal{S}= ((s_{1j},\epsilon_{1j})_{j \in [m_1]},...,(s_{kj},\epsilon_{kj})_{j \in [m_k]}),$$
	using the same notations than in Proposition \ref{cvseries}, for all $n \ge 1$ integer we have:
	%For a given $(s_1,...,s_{2m}) \in [d]^{2m}$, for all $n \in \mathbb{N}$ and all $e\in \mathcal{F}_n$ we have:
	\begin{align*}
		|\mathcal{F}_{n}(\mathcal{S})| &\le (m-1)^n\\
		\forall e \in \mathcal{F}_n(\mathcal{S})~~~~|\bar{e}| &\le N^{k -2/3 n}.
	\end{align*}
	
\end{lem}
The proof of this lemma follows the proof given by Hastings \cite[Part C.]{hastings2007random}.
\begin{proof}
	Let $e=e((\ell_b,j_b,\epsilon_b)_{b\in[n]}) \in \mathcal{F}_{n} (\mathcal{S})$ such that $\bar{e}\neq 0$. Recall that for all $1 \le b \le n$ we denote by $k_{b}$ the integer such that the minimal writing $\mathcal{S}_{min}((\ell_{b^\prime},j_{b^\prime},\epsilon_{b^\prime})_{{b^\prime}\in[b]})$ of generation $b$ is a $k_{b}$-word and $p_b \in \mathbb{N}$ such that:
	$$|\bar{e}((\ell_{b^\prime},j_{b^\prime},\epsilon_{b^\prime})_{b^\prime\in[b]})| = \frac{N^{p_{b}}}{N^b}|\mathbb{E}^\prime(\mathcal{S}_{min}((\ell_{b}^{\prime},j_{b^\prime},\epsilon_{b^\prime})_{b^\prime\in[b]}))|$$
	where $\mathbb{E}^\prime$ is defined by Equation \eqref{Etr}. Although we set $p= p_n$. We denote by $q := |\{b \le n; ~~ \ell_b =1\}|$ the number of times we have considered a path that gives a term in the first \eqref{SD1} or second \eqref{SD2} line of the Schwinger-Dyson equations. For all $1 \le b \le n$ such that $\ell_b =1 $ we necessarily have $k_{b}-k_{b-1} = 1$. Indeed, since we assume at each step that we are applying our algorithm to minimal words that are not empty, cutting a word in half cannot produce a word of trivial trace and therefore the minimal word of the $b$-th generation is a $(k_{b-1}+1)$-word. Also the integers $1 \le b \le n$ such that $k_{b}-k_{b-1} = -2$ are exactly the terms where $p_b = p_{b-1}+1$. Finally, $e$ being a finishing term, we have $k_n =0$, $k_{0}=k$ and therefore:
	\begin{align*}
		k_n = 0 &= k_{0} + \sum_{b=1}^n k_{b}-k_{b-1}\\
		&= k + \sum_{b=1}^n \mathds{1}(\ell_b =1) - \sum_{b=1}^n (\mathds{1}(\ell_b \neq 1) + \mathds{1}(p_b = p_{b-1}+1))\\
		&= k + q - (n-q) - p.
	\end{align*}
	The minimality of the words to which we apply the Schwinger-Dyson equation also implies that the words of trivial traces come from the iterations $\{1 \le b \le n;  ~~ \ell_b \neq 1\}$, i.e. when we consider terms from the third \eqref{SD3} or fourth \eqref{SD4} line. In particular we have:
	$$p\le n-q$$
	and therefore:
	$$p \le \frac{k+n}{3}.$$
	Finally, the argument for the bound on $|\mathcal{F}_{n}(\mathcal{S})|$ is exactly the one given to prove the last assertion of Proposition \ref{beforecvseries}, noting that $\mathcal{F}_{n}$ is included in the number of non-empty terms.
\end{proof}
We now give the proof of the previous proposition.
\begin{proof}[Proof of Proposition \ref{cvseries}]  
	The convergence of the series is given by the fact that we have:
	$$\left|\sum_{e \in \mathcal{F}_n(\mathcal{S})} \bar{e} \right| \le N^k \left( \frac{m-1}{N^{2/3}}\right)^n$$
	where we have $\frac{m-1}{N^{2/3}} <1$. For $n\ge 1$ we denote:
	$$r_n := \sum_{\mathrm{P}\in ([m]\times[m]\times\{+,-\})^n} \bar{e}(\mathrm{P}).$$
	We notice that for all iterations $n \ge 1$ and for all patterns $\mathrm{P}$ of length $n$, using the notation of the algorithm given in Section \ref{encodingmat}  we have that $\sum_{\ell=1}^{k_n}m^{(n)}_{\ell} \le m$ and for all $1 \le \ell \le k_n$ we have:
	$$|\Tr[U(s_{\ell 1}^{(n)})^{\epsilon_{\ell 1}^{(n)}} \cdots U(s_{\ell m^{(n)}_{\ell}}^{(n)})^{\epsilon_{\ell m^{(n)}_{\ell}}^{(n)}}]| \le N$$ 
	and therefore: 
	\begin{align*}
		|r_n| &\le \sum_{\mathrm{P} \in ([m]\times[m]\times\{+,-\})^n} |\bar{e}(\mathrm{P})|\\
		&\le (m-1)^n \frac{N^{m}}{N^n}.
	\end{align*}
	For $N$ large enough we have $(m-1)/N < 1$ and therefore $|r_n|\underset{n \rightarrow \infty}{\longrightarrow} 0$. To conclude we apply Proposition \ref{beforecvseries} and we have:
	\begin{equation*}
		|r_n|=\left|\mathbb{E}_0(\mathcal{S}) - \sum_{b=1}^{n-1}~~~ \sum_{e \in \mathcal{F}_{b}(\mathcal{S})} \bar{e}\right|.
	\end{equation*}
\end{proof}

\section{Computations of traces}
Throughout this section we consider $\mathcal{S} = ((s_{1},\epsilon_1),...,(s_{m},\epsilon_{m}))$ a minimal $1$-word of length $m$ and we want to compute:
\begin{equation*}
	\mathbb{E}^{\prime}(\tilde{\mathcal{S}})= \mathbb{E}_0 (\mathcal{S}) = \mathbb{E}(\Tr[\mathcal{U}(\mathcal{S})]\Tr[\mathcal{U}(\mathcal{S})^\ast])
\end{equation*}
where we set $\tilde{\mathcal{S}}=(\tilde{\mathcal{S}}(1),\tilde{\mathcal{S}}(2))$ with $\tilde{\mathcal{S}}(1)= \mathcal{S}$ and $\tilde{\mathcal{S}}(2) = ((s_{m},-\epsilon_m),...,(s_1,-\epsilon_1))$ (see Equations \eqref{Etr} and \eqref{defEtrm} for definitions of $\mathbb{E}^\prime(\cdot)$ and $\mathbb{E}_0(\cdot)$). Since the previous definition of $\tilde{\mathcal{S}}$ depends only on $\mathcal{S}$ we abuse notations and use the notation $\mathcal{S}$ to refer to the $2$-word $\tilde{\mathcal{S}}$.
\begin{defi}\label{defposition}
	We fix $\ell_{0} \in \{1,2\}$, $i \in [m]$ and a pattern $\mathrm{P}=(\ell_b,j_b,\epsilon_b)_{b\in [n]} \in ([m]\times [m]\times \{+,-\})^n$. 
	We say that \textbf{the matrix in position $(\ell_0,i)$ moves to position $(\ell,j)$ in the path $e(\mathrm{P})$ at the $n$-th iteration} if $f_{n}\circ f_{n-1}^{rs} \circ \cdots \circ f_{1}^{rs} (\ell_0,i) = (\ell,j)$. We denote:
	$$ (1,m+1-i) \overset{1}{\rightarrow} (\ell_1(i),j_1(i))\overset{2}{\rightarrow} (\ell_2(i),j_2(i))\overset{3}{\rightarrow} \cdots\overset{n}{\rightarrow} f_n(\ell_{n-1}(i),j_{n-1}(i))=(\ell,j),$$
	where for all $1 \le b \le n-1$ we set $(\ell_{b}(i),j_{b}(i))= f^{rs}_{b} (\ell_{b-1}(i),j_{b-1}(i))$ with convention $(\ell_0(i),j_{0}(i)):= (1,m+1-i)$. Similarly we denote:
	$$ (2,i) \overset{1}{\rightarrow} (\ell_1^\prime(i),j_1^\prime (i))\overset{2}{\rightarrow} (\ell_2^\prime (i),j_2^\prime (i))\overset{3}{\rightarrow} \cdots\overset{n}{\rightarrow} f_n(\ell_{n-1}^\prime(i),j_{n-1}^\prime(i))$$
	where for all $1 \le b\le n-1$, $(\ell_{b}^\prime(i),j_{b}^\prime(i))= f^{rs}_{b} (\ell_{b-1}^\prime(i),j_{b-1}^\prime(i))$ with convention $(\ell_0^\prime(i),j_{0}^\prime(i)):= (2,i)$. 
	
	%\begin{rem}
	%	In the previous definition we take as convention that if we had $\ell_{n-1} = 1$ and for instance $\epsilon_{n-1} = \epsilon_{11}$ and - considering the notations of the simplified word of generation $n-1$, $s_{11}=s_{1j_{n-1}}$, then:
	%	$$U(s_{21})^{\epsilon_{21}} U(s_{22})^{\epsilon_{22}} \cdots  U(s_{2m_{2}})^{\epsilon_{2m_2}} = U(s_{1j_{n-1}})^{\epsilon_{1j_{n-1}}} U(s_{1 j_{n-1}+1})^{\epsilon_{1 j_{n-1}+1}} \cdots  U(s_{1m_{1}})^{\epsilon_{1m_1}}. $$
	%\end{rem}
	
\end{defi} 
\subsection{Rung cancellations of matrices}
In this section we analyze the terminating non-zero terms of generation $n$ verifying that the matrix originally at position $(\ell_0,i)$ was never used before the cancellation, and is then canceled against its original inverse in the second trace (i.e., the matrix originally at position $(\bar{\ell}_{0}, m+1-i)$ with $\bar{\ell}_0 \neq \ell_0$).
\begin{defi}\label{trivandrung}
	Let $i \in [m]$ and $e=e((\ell_b,j_b,\epsilon_b)_{b\in[n+1]}) \neq \emptyset$ a path. We say that the matrix $i$ \textbf{is trivially moved under the $(n+1)$-th iteration} in this path if using the notations of Definition \ref{mouv&term} we have:
	$$f_{n}^{rs}(\ell_{n-1}(i),j_{n-1}(i)) \notin \{(\ell_{n+1},j_{n+1}),(1,1)\} ~~\text{and}~~f_{n}^{rs}(\ell_{n-1}^\prime(i),j_{n-1}^\prime(i)) \notin \{(\ell_{n+1},j_{n+1}),(1,1)\}.$$
	If we suppose $e \in \mathcal{F}_{n+1}(\mathcal{S})$, we say that we have \textbf{a rung cancellation of matrix $i$} if there exists $n_i \le n+1$ such that for all $b \le n_i$, the matrix $i$ is trivially moved under the $b$-th iteration and if $f_{n_i}(\ell_{n_{i}-1}^\prime (i),j_{n_{i}-1}^\prime(i)) =f_{n_i}(\ell_{n_{i}-1}(i),j_{n_{i}-1}(i))\pm 1$ (where we take it as conventions that  $(\ell^\prime,j)\pm 1 =(\ell^\prime,j^\prime \pm 1)$, $(\ell,m_{\ell}+1)=(\ell,1)$ and $(\ell,1 -1) = (\ell,m_{\ell})$).
\end{defi}

\begin{rem}
	One can interpret the previous definition as follows: the matrix $i$ is trivially moved under the $n$-th iteration if the term we consider after applying Schwinger-Dyson equation does not originate from a term of the form $T U(s_{\ell i})$ (see \cite{hastings2007random} and proof of Proposition \ref{SDPROP}). We have a rung cancellation of the matrix $i$ if in all the previous paths we never crossed the matrices originally at positions $(1,2m+1-i)$ and $(2,i)$ and at some point they canceled each other.
\end{rem}

We will now describe the terms finishing at the $n$-th iteration and having a rung cancellation of matrix $i$ for any $ i \in [m]$. Let $X = U(d+1)$ Haar distributed and independent of $(U(s))_{s \in[d]}$. For $1\le i \le m$ we define:
\begin{align*}
	\mathbb{E}_{0}^{i}(\mathcal{S}) &:= \mathbb{E}(\Tr[U(s_{m})^{\epsilon_{m}}\cdots U(s_{i+1})^{\epsilon_{i+1}} XU(s_{i-1})^{\epsilon_{i-1}} \cdots U(s_1)^{\epsilon_1}] \times\\
	&~~~~~~~~~~~~~~~~~~~~~~~~~~~~~~~~~~~~~~~~~~~~\Tr[U(s_1)^{-\epsilon_1}\cdots U(s_{i-1})^{-\epsilon_{i-1}} X^\ast U(s_{i+1})^{-\epsilon_{i+1}}\cdots U(s_{2m})^{\ast}]).
\end{align*}
We then set:
$$e_{0}^i = e_{0}$$
with $e_{0}$ as defined in the beginning of Section \ref{encodingmat}. We apply the previous algorithm to $e^{i}_{0}$ considering as fixed initial word:
$$\mathcal{S}^{i} := ((s_1,\epsilon_1),...,(s_{i-1},\epsilon_{i-1}),(d+1,+), (s_{i+1},\epsilon_{i+1}),...,(s_{m},\epsilon_{m})).$$
In particular we have $\mathcal{S}^{i}$ also minimal.
\begin{prop}\label{rungcanc}
	Let $\mathrm{P}=(\ell_b,\ell_b,\epsilon_b)_{b\in[n]} \in (\mathbb{N}\times \mathbb{N} \times \{+,-\})^n$ be a pattern. If the path $e:=e((\ell_b,\ell_b,\epsilon_b)_{b\in[n]} )$ terminates at the $n$-th iteration and has a rung cancellation of matrix $i$ then $e^i := e^i((\ell_b,\ell_b,\epsilon_b)_{b\in[n]} )=e((\ell_b,\ell_b,\epsilon_b)_{b\in[n]})$ and their mean are equal. Conversely if $e^i((\ell_b,\ell_b,\epsilon_b)_{b\in[n]} )$ is a term that terminates at the $n$-th iteration when applying the algorithm to $\mathcal{S}^i$, then $e$ terminates at the $n$-th iteration, has a rung cancellation of matrix $i$ and has same mean than $e^i$. In particular we have for all $n\in \mathbb{N}$:
	$$\mathcal{F}_{n}(\mathcal{S}^i)  \subset \mathcal{F}_{n}(\mathcal{S}).$$
\end{prop}
We set for all $n\ge 1$:
\begin{equation*}
	\mathcal{R}_n(\mathcal{S}) := \mathcal{F}_n(\mathcal{S})\backslash \bigcup_{i\in [m]} \mathcal{F}_n(\mathcal{S}^i).
\end{equation*}

\begin{defi}
	Let $1 \le i_1 < \cdots < i_q \le m$ be elements of $[m]$ and $e=e((\ell_b,\ell_b,\epsilon_b)_{b\in[n]} ) \neq \emptyset$ be a path that terminates at the $n$-th iteration. We say that there is a \textbf{rung cancellation of matrices $\vec{i}=(i_1,...,i_q)$} if for all $1 \le t \le q$ there is a rung cancellation of matrix $i_t$.
\end{defi}
We fix $1 \le i_1 < \cdots < i_q \le m$ and we generalize the previous description of the terms with a rung cancellation of matrices $\vec{i}= (i_1, ..., i_q)$. We consider $(X_1,...,X_{q})= (U(d+1),...,U(d+q))$ such that $(U(s))_{s \in [d+q]}$ are Haar distributed independent unitaries. We denote:
\begin{align} 
	\mathbb{E}_{0}^{\vec{i}} (\mathcal{S}) = \mathbb{E}(\Tr[U(s_{m})^{\epsilon_{m}}& \cdots U(s_{i_{t}+1})^{\epsilon_{i_t +1}} X_{t} U(s_{i_{t}-1})^{\epsilon_{i_t -1}} \cdots U(s_1)^{\epsilon_1} ] \times \notag\\
	&	\Tr[U(s_1)^{-\epsilon_{1}} \cdots U(s_{i_{t}-1})^{-\epsilon_{i_t -1}} X_{t}^\ast U(s_{i_{t}+1})^{-\epsilon_{i_t +1}}  \cdots U(s_{m}^{-\epsilon_{m}}) ] ).\notag
\end{align}
We set as for the case $q=1$, $e^{\vec{i}}_0 = e_{0}$ and we then apply the previous algorithm to $e^{\vec{i}}_0$ by considering the corresponding fixed initial word of the algorithm to be:
\begin{align*}
	\mathcal{S}^{\vec{i}}_0 &= ((s_1,\epsilon_1),...,(s_{i_{1}-1},\epsilon_{i_{1}-1}),(d+1,+),(s_{i_{1}+1},\epsilon_{i_{1}+1}),...\\
	&~~~~~...,(s_{i_{t}-1},\epsilon_{i_{t}-1}),(d+t,+),(s_{i_{t}+1},\epsilon_{i_{t}-1}),...,(s_{i_q-1},\epsilon_{i_{q}-1}),(d+q,+),(s_{i_q+1}\epsilon_{i_{q}+1}),...,(s_{m},\epsilon_{m})).
\end{align*}
For all $n \in \mathbb{N}$ we denote by $\mathcal{F}_{n}(\mathcal{S}^{\vec{i}})$ the set of terms $e^{\vec{i}}((\ell_b,j_b,\epsilon_b)_{b\in[n]})$ that terminate at the $n$-th iteration.
\begin{prop}\label{computexprungcancel}
	Let $1\le q \le m$ and $1 \le i_1 < \cdots < i_q \le m$ be fixed. The Proposition \ref{rungcanc} holds when we replace $i$ by $\vec{i}$. Besides we also have:
	\begin{itemize}
		\item $\mathcal{F}_{n}(\mathcal{S}^{\vec{i}}) = \bigcap_{t=1}^q \mathcal{F}_{n}(\mathcal{S}^{{i_t}})$
		\item $\mathbb{E}_{0}^{\vec{i}} (\mathcal{S}) =1$.
	\end{itemize}
\end{prop}
\begin{proof}
	The first is proved by induction. Denoting $\vec{i}_t = (i_1,...,i_t)$ and applying the algorithm to $\mathcal{S}^{\vec{i}_{t}}$, we have that the terms with rung cancellation of matrices $(i_1,...,i_q)$ in the algorithm applied to $\mathcal{S}$ are the terms with rung cancellation of matrix $i_q$ in the algorithm applied to $\mathcal{S}^{\vec{i}_{q-1}}$.\\
	For the second point we start with the case where $q=1$: 
	\begin{align*}
		\mathbb{E}_{0}^{i} (\mathcal{S}) &= \mathbb{E}(\Tr[X U(s_{i-1})^{\epsilon_{i-1}}\cdots U(s_1)^{\epsilon_1} U(s_{m})^{\epsilon_m}\cdots U(s_{i+1})^{\epsilon_{i+1}}]\times \\
		&~~~~~~~~~~~~~~~~~~~~~~~~~~~~ \Tr[X^\ast U(s_{i+1})^{-\epsilon_{i+1}}\cdots U(s_{m})^{-\epsilon_m} U(s_1)^{-\epsilon_1}\cdots U(s_{i-1})^{-\epsilon_{i-1}} ])\\
		&=\frac{1}{N} \Tr[U(s_{i-1})^{\epsilon_{i-1}}\cdots U(s_1)^{\epsilon_1} U(s_{m})^{\epsilon_m}\cdots U(s_{i+1})^{\epsilon_{i+1}} U(s_{i+1})^{-\epsilon_{i+1}}\cdots U(s_m)^{-\epsilon_m}U(s_1)^{-\epsilon_1}\cdots U(s_{i-1})^{-\epsilon_{i-1}}]\\
		& =1
	\end{align*}
	where we applied the Schwinger-Dyson equation between the first and second equality. The general is obtained by induction, noticing that:
	$$\mathbb{E}_{0}^{\vec{i}} (\mathcal{S})= \mathbb{E}_{0}^{i_q}(\mathcal{S}^{\vec{i}_{q-1}})=\mathbb{E}_{0}^{i_q} ((s_1,\epsilon_{1}),...,(s_{i_{t}-1},\epsilon_{i_{t}-1}),(d+t,+),(s_{i_{t}+1},\epsilon_{i_{t}+1}),...,(s_{m},\epsilon_{m})).$$
\end{proof}
We may now give the following lemma and first computation of $\mathbb{E}_{0}(\mathcal{S})$.

\begin{lem}\label{cancelrungcancel}
	Let $\mathcal{S}=((s_1,\epsilon_{1}),...,(s_{m},\epsilon_{m})) \in ([d]\times\{+,-\})^{m}$ be an initial word with $m < \frac{1}{2}(N^{2/3}+1)$. There exists a unique $\mathcal{S}^\prime=((s^{\prime}_1,\epsilon_{1}^{\prime}),...,(s^{\prime}_{m^\prime},\epsilon_{m^{\prime}}))$ minimal writing of $((s_1,\epsilon_1),...,(s_{m},\epsilon_{m})) \in ([d]\times\{+,-\})^{m}$ and we have:
	$$\mathbb{E}_{0}(\mathcal{S}) = 1 + \sum_{n=0}^{\infty} \sum_{e \in \mathcal{R}_n(\mathcal{S}^{\prime})} \bar{e}.$$
	Also for all integer $n$ we have $\mathcal{F}_{n}(\mathcal{S})= \mathcal{F}_{n}(\mathcal{S}^\prime)$ and $\mathcal{R}_{n}(\mathcal{S})= \mathcal{R}_{n}(\mathcal{S}^\prime)$.
\end{lem}
\begin{proof}
	We have:
	$$\mathbb{E}_0(\mathcal{S}) = 	\mathbb{E}_0(\mathcal{S}^\prime)$$
	and therefore we consider directly $\mathcal{S}$ minimal word and we have:
	\begin{align*}
		\mathbb{E}_0(\mathcal{S}) =\sum_{n=0}^\infty \sum_{e \in \mathcal{R}_n(\mathcal{S})} \bar{e} + \underbrace{\sum_{n=0}^\infty \sum_{e \in \bigcup_{i=1}^m \mathcal{F}_n(\mathcal{S}^i)} \bar{e}}_{\mathbb{E}_{0}^{r}}
	\end{align*}
	where we applied Proposition \ref{cvseries} for the first equality.
	We want to show that $\mathbb{E}_{0}^{r}=1$. Indeed:
	\begin{align*}
		\mathbb{E}_0^r &= \sum_{n=0}^\infty \sum_{e \in \bigcup_{i=1}^m \mathcal{F}_n(\mathcal{S}^i)} \bar{e} = \sum_{n=0}^\infty \sum_{q=1}^{m} (-1)^{q-1}\sum_{i_1 <...< i_q}\sum_{e \in \mathcal{F}_n(\mathcal{S}^{\vec{i}})} \bar{e}\\
		&=\sum_{n=0}^\infty \sum_{e \in \mathcal{F}_n(\mathcal{S}^{\vec{i}})} \bar{e} = \sum_{q=1}^{m} (-1)^{q-1}\sum_{i_1 <...< i_q} \mathbb{E}_{0}^{\vec{i}}(\mathcal{S})\\
		&=\sum_{q=1}^{m} (-1)^{q-1} \binom{m}{q}=1,
	\end{align*}
	where in the first line of calculation we applied the inclusion-exclusion principle to the union $\bigcup_{i \in [m]} \mathcal{F}_n(\mathcal{S}^{i})$ and the first point of Proposition \ref{computexprungcancel}. Between the second and third line we used the second point of Proposition \ref{computexprungcancel}.
\end{proof}

All that remains is to make conditions explicit on the original word $((s_1,\epsilon_1),...,(s_{m},\epsilon_{m})) \in ([d]\times\{+,-\})^{m}$ to have terms without rung cancellation and then use these conditions to compute either $\mathbb{E}_{1}$ or $\mathbb{E}_2$ (defined by Equations \eqref{defE_1} and \eqref{defE_2} respectively). To do this, we have to take into account the dependence of the paths on the original word. From now on, we will specify the dependence in $\mathcal{S}$ for all sequences obtained by the algorithm in Section \ref{algo}. For example, for $\mathcal{S}$ a minimal word and $\mathrm{P}$ a pattern of length $n$, we will denote by $e(\mathcal{S},\mathrm{P})$ the path of generation $n$ obtained by applying the algorithm to $\mathcal{S}$ and $\mathrm{P}$, where before we would have denoted by $e(\mathrm{P})$ without confusion. %We therefore set for all $n \ge 1$:
%\begin{equation*}
%	\mathcal{R}_{n} :=  \bigcup_{\mathcal{S} \in [d]^{2m}} \mathcal{R}_n (\mathcal{S}).
%\end{equation*}
%We have the following proposition.
%\begin{prop}
%	There exists $C>0$ or all integer $n \ge 1$ we have:
%	$$|\mathcal{R}_n| \le (Cm)^{4n} d^m.$$
%\end{prop}
%This is a consequence of the following lemma.
\begin{defi}\label{equclasspattern}
	For all integers $m,p\ge 1$ we denote:
	\begin{itemize}
		\item $\mathcal{W}(m)$ the set of minimal words $\mathcal{S}= ((s_1,\epsilon_1),...,(s_{m},\epsilon_{m}))$ of length $m$,
		\item $\mathcal{W}^\prime(m)$ the subset of minimal words of length $m$ such that there exists a non-empty finishing path with no rung cancellation of matrix,
		\item $\mathcal{W}_1(p)$ the set of minimal words $\mathcal{S}= ((s_1,+),(s_{2},-),...,(s_{2p-1},+),(s_{2p},-))$ of length $2p$, and likewise $\mathcal{W}_{1}^\prime$ the subset of $\mathcal{W}_{1}(p)$ of minimal words such that there exists a non-empty finishing path with no rung cancellation.
	\end{itemize} 
	Let $\mathrm{P}=(\ell_b,j_b,\epsilon_b)_{b \in [n]}$ be a pattern of length $n$. We set $\mathcal{W}^\prime(m,\mathrm{P}) \subset \mathcal{W}^{\prime}(m)$ the set of $\mathcal{S} \in \mathcal{W}^\prime(m)$ (\textup{resp}. $\mathcal{W}_{1}^\prime(p,\mathrm{P}) \subset \mathcal{W}^\prime_1(p)$ the set of $\mathcal{S}\in \mathcal{W}^\prime_1(p)$) such that $e(\mathcal{S},\mathrm{P}) \neq \emptyset$ terminates at the $n$-th generation. For $\mathcal{S} \in \mathcal{W}^\prime(m,\mathrm{P})$ we denote by $[\mathcal{S}]_{\mathrm{P}}$ the set of $\mathcal{S}^\prime \in \mathcal{W}^\prime(m,\mathrm{P})$  (\textup{resp}. $\mathcal{S}^\prime \in \mathcal{W}_1^\prime(p,\mathrm{P})$) of minimal words of same length and such that for all $1 \le b \le n$ we have:
	\begin{equation*}
		f_b(\mathcal{S}^\prime,\mathrm{P})(\cdot)= f_b(\mathcal{S},\mathrm{P})(\cdot)~~\text{and} ~~ f_b^{rs}(\mathcal{S}^\prime,\mathrm{P})(\cdot)= f_b^{rs}(\mathcal{S},\mathrm{P})(\cdot).
	\end{equation*}
	where $f_{b}(\mathcal{S},\mathrm{P})(\cdot),f_{b}^{rs}(\mathcal{S},\mathrm{P})(\cdot)$ are the tracking and rescaled tracking functions defined in \eqref{trackingfunction}. This last definition is an equivalence relation on the sets $\mathcal{W}^\prime(m,\mathrm{P})$ and $\mathcal{W}_1^\prime(p,\mathrm{P})$ which we denote by $\sim_{\mathrm{P}}$.
\end{defi}
\begin{lem}\label{termsnorungcancel}
	Let $\mathrm{P} =(\ell_b,j_b,\epsilon_b)_{b\in[n]} \in ([2m] \times [2m] \times \{+,-\})^n$ be a pattern and  $\mathcal{S}\in \mathcal{W}^\prime(m,\mathrm{P})$. Then there exists $f:[m] \rightarrow [m]$ with no fixed point and such that for all $$\mathcal{S}^\prime=((s_1^\prime,\epsilon_{1}^\prime),...,(s_{m}^\prime,\epsilon^{\prime}_m)) \in [\mathcal{S}]_{\mathrm{P}}$$
	we have for all $ 1 \le i \le m$:
	$$s_{f(i)}^\prime = s_{i}^\prime.$$ 
	Likewise for $\mathcal{S} \in \mathcal{W}^\prime_1(p,\mathrm{P})$, there exists $f : [2p] \rightarrow [2p]$ with no fixed point such that for all 
	$$\mathcal{S}^\prime =((s_{1}^\prime,+),(s_2^\prime,-),...,(s_{2p-1}^\prime,+), (s_{2p},-))\in [\mathcal{S}]_{\mathrm{P}}$$
	and for all $1 \le i \le 2p$ we have:
	$$s_{f(i)}^\prime = s_{i}^\prime.$$ 
\end{lem}
\begin{proof}
	The proof is the same in both cases of the previous lemma, therefore we prove only the first case. For all $1 \le b \le n$ we consider $e_b := e(\mathcal{S},\mathrm{P}_{b})$ where $\mathrm{P}_b=(\ell_{b^\prime},j_{b^\prime},\epsilon_{b^\prime})_{b^\prime \in [b]}$. We construct the function $f: [m] \rightarrow [m]$ as follows. First we notice that it is sufficient to show that for all $i\in [m]$ there exists $j\neq i$ such that for all $\mathcal{S}^\prime =((s_{k}^\prime,\epsilon_{k}^\prime)_{k\in [m]}) \sim_{\mathrm{P}} \mathcal{S}$ we have $s_i^\prime=s_j^\prime$. Let now $i \in [m]$ be any index. We first suppose that there exists $1\le b \le n-1$ such that $(\ell_{b}(i),j_{b}(i))\in \{ (1,1), (\ell_{b},j_{b})\}$ or $(\ell_{b}^{\prime}(i),j_{b}^\prime(i))\in \{ (1,1), (\ell_{b+1},j_{b+1})\}$. In other words we suppose that for some iteration, the matrix $i$ is not trivially moved in the path $e(\mathcal{S},\mathrm{P})$ (cf Definition \ref{trivandrung}). We consider the smallest $1\le b \le n-1 $ verifying this condition. If $(\ell_{b}(i),j_{b}(i))=(1,1)$ and since $e_{b+1} \neq \emptyset$, then there exists $ j \in[m]$ such that $(\ell_{b+1}, j_{b+1}) \in \{(\ell_{b}(j), j_{b}(j) ) ,(\ell_{b}^\prime(j), j_{b}^\prime(j) ) \}$ and therefore $s_{i}= s_{j}$. In the case where $(\ell_{b+1}, j_{b+1}) =(\ell_{b}(j), j_{b}(j))$ it means is we have the paths:
	\begin{align*}
		(1,m+1-i) &\overset{1}{\rightarrow} (\ell_{1}(i),j_{1}(i)) \overset{2}{\rightarrow} \cdots \overset{b}{\rightarrow} (\ell_{b}(i),j_{b}(i)) =(1,1) \overset{b+1}{\rightarrow} \cdots \\
		(1,m+1-j) &\overset{1}{\rightarrow} (\ell_{1}(j),j_{1}(j)) \overset{2}{\rightarrow} \cdots \overset{b}{\rightarrow} (\ell_{b}(j),j_{b}(j)) =(\ell_{b+1}, j_{b+1}) \overset{b+1}{\rightarrow} \cdots .
	\end{align*}
	By uniqueness of the paths we have $m+1-i \neq m+1-j$ and therefore $i \neq j$. If now we have $(\ell_{b+1}, j_{b+1}) =(\ell_{b}^\prime(j), j_{b}^{\prime}(j))$, we have the paths:
	\begin{align*}
		(1,m+1-i) &\overset{1}{\rightarrow} (\ell_{1}(i),j_{1}(i)) \overset{2}{\rightarrow} \cdots \overset{b}{\rightarrow} (\ell_{b}(i),j_{b}(i)) =(1,1) \overset{b+1}{\rightarrow} \cdots \\
		(2,j) &\overset{1}{\rightarrow} (\ell_{1}^\prime(j),j_{1}^\prime(j)) \overset{2}{\rightarrow} \cdots \overset{b}{\rightarrow} (\ell_{b}^\prime(j),j_{b}^\prime(j)) =(\ell_{b+1}, j_{b+1}) \overset{b+1}{\rightarrow} \cdots .
	\end{align*}
	Since we assumed $b$ the first time that the matrix $i$ is not trivially moved and that we have no rung cancellation of the matrix $i$, then $i \neq j$. In the cases where $(\ell_{b}(i),j_{b}(i))=(\ell_{b+1},j_{b+1})$ or $(\ell_{b}^{\prime}(i),j_{b}^\prime(i))\in \{ (1,1), (\ell_{b+1},j_{b+1})\}$ we similarly show the existence of $j \neq i$ such that $s_i = s_j$.\\
	Now if we suppose the matrix $i$ is never trivially moved. Still the term $e_n$ being of trivial traces, there exists $n_{i} \le n$ such that:
	\begin{align*}
		f_{n_{i}}(\ell_{n_{i}-1}(i),j_{n_i -1}(i)) &\neq (0,0)\\
		\text{and} ~~~f_{n_{i}}^{rs}(\ell_{n_{i}-1}(i),j_{n_i -1}(i)) &= (0,0).
	\end{align*}
	In other words, the matrix $U(s_{i})^{\epsilon_{i}}$ is cancelled at the $n_{i}$-th iteration. If it is cancelled by a matrix originally positioned in the first trace, then there exists $j\neq i$ such that either we have:
	\begin{align*}
		(1,m+1-i) &\overset{1}{\rightarrow} (\ell_{1}(i),j_{1}(i)) \overset{2}{\rightarrow} \cdots \overset{n_i}{\rightarrow} f_{n_i}(\ell_{n_i-1}(i),j_{n_i-1}(i))  \\
		(1,m+1-j) &\overset{1}{\rightarrow} (\ell_{1}(j),j_{1}(j)) \overset{2}{\rightarrow} \cdots  \overset{n_i}{\rightarrow} f_{n_i}(\ell_{n_i-1}(j),j_{n_i-1}(j)),
	\end{align*}
	such that $f_{n_i}(\ell_{n_i-1}^\prime(j),j_{n_i-1}^\prime(j))=f_{n_i}(\ell_{n_i-1}(i),j_{n_i-1}(i)) \pm1$ and $U(s_{1, m+1-i})^{\epsilon_{1, m+1- i}}U(s_{1, m+1-j})^{\epsilon_ {1,m+1-j}} = U(s_{i})^{\epsilon_i}U(s_{ j})^{\epsilon_ { j}}=\mathrm{Id}$.  This is the case when the matrix originally at position $(1,m+1-i)$ is cancelled by a matrix originally at position $(1,m+1-j)$ in the same trace and with $\epsilon_j = -\epsilon_i$. In particular, we have $s_i=s_j$. Otherwise, and since we assume there is no rung cancellation of matrix $i$, there exists $j\neq i$ such that:
	\begin{align*}
		(1,m+1-i) &\overset{1}{\rightarrow} (\ell_{1}(i),j_{1}(i)) \overset{2}{\rightarrow} \cdots \overset{n_i}{\rightarrow} f_{n_i}(\ell_{n_i-1}(i),j_{n_i-1}(i))  \\
		(2,j) &\overset{1}{\rightarrow} (\ell_{1}^\prime(j),j_{1}^\prime(j)) \overset{2}{\rightarrow} \cdots  \overset{n_i}{\rightarrow} f_{n_i}(\ell_{n_i-1}^\prime(j),j_{n_i-1}^\prime(j)),
	\end{align*}
	such that $f_{n_i}(\ell_{n_i-1}^\prime(j),j_{n_i-1}^\prime(j))=f_{n_i}(\ell_{n_i-1}(i),j_{n_i-1}(i)) \pm1$ and $U(s_{1, m+1-i})^{\epsilon_{1, 2m+1- i}}U(s_{2,j})^{\epsilon_ {2,j}} = U(s_{i})^{\epsilon_i}U(s_{j})^{-\epsilon_ {j}}=\mathrm{Id}$. In particular once again $s_i = s_j$. The fact that all of the above remains true for $(s_{j}^\prime,\epsilon_{j}^\prime)_{j\in [m]} \sim_{\mathrm{P}}\mathcal{S}$ is a consequence of the previous construction of the function $f : [m] \rightarrow [m]$ that depended only on the movements of matrices and therefore the functions $f_{b}$ and $f_{b}^{rs}$ for $1 \le b \le n$.
	
\end{proof}

\begin{rem}
	In the previous proof we used, without mentioning it, the fact that for all $s \neq s^{\prime}$ for all $\epsilon,\epsilon^\prime \in \{+,-\}$ we have $U(s)^\epsilon U(s^\prime)^{\epsilon^\prime} \neq \Id$ and $U(s)^2 \neq \Id$ with probability one. 
\end{rem}
\begin{lem}\label{cardeqclass}
	Let  $\mathrm{P}$ be a pattern of length $n$ and $m,p \ge 1$ be integers. For all $\mathcal{S} \in \mathcal{W}^\prime(m, \mathrm{P})$ we have:
	\begin{equation*}
		|[\mathcal{S}]_{\mathrm{P}}| \le d^{m/2} 
	\end{equation*}
	and:
	\begin{equation*}
		|\mathcal{W}^\prime(m,\mathrm{P})/\sim_{\mathrm{P}}| \le (2m)^{2n}.
	\end{equation*}
	Similarly, for all $\mathcal{S} \in \mathcal{W}^{\prime}_1(p, \mathrm{P})$ we have:
	\begin{equation*}
		|[\mathcal{S}]_{\mathrm{P}}| \le \frac{d}{d-1}(d-1)^{p} 
	\end{equation*}
	and:
	\begin{equation*}
		|\mathcal{W}_1^\prime(p,\mathrm{P})/\sim_{\mathrm{P}}| \le (4p)^{2n}.
	\end{equation*}
\end{lem}
\begin{proof}
	We consider $\mathrm{P}=(\ell_b,j_b,\epsilon_b)_{b\in [n]}$ a pattern and $m,p\ge 1$ integers fixed. We start by bounding the cardinal of $[\mathcal{S}]_{\mathrm{P}}$ in both cases. We consider $\mathcal{S} \in \mathcal{W}^\prime(m,\mathrm{P})$ and, considering Lemma \ref{termsnorungcancel}, there exists $f : [m] \rightarrow [m]$ with no fixed point such that for all $1 \le i \le [m]$ and for all $\mathcal{S}^\prime = (s_1^\prime,\epsilon_{j}^\prime)_{j\in[m]} \in [S]_{\mathrm{P}}$ we have $s_{f(i)}^\prime = s_i^\prime$.  Since there is no fixed point it means that, knowing $f$, there exists a subset $F \subset [m]$ that verifies $|F| \le m/2$ and such that one only needs to know $s_{j}^\prime$ for $j \in F$ to know all the values $s_{j}^\prime$ for $1 \le j \le m$. Therefore we have $|[\mathcal{S}]_\mathrm{P}| \le d^{m/2}$. For $\mathcal{S} \in \mathcal{W}_{1}^\prime(p,\mathrm{P})$ we will construct the subset $F$ as follows. Let $f:[2p] \rightarrow [2p]$ the function given by Lemma \ref{termsnorungcancel}. We consider $F=F_{p}$ the set constructed as follows:
	\begin{align*}
		F_1 &= \{1\}\\
		\text{For all}~ 2\le t \le p,~~ F_{t}&=F_{t-1}\cup \{\inf\{i\in [2p]: i \notin F_{t-1},~~f(i)\notin f(F_{t-1})\}\}.
	\end{align*}
	We set $F= \{1< i_{2}< \cdots< i_{p^\prime}\}$ where  $p^\prime=|F| \le p$. We are left with counting the possible values $s_{i_t} \in [d]$ for all $1 \le t \le 2p$. In this case, by minimality we have that for all $1 \le t \le 2p$, $s_{t+1} \neq s_{t}$. In particular we have necessarily $2 \in F$. Therefore we have $d$ choices for $s_{1}$ and $d-1$ choices for $s_2$. Likewise for all $2 \le t \le p^\prime$ we have at most $d-1$ choices for $s_{i_t}$. Indeed $s_{i_t -1}$ is determined either by choice if $i_t-1 \in F$, either by the function $f$ since there exists $i\in F_{t-1}$ such that $f(i)=i_{t}-1$ and $s_{i_t} \neq s_{i_t -1}$. It gives that $|[\mathcal{S}]_{\mathrm{P}}| \le d(d-1)^{p^\prime-1}$, hence the result.\\
	Now to upper-bound $|\mathcal{W}^\prime(m,\mathrm{P}) / \sim_{\mathrm{P}}|$ we consider for all $1 \le b \le n$:
	$$T_{b}(\mathrm{P})= \{(f_t(\mathcal{S},\mathrm{P}),f_t^{rs}(\mathcal{S},\mathrm{P}))_{t\in [b]};~~ \mathcal{S} \in \mathcal{W}^{\prime}(m)~~ \text{and}~~ e(\mathcal{S},\mathrm{P})\neq \emptyset, e(\mathcal{S},\mathrm{P})\in \mathcal{R}_n\},$$
	where the tracking functions are defined by the algorithm if Section \ref{algo} applied to the word $\mathcal{S}$ in regard of the pattern $\mathrm{P}$. We then have:
	$$|\mathcal{W}^\prime(m,\mathrm{P}) / \sim_{\mathrm{P}}| = |T_n(\mathrm{P})|.$$
	We consider the application:
	\begin{align*}
		\phi_b : T_{b}  &\rightarrow T_{b-1}\times [2m]\times[2m]\\
		F=(F_t)_{t\in [b]}& \mapsto ((F_{t})_{t \in [b-1]},c_1(F),c_2(F))
	\end{align*}
	where we define $c_{1},c_{2}$ as the applications that count the number of cancellations between generation $b-1$ and generation $b$ as follows. Let $F=(F_{t})_{t\in[b]}= (f_t,f^{rs}_t)_{t\in[b]}\in T_b$ and $\mathcal{S} \in \mathcal{W}^{\prime}(m,\mathrm{P})$ such that $F = (f_t(\mathcal{S},\mathrm{P}),f_t^{rs}(\mathcal{S},\mathrm{P}))_{t \in [b]}$. For all generations $1\le b\le n$, we use the notations introduced by Equation \eqref{minngene} in the following description. Then we have:
	\begin{itemize}
		\item If we have $\ell_b =1$, $\epsilon_{b}= +$, then necessarily we have $\epsilon^{(b-1)}_{11} = \epsilon^{(b-1)}_{\ell_b j_b}$. It corresponds in the algorithm to the case where at the $b$-th iteration of Schwinger-Dyson equation we chose a term of the first line \eqref{SD1}. Then necessarily we have $k_{b} = k_{b-1}+1$ and no possible cancellation of matrices in the following generation because of the hypothesis of minimality of words in the algorithm. Indeed we necessarily have  $(s^{(b-1)}_{11},\epsilon_{11}^{(b-1)})\neq(s^{(b-1)}_{1j_{b}-1},-\epsilon_{1j_{b}-1}^{(b-1)})$ since we have $(s^{(b-1)}_{11},\epsilon_{11}^{(b-1)})=(s^{(b-1)}_{1j_{b}},\epsilon_{1j_{b}}^{(b-1)})$ and we suppose the $k_{b-1}$-word of generation $b-1$ minimal. We therefore set $c_1(F)=c_2(F)=0$.
		\item If we have $\ell_b =1$, $\epsilon_{b}= -$ then we have $\epsilon^{(b-1)}_{11} = \epsilon^{(b-1)}_{\ell_b j_b}$ we still have $k_{b} = k_{b-1}+1$. It corresponds to the case where we chose the second line \eqref{SD2}. Also now we have possible cancellations in the traces indexed $1$ and $2$. This number of cancellations is given by $c_{1}(F)=j_b-m^{(b)}_1 = \min\{1 \le j\le j_b-1; f_{b}^{rs}(1,j) =(0,0)\}$ and $c_{2}(F) = m_{1}^{(b-1)}-j_{b}-m_2^{(b)}= \max\{j_{b}+1 \le j \le m_1^{(b-1)}; f_{b}^{rs}(1,j) =(0,0) \} - j_{b}$, respectively the number of cancellations in the first and second traces.
		\item If we have $\ell_b \ge 2$, $\epsilon_{b}=+$  then we have $\epsilon_{11}^{(b-1)}=\epsilon^{(b-1)}_{\ell_b j_b}$ and $k_b = k_{b-1}-1$. It corresponds to the case where we chose the third line \eqref{SD3}. As for the first case, the hypothesis of minimality implies that there are no cancellations possible. Therefore we set $c_1(F) = c_{2} (F) = 0$.
		\item Finally if $\ell_b \ge 2$ and $\epsilon_{b}= -\epsilon_{11}$ we have either $k^\prime = k-1$ or $k^\prime = k-2$. It corresponds to the case where we chose the fourth line \eqref{SD4}. In this case we have cancellations in two different places in a same trace. We set in this case  $c_{1}(F)=\min\{1 \le j \le m_{1}^{(b-1)};~~f^{rs}_{b}(1,j) \neq (0,0)\}$ and $c_{2}(F)=\max\{1 \le j \le m_{1}^{(b-1)};~~f^{rs}_{b}(1,j) \neq (0,0)\}$.
	\end{itemize}
	The function $\phi_b$ defined this way is injective. Therefore we have $|T_{b}| \le (2m)^{2}|T_{b-1}|$ and for all $\mathrm{P} = (\ell_b,j_b,\epsilon_b)_{b\in[n]}$ pattern we have:
	$$|\mathcal{W}^\prime(m,\mathrm{P}) / \sim_{\mathrm{P}}| = |T_{n}| \le (2m)^{2n}.$$
	The same reasoning gives the result for $\mathcal{W}_{1}^{\prime}(p, \mathrm{P})/ \sim_\mathrm{P}$.
\end{proof}
\section{Proof of Theorem \ref{hastingsdetails} and Corollary \ref{hastingscor}}\label{proofHastings} 
In this section, for all integer $m \ge 1$ and all $(s_{j})_{j\in[2m]} \in [d]^{2m}$, we set $\mathcal{S}$ as the word given by:
\begin{equation}\label{symword}
	\mathcal{S}:= ((s_{2m},+),(s_{2m-1},+),...,(s_{m+1},+),(s_{m},-),(s_{m-1},-),...,(s_1,-)).
\end{equation}
If we refer to $(s_{j})_{j\in[2m]} \in [d]^{2m}$ as a word we refer to the word obtained by considering the previous definition. For instance we have: 

$$1 + \mathbb{E}(|\lambda_2|^{2m}) \le \left(\frac{1}{d}\right)^{2m} \sum_{\mathcal{S} \in [d]^{2m}} \mathbb{E}_{0} (\mathcal{S})= \mathbb{E}_1$$
where $\mathbb{E}_1$ is defined by \eqref{defE_1}. For all $1 \le m^\prime\le m$ we denote by $\mathcal{W}_{sym}(m^\prime) \subset\mathcal{W}(2 m^\prime)$ the set of $(s_j)_{j\in[2m^\prime]}$ such that the corresponding word given by \eqref{symword} is minimal. Also we denote by $\mathcal{W}^\prime_{sym}(m^\prime) \subset \mathcal{W}_{sym}(m^\prime)\cap\mathcal{W}^\prime(2m^\prime)$ the set of minimal word $\mathcal{S} \in\mathcal{W}_{sym}(m^\prime)$  such that there exists $n \ge 1$ such that $\mathcal{R}_n(\mathcal{S}) \neq \emptyset$, i.e. there exists a term in the algorithm that terminates after the $n$-th iteration with no rung cancellation and different from the empty set. We consider: 
\begin{equation*}
	m = \lfloor\frac{1}{4\sqrt{2}}N^{1/12}\rfloor,
\end{equation*} 
and we rewrite $\mathbb{E}_1$ as follows:
\begin{align}
	\mathbb{E}_1 &= \left(\frac{1}{d}\right)^{2m} \sum_{m^\prime = 1}^{m} \sum_{\mathcal{S} \in \mathcal{W}_{sym}(m^\prime)} ~~\sum_{\substack{(s^\prime_{j})_{j\in[2m]} \in [d]^{2m}\\ \mathcal{S}^\prime \sim \mathcal{S}}} \mathbb{E}_{0}(\mathcal{S})\notag \\
	&= 1 +\sum_{m^\prime = 1}^{m}d^{m-m^\prime}(m-m^\prime+1) \sum_{\mathcal{S} \in \mathcal{W}^\prime_{sym}(m^\prime)} ~~\sum_{n=0}^{\infty}\sum_{e \in \mathcal{R}_n (\mathcal{S})} \bar{e}(\mathcal{S}).\label{rigoraftercancelrung}
\end{align}
Between the first and the second line we applied Lemma \ref{cancelrungcancel} and the fact that for all $\mathcal{S} \in \mathcal{W}_{sym}(m^\prime)$ we have:  
\begin{equation*}
	|(s_{j}^{\prime})_{j\in [2m]} \in [d]^{2m};~\mathcal{S}^\prime\sim \mathcal{S}\}| = \sum_{k=0}^{m-m^\prime} d^{k} d^{m-m^\prime-k} = d^{m-m^\prime}(m-m^\prime+1).
\end{equation*}
Indeed it comes down to count the number of $(s_{j}^{\prime})_{j\in[2m]} \in[d]^{2m}$ such that we have:
\begin{equation*}
	U(s_{2m}^\prime)\cdots U(s_{m+1}^\prime) U(s^\prime_m)^\ast \cdots U(s_1^\prime)^\ast \sim U(s_{2m^\prime})\cdots U(s_{m^\prime+1}) U(s_{m^\prime})^\ast \cdots U(s_1)^\ast
\end{equation*}
where for $((t_1,\epsilon_1),...,(t_k,\epsilon_k))\in \left([d]\times \{+,-\}\right)^k$ the equivalent class of $U(t_1)^{\epsilon_1}\cdots U(t_k)^{\epsilon_{k}}$ is the set 

\begin{equation}\label{eqclassU(S)}
    \{U(t_{1+o})^{\epsilon_{1+o}}U(t_{2+o})^{\epsilon_{2+o}} \cdots U(t_{k+o})^{\epsilon_{k+o}};~ o \in \NN\}.
\end{equation}
It means that for all such $(s_{j}^\prime)_{j\in[2m]}$ we have:
\begin{align*}
	U(s_1^\prime)^\ast\cdots U(s_{k}^\prime)^{\ast}U(s_{2m-k+1}^\prime) U(s_{2m-k+2}^\prime)\cdots U(s_{2m}^\prime) &=\Id \\
	~~\text{and}~~ U(s_{m+p}^\prime)U(s_{m+p-1}^\prime)\cdots U(s_{m+1}^\prime)U(s_{m}^\prime)^\ast U(s_{m-1})^\ast \cdots U(s_{m-p+1})^\ast &= \Id
\end{align*}
for some $(k,p)$ integers such that $p+k = m-m^\prime$. Therefore for $(k,p)$ fixed we have $d^k d^{p}$ possible $(s_{j}^\prime)_{j\in[2m]}$. We now rewrite $\mathbb{E}_1$:

\begin{align*}
	\mathbb{E}_1 & \le 1 + \left(\frac{1}{d}\right)^{2m}m\sum_{m^\prime = 1}^{m} \sum_{\mathcal{S} \in \mathcal{W}_{sym}^\prime(m^\prime)}  d^{m-m^\prime}\sum_{n=0}^{\infty} ~~\sum_{\mathrm{P}=(\ell_{b},j_{b},\epsilon_b)_{b\in [n]} } \mathds{1}(e(\mathcal{S},\mathrm{P}) \neq \emptyset)N^{2 - 2/3n}
\end{align*} 
where the last summand is over all patterns $\mathrm{P}=(\ell_{b},j_{b},\epsilon_b)_{b\in [n]} \in  ([4m]\times [4m]\times \{+,-\})^n$ and where we applied Lemma \ref{lemcv} to bound $|\bar{e}(\mathcal{S},\mathrm{P})|$ for $e(\mathcal{S},\mathrm{P}) \in \mathcal{R}_n(\mathcal{S})$. 
Exchanging the summand in the previous equation we have:
\begin{align*}
	\mathbb{E}_1 & \le 1 + \left(\frac{1}{d}\right)^{2m}m\sum_{n=0}^{\infty} ~~\sum_{\mathrm{P}=(\ell_b,j_b,\epsilon_b)_{b \in [n]}}\sum_{m^\prime = 1}^{m} d^{m-m^\prime} \sum_{[\mathcal{S}]_{\mathrm{P}} \in \mathcal{W}^\prime(2m^\prime,\mathrm{P})/ \sim_{\mathrm{P}} } |[\mathcal{S}]_{\mathrm{P}}| N^{2 - 2/3n}.
\end{align*}

Applying Lemma \ref{cardeqclass} we have:
\begin{align}
	\mathbb{E}_1 & \le 1 + \left(\frac{1}{d}\right)^{2m}m\sum_{n=0}^{\infty} ~~\sum_{\mathrm{P}=(\ell_b,j_b,\epsilon_b)_{b \in [n]}}\sum_{m^\prime = 1}^{m} d^{m-m^\prime} (4m)^{2n} d^{m^\prime} N^{2 - 2/3n} \notag \\
	& \le 1+ \left(\frac{1}{d}\right)^{2m}m^2 N^2 d^{m} \sum_{n=0}^{\infty} \left(\frac{512 m^4}{N^{1/3}}\right)^n \notag \\
	& \le 1 +  \rho_d^{2m} 2m^2N^2,\label{goal1}
\end{align}
where the last inequality is obtained by replacing $m$ by its value. We consider $N$ large enough so we have $m \ge \frac{1}{8\sqrt{2}}N^{1/12}$. Let now consider $0 <\epsilon < 1$, replacing $m$ by its expression we have that:
\begin{align} 
\left(\frac{\mathbb{E}\left( |\lambda_2(\mathcal{E}) |\right)}{\rho_d(1+ \epsilon)} \right)^{2m} \le 	\frac{\mathbb{E}(s_{2}(\mathcal{E}^m)^{2})}{\rho_d^{2m}(1+\epsilon)^{2m}} &\le \frac{\mathbb{E}_1 -1}{(1+\epsilon)^{2m}\rho_d^{2m}}\notag \\
	& \le \frac{2m^2 N^2}{(1+\epsilon)^{2m}} \le \frac{N^{13/6}}{8} e^{-\frac{1}{4\sqrt{2}}\ln(1+\epsilon)N^{1/12}} \notag
\end{align}

where we used that $|\lambda_2(\mathcal{E})| \le s_{2}(\mathcal{E}^m)^{1/m}$ and Jensen inequality for the first inequality above. One now can conclude directly for the proof of Theorem \ref{hastingsdetails} by ignoring the factor $(1+ \epsilon)$ and implement in \eqref{markov} to finish the proof of Corollary \ref{hastingscor}.

\section{Proof of Proposition \ref{powers}}\label{proofpowers}  
In all the following section, for all integers $m,p \ge 1$ and all $(s_{j}^t)_{j\in[m], t\in [2p]} \in [d]^{2mp}$, we set $\mathcal{S}$ as the word given by:
\begin{equation*} 
	\mathcal{S}:= ((s_{m}^{2p},+),(s_{m-1}^{2p},+),...,(s_{1}^{2p},+),(s_{m}^{2p-1},-),(s_{m-1}^{2p-1},-),...,(s_{j}^{t},\epsilon_t),(s_{j-1}^t,\epsilon_t),...,(s_{1}^1,-))
\end{equation*}
where for all $1 \le t \le 2p$ we set $\epsilon_{t} =+$ if $t$ is even and $\epsilon_{t} = -$ if $t$ is odd. Again if we refer to $(s_{j}^t)_{j\in[m],t\in[2p]} \in [d]^{2mp}$ as a word we refer to the word obtained by considering the previous definition. For instance we have: 

$$1 + \mathbb{E}(\lambda_2(\{{\mathcal{E}^{\ast}}^m\mathcal{E}^m\}^p)) \le \mathbb{E}\left(\sum_{a =1}^{N^2} \lambda_a(\{{\mathcal{E}^{\ast}}^m\mathcal{E}^m\}^p) \right)=  \left(\frac{1}{d}\right)^{2mp} \sum_{ (s^t_j) \in [d]^{2mp}} \mathbb{E}_0(\mathcal{S})= \mathbb{E}_2.$$
We apply Proposition \ref{cvseries} and Lemma \ref{cancelrungcancel} and we re-write $\mathbb{E}_2$:
\begin{align}
	\mathbb{E}_2&= \left(\frac{1}{d}\right)^{2mp} \sum_{m^\prime = 1}^{mp} \sum_{\mathcal{S} \in \mathcal{W}(m^\prime)} \sum_{\substack{\mathcal{S}^\prime \in [d]^{2mp}\\ \mathcal{S}^\prime \sim \mathcal{S}}} \mathbb{E}_{0}(\mathcal{S}) \notag\\
	&= 1 +\sum_{m^\prime = 1}^{mp}\sum_{\mathcal{S} \in \mathcal{W}^\prime(m^\prime)} |\underbrace{\{\mathcal{S}^{\prime} \in [d]^{2mp};~~ \mathcal{S}^{\prime} \sim \mathcal{S}\}}_{\mathcal{A}(m,p,\mathcal{S})}| \sum_{n=0}^{\infty}\sum_{e \in \mathcal{R}_n (\mathcal{S})} \bar{e}(\mathcal{S})\label{ooo}. 
\end{align}
We will now upper bound the cardinal of $\mathcal{A}(m,p,\mathcal{S})$ independently of the choice $\mathcal{S}$ minimal word. We first consider $(m_1,...,m_{2p}) \in \{0,...,m\}^{2p}$ such that $\sum_{t} m_t = 2m^\prime$ and for all $1 \le t \le 2p$, there exists $(\tilde{s}_{j}^t) \in [d]^{m_t}$ such that:
\begin{equation*}
	\mathcal{U}(\mathcal{S})= U(\tilde{s}_{m_1}^{1})\cdots U(\tilde{s}_{1}^1)U(\tilde{s}_{m_2}^{2})^\ast \cdots U(\tilde{s}_{1}^{2})^\ast \cdots U(\tilde{s}_{m_{2p-1}}^{2p-1})U(\tilde{s}^{2p}_{m_{2p}})^\ast \cdots U(\tilde{s}^{2p}_{1})^\ast.
\end{equation*}
where we take as convention that the product over an empty set is $\prod_{\emptyset}= \Id$. For all minimal $\mathcal{S} \in \mathcal{W}(m^{\prime})$ we have at most $(m+1)^{2p}$ ways to chose $(m_t)_{t\in[2p]}$. For such a choice of $(m_t)_t$ we consider a choice of $0\le k_1\le m-m_1$ from which we define the sequence:
\begin{equation*}
	k_{t+1}=m-m_t-k_{t}.
\end{equation*}
Once again for $\mathcal{S}$ minimal and $(m_t)_t$ fixed we have at most $(m+1)$ choices. Finally we consider all the words $\mathcal{S}^{\prime}= (s_{j}^t)_{j\in [m], t\in [2p]}$ such that for all $1 \le t \le 2p$ and for all $1 \le j \le m_t$ we have:
\begin{equation*}
	(s_{k_t+j}^t,\epsilon_{k_t+j}^t) =(\tilde{s}^{t}_j,\epsilon_t).
\end{equation*} 
In other words for all $2 \le t \le 2p$ we have:
\begin{align*}
	  U(s^{t-1}_{k_{t-1}+1})^{\epsilon_t}\cdots U(s^{t-1}_{k_{t-1}+m_{t-1}})^{\epsilon_{t-1}}&=U(\tilde{s}_1^{t-1})^{\epsilon_{t-1}} \cdots U(\tilde{s}_{m_{t-1}}^{t-1})^{\epsilon_{t-1}} \\
	  \text{and}~~U(s^{t-1}_{k_{t-1}+m_{t-1}+1})^{\epsilon_{t-1}}\cdots U(s_{m}^{t-1})^{\epsilon_{t-1}}U(s_1^t)^{\epsilon_t} \cdots U(s_{k_t}^t)^{\epsilon_t}&=\Id.
\end{align*}
In particular we have $\mathcal{U}(\mathcal{S}) \sim \mathcal{U}(\mathcal{S}^\prime)$ in the sense given in the previous section, see Equation \eqref{eqclassU(S)}.
For $\mathcal{S}$ minimal, $(m_{t})_{t\in [2p]}$ and $(k_{t})_{t\in [2p]}$ fixed we can construct at most $\prod_{t=1}^{2p} d^{k_t} = d^{mp-m^\prime}$ corresponding equivalent word. We rewrite $\mathbb{E}_{2}$ as follows 
\begin{align*} 
	\mathbb{E}_2 & \le 1 + \left(\frac{1}{d}\right)^{2mp}\sum_{m^\prime = 1}^{mp} \sum_{\mathcal{S} \in \mathcal{W}^\prime(m^\prime)}  (m+1)^{2p+1}d^{mp-m^\prime}\sum_{n=0}^{\infty} \sum_{\mathrm{P}=(\ell_b,j_b,\epsilon_b)_{b\in [n]}} \mathds{1}(e(\mathcal{S},\mathrm{P}) \neq \emptyset)N^{2 - 2/3n}.
\end{align*}
Exchanging the summand we have:
\begin{align*} 
	\mathbb{E}_2& \le 1 + \left(\frac{1}{d}\right)^{2mp}\sum_{n=0}^{\infty} ~~\sum_{\mathrm{P}=(\ell_b,j_b,\epsilon_b)_{b \in [n]}}\sum_{m^\prime = 1}^{mp} (m+1)^{2p+1}d^{mp-m^\prime} \sum_{[\mathcal{S}]_\mathrm{P} \in \mathcal{W}^\prime(m^\prime,\mathrm{P})/ \sim (\mathrm{P}) } |[\mathcal{S}]_\mathrm{P}|N^{2 - 2/3n}.
\end{align*}
Applying Lemma \ref{cardeqclass} we have that for all pattern $\mathrm{P}=(\ell_{b},j_b,\epsilon_{b})_{b\in [n]}$:
\begin{align*}
	|[\mathcal{S}]_\mathrm{P}|\le d^{m^\prime}\\
	|\mathcal{W}^{\prime}(m^\prime,\mathrm{P})/ \sim (\mathrm{P})| \le (4mp)^{2n}.
\end{align*}
Therefore we have:
\begin{align*}  
	\mathbb{E}_2& \le 1 + \left(\frac{1}{d}\right)^{2mp}\sum_{n=0}^{\infty} ~~\left(16(mp)^2\right)^n\sum_{m^\prime = 1}^{mp} (m+1)^{2p+1}d^{mp-m^\prime} (4mp)^{2n} d^{m^\prime} N^{2 - 2/3n}\\
	&\le 1 + 4(m+1)^{2p+1}p N^2\rho_d^{2mp} \sum_{n=0}^\infty \left(\frac{512(mp)^4}{N^{1/3}} \right)^n.
\end{align*}
We now consider $1\le m \le \frac{1}{2} \lfloor\frac{1}{4\sqrt{2}}N^{1/12}\rfloor$. We set:
$$p = \lfloor \frac{1}{m}\lfloor\frac{1}{4\sqrt{2}}N^{1/12}\rfloor \rfloor \le \frac{1}{m}\frac{1}{4\sqrt{2}}N^{1/12}$$
which implies that $\frac{512(mp)^4}{N^{1/3}} \le 1/2$. Therefore we have: 
\begin{align*}
	\frac{\mathbb{E}(\lambda_{2}\{{\mathcal{E}^{\ast}}^m\mathcal{E}^m\}^p)}{\rho_d^{2mp}(1 + \epsilon)^{2mp}} \le \frac{\mathbb{E}_2- 1}{\rho_d^{2mp}(1 + \epsilon)^{2mp}} \le  \frac{8(m+1)^{2p+1}p N^2}{(1+\epsilon)^{2mp}}.
\end{align*} 
For $N$ large enough, for all sequence $m\le \frac{1}{2} \lfloor\frac{1}{4\sqrt{2}}N^{1/12}\rfloor$ and $p = p(m,N)$ defined above we have:
$$mp \ge \frac{1}{8\sqrt{2}}N^{1/12}.$$
Therefore for all $\epsilon >0$, for all $N$ and all $1\le m\le \frac{1}{2} \lfloor\frac{1}{4\sqrt{2}}N^{1/12}\rfloor$, replacing $p$ by its value we have: 
$$\frac{8(m+1)^{2p+1}p N^2}{(1+\epsilon)^{2mp}} \le \frac{\sqrt{2}(m+1)}{m}N^{25/12}e^{[2\frac{\ln(m+1)}{m} -\ln(1+\epsilon)]\frac{1}{4\sqrt{2}}N^{1/12}}$$
which implemented in Equation \eqref{markov1} concludes the proof.   
%% We first bound the value of a term terminating at the $n$- iteration and the number of term terminating at the $n$-th iteration for all $n \in \mathbb{N}$.
%\begin{proof}[Proof of Hastings Theorem \ref{Hastings2007}]
%	For all $\mathcal{S}=(s_1,...,s_{2m}) \in [d]^{2m}$ we now explicit the dependence of the set $\mathcal{R}_n (\mathcal{S})$ and $\mathcal{F}_n(\mathcal{S})$ in our computations. 	
%\end{proof}
\section{Proof of Theorem \ref{powers1} and Corollary \ref{powers1cor}}\label{proofpowers1} 
To prove the last Theorem we need to look further into the combinatorial computations of the previous section. Indeed we consider Equation \eqref{ooo} for $m=1$, that is:
\begin{align}\label{E3proof}
	\mathbb{E}_3= 1 +\left(\frac{1}{d}\right)^{2p}\sum_{p^\prime = 1}^{p}\sum_{\mathcal{S} \in \mathcal{W}^\prime(p^\prime)} |\underbrace{\{\mathcal{S}^{\prime} \in [d]^{2p};~~ \mathcal{S}^{\prime} \sim \mathcal{S}\}}_{\mathcal{A}(p,\mathcal{S})}| \sum_{n=0}^{\infty}\sum_{e \in \mathcal{R}_n (\mathcal{S})} \bar{e}(\mathcal{S}).
\end{align}
We upper bound $|\mathcal{A}(p,\mathcal{S})|$ independently of the choice of $\mathcal{S}$ minimal. Therefore we may suppose that $\mathcal{S}$ starts with $\epsilon_1=+$, i.e. we can write $((\tilde{s}_1,+),(\tilde{s}_2,-),...,(\tilde{s}_{2p^\prime},-))= \mathcal{S}$. First notice that we have:
$$|\mathcal{A}(p,\mathcal{S})| \le 2p|\{\mathcal{S}=(s_1,...,s_{2p});~~ U(s_1)U(s_2)^{\ast}\cdots U(s_{2p-1})U(s_{2p})^\ast= \mathcal{U}(\mathcal{S})\}| =:2p |\mathcal{A}^\prime(p,\mathcal{S})|.$$
Indeed if for some word $(s_{t})_{t \in [2p]}$ we have $U(s_1)U(s_2)^{\ast}\cdots U(s_{2p-1})U(s_{2p})^\ast= \mathcal{U}(\mathcal{S}) $ then for all $1 \le t \le 2p$ we have $(s_{t+1},s_{t+2},...,s_{2p+t}) \sim \mathcal{S}$ where we recall that by convention $s_{2p+1} = s_1$. Now to upper bound $|\mathcal{A}^\prime(p,\mathcal{S})|$ we consider the following random walk:
\begin{align*}
	(s_{t})_{t\in\mathbb{N}} &\sim \mathrm{Unif}([d])\\
	X_0 &:= \Id \\
	\forall t\in \NN,~\forall s \in [d]~~~~\mathbb{P}(X_{2t+1}=X_{2t}U(s)) = \frac{1}{d},~&~\mathbb{P}(X_{2t+2}=X_{2t+1}U(s)^\ast) = \frac{1}{d}.
\end{align*}
The real random walk of interest is in fact the random walk obtained by simply considering the distance to $\Id$ for each $t \in \NN$, i.e. $\ell_t := \ell(X_t)= \ell(U(s_1)U(s_{2})^\ast \cdots U(s_{t})^{\epsilon_t})$. We denote by $\mathrm{N}(p,p^\prime,d)$ the number of possible words of length $2p^{\prime}$ (in the free group $\mathbb{F}_d$) obtained after $2p$ steps of the random walk above, that is $\mathrm{N}(p,p^\prime,d)  := |\{(s_1,...,s_{2p});~~\ell(U(s_1)U(s_2)^\ast\cdots U(s_{2p})^\ast)=2p^\prime\}|$. By induction on $p\ge 1$, we have that for all $1 \le p^\prime \le p$:
$$\mathbb{P}(\ell_{2p}=2p^\prime) = \frac{\mathrm{N}(p,p^{\prime},d)}{d^{2p}} \le \frac{2^{2p}}{d^{2p}}(d-1)^{p+p^\prime}.$$
Also for all integer $p^\prime$ we denote by $\mathrm{Red}(2p^\prime)=d(d-1)^{2p^\prime-1}$ the number of reduced words of length $2p^\prime$ that the previous random walk can have, i.e. the number of $(s_{1}^\prime,...,s_{2p^\prime}^\prime)\in [d]^{2p^\prime}$ such that for all $1 \le t \le 2p^\prime$ we have $s_{t}^\prime \neq s_{t+1}^\prime$. We then have:
$$\mathrm{N}(p,p^\prime,d)= |\mathcal{A}^\prime(p,\mathcal{S})|\mathrm{Red}(2p^\prime)$$ 
and therefore:
$$|\mathcal{A}(p,\mathcal{S})| \le 2p 2^{2p}\frac{d-1}{d}(d-1)^{p- p^\prime}.$$
Implementing the previous upper bounds in \eqref{E3proof} we have:
\begin{align*} 
	\mathbb{E}_3 & \le 1 + \left(\frac{1}{d}\right)^{2p}\sum_{p^\prime = 1}^{p} \sum_{\mathcal{S} \in \mathcal{W}_1^\prime(p^\prime)}  2p 2^{2p}\frac{d-1}{d}(d-1)^{p- p^\prime}\sum_{n=0}^{\infty} \sum_{\mathrm{P}=(\ell_b,j_b,\epsilon_b)_{b\in [n]}} \mathds{1}(e(\mathcal{S},\mathrm{P}) \neq \emptyset)N^{2 - 2/3n}.
\end{align*}
Exchanging the summand we have:
\begin{align*} 
	\mathbb{E}_3& \le 1 + \left(\frac{1}{d}\right)^{2p}\sum_{n=0}^{\infty} ~~\sum_{\mathrm{P}=(\ell_b,j_b,\epsilon_b)_{b \in [n]}}\sum_{p^\prime = 1}^{p} 2p 2^{2p}\frac{d-1}{d}(d-1)^{p- p^\prime} \sum_{[\mathcal{S}]_\mathrm{P} \in \mathcal{W}_1^\prime(p^\prime,\mathrm{P})/ \sim (\mathrm{P}) } |[\mathcal{S}]_\mathrm{P}|N^{2 - 2/3n}.  
\end{align*}
Finally we apply Lemma \ref{cardeqclass}:
\begin{align*}  
	\mathbb{E}_3& \le 1 + \left(\frac{1}{d}\right)^{2p}\sum_{n=0}^{\infty} ~~\left((4p)^2 2\right)^n\sum_{m^\prime = 1}^{mp}  2p 2^{2p}\frac{d-1}{d}(d-1)^{p- p^\prime}(4p)^{2n} (d-1)^{p^\prime} \frac{d}{d-1} N^{2 - 2/3n}\\
	&\le 1 + 2p N^2\sigma_d^{2p} \sum_{n=0}^\infty \left(\frac{512p^4}{N^{1/3}} \right)^n.
\end{align*}
We consider:
\begin{equation*}
	p:= \lfloor \frac{1}{4\sqrt{2}}N^{1/12}\rfloor
\end{equation*}
and obtain:
\begin{align*}  
	\mathbb{E}_3 \le 1 + 4p N^2\sigma_d^{2p}
\end{align*}
and
\begin{equation}\label{eq:lastone}
	\frac{\mathbb{E}(\lambda_{2}({\mathcal{E}^\ast}\mathcal{E})^p)}{\sigma_d^{2p}(1 + \epsilon)^{2p}} \le \frac{\mathbb{E}_3- 1}{\sigma_d^{2p}(1 + \epsilon)^{2p}} \le  \frac{4p N^2}{(1+\epsilon)^{2p}}.
\end{equation} 
%%%%FIN

\bibliographystyle{alpha}
\bibliography{biblio.bib}
\end{document}